\documentclass[reqno, 12pt, reqno]{amsart}

\usepackage{amsfonts}
\usepackage{amsthm}
\usepackage{amsmath}
\usepackage{amscd}
\usepackage[latin2]{inputenc}
\usepackage{t1enc}
\usepackage[mathscr]{eucal}
\usepackage{indentfirst}
\usepackage{graphicx}
\usepackage{graphics}
\usepackage{pict2e}
\usepackage{epic}
\usepackage{hyperref}

\hypersetup{colorlinks=false}

\usepackage{enumitem}
\usepackage[margin=1.5in]{geometry}
\usepackage{lipsum}
\usepackage{adjustbox}

\RequirePackage{mathrsfs} \let\mathcal\mathscr
\numberwithin{equation}{section}

  \DeclareFontFamily{U}{wncy}{}
    \DeclareFontShape{U}{wncy}{m}{n}{<->wncyr10}{}
    \DeclareSymbolFont{mcy}{U}{wncy}{m}{n}
    \DeclareMathSymbol{\Sh}{\mathord}{mcy}{"58}

\RequirePackage{tikz-cd}
\RequirePackage{amssymb}
\usetikzlibrary{calc}
\usetikzlibrary{decorations.pathmorphing}

\tikzset{curve/.style={settings={#1},to path={(\tikztostart)
    .. controls ($(\tikztostart)!\pv{pos}!(\tikztotarget)!\pv{height}!270:(\tikztotarget)$)
    and ($(\tikztostart)!1-\pv{pos}!(\tikztotarget)!\pv{height}!270:(\tikztotarget)$)
    .. (\tikztotarget)\tikztonodes}},
    settings/.code={\tikzset{quiver/.cd,#1}
        \def\pv##1{\pgfkeysvalueof{/tikz/quiver/##1}}},
    quiver/.cd,pos/.initial=0.35,height/.initial=0}

\tikzset{tail reversed/.code={\pgfsetarrowsstart{tikzcd to}}}
\tikzset{2tail/.code={\pgfsetarrowsstart{Implies[reversed]}}}
\tikzset{2tail reversed/.code={\pgfsetarrowsstart{Implies}}}
\tikzset{no body/.style={/tikz/dash pattern=on 0 off 1mm}}


\usepackage{mathtools}
\DeclarePairedDelimiter{\ceil}{\lceil}{\rceil}
\usepackage{comment}
\DeclareMathAlphabet\mathbfcal{OMS}{cmsy}{b}{n}
\newcommand{\mcbf}{\mathbfcal}

\usepackage[capitalise, noabbrev]{cleveref}
\usepackage{colonequals}

\newif\ifmoditem
\newcommand{\setupmodenumerate}{%
  \global\moditemfalse
  \let\origmakelabel\makelabel
  \def\moditem##1{\global\moditemtrue\def\mesymbol{##1}\item}%
  \def\makelabel##1{%
    \origmakelabel{##1\ifmoditem\rlap{\mesymbol}\fi\enspace}%
    \global\moditemfalse}%
}

\newcommand{\CT}{\textrm{Colliot-Th{\'e}l{\`e}ne}}
\newcommand{\inj}{\hookrightarrow}
\newcommand{\surj}{\twoheadrightarrow}
\newcommand{\birat}{\dashrightarrow}

\newcommand{\Vol}{\operatorname{Vol}}
\newcommand{\Mod}[1]{\ (\mathrm{mod}\ #1)}
\newcommand{\op}[1]{\operatorname{#1}}
\newcommand{\mc}[1]{\mathcal{#1}}
\newcommand{\mf}[1]{\mathfrak{#1}}

\renewcommand{\phi}{\varphi}
\renewcommand{\rho}{\varrho}

\renewcommand{\l}{\left}
\renewcommand{\r}{\right}
\renewcommand{\ss}{\substack}
\newcommand{\f}{\frac}

\renewcommand{\P}{\mathbb{P}}
\newcommand{\Proj}{\P}

\newcommand{\A}{\mathbb{A}}
\newcommand{\F}{\mathbb{F}}
\newcommand{\Z}{\mathbb{Z}}

\newcommand{\N}{\mathbb{N}}
\newcommand{\Q}{\mathbb{Q}}
\newcommand{\R}{\mathbb{R}}

\newcommand{\G}{\mathbb{G}}

\newcommand{\OO}{\mathcal{O}}
\newcommand{\OOK}{\mathcal{O}_K}
\newcommand{\OOk}{\mathcal{O}_k}

\newcommand{\p}{\mathfrak{p}}

\newcommand{\del}{\partial}
\newcommand{\isom}{\cong}

\newcommand{\Gal}{{\rm Gal}}
\newcommand{\Emb}{\op{Emb}}

\renewcommand{\leq}{\leqslant}
\renewcommand{\geq}{\geqslant}
\renewcommand{\bar}{\overline}

\newcommand{\bN}{\mathbf{N}}

\newcommand{\bx}{\mathbf{x}}

\newcommand{\bz}{\mathbf{z}}

\newcommand{\ba}{\mathbf{a}}

\newcommand{\eps}{\epsilon}

\newcommand{\e}{\mathrm{e}}

\DeclareMathOperator{\Pic}{Pic}

\DeclareMathOperator{\Br}{Br}

\newcommand{\ord}{\op{ord}}
\newcommand{\Hom}{\op{Hom}}

\renewcommand{\hat}{\widehat}

\theoremstyle{theorem}
\newtheorem{theorem}{Theorem}[section]

\newtheorem{conjecture}[theorem]{Conjecture}
\newtheorem{corollary}[theorem]{Corollary}

\newtheorem{proposition}[theorem]{Proposition}

\newtheorem{lemma}[theorem]{Lemma}

\theoremstyle{definition}

\newtheorem{remark}[theorem]{Remark}

\newtheorem{example}[theorem]{Example}

\newtheorem{notation}[theorem]{Notation}
\newtheorem{assumption}[theorem]{Assumption}

\numberwithin{equation}{section}

\newcommand{\nid}{\noindent}

\newcommand{\ra}{\rightarrow}

\newcommand{\tensor}{\otimes}

\newcommand{\bs}{\backslash}

\newcommand{\PP}{\Proj}
\usepackage{tikz}
\usetikzlibrary{3d,calc}
\usetikzlibrary{positioning}

\begin{document}
\title{Polynomials represented by norm forms via the beta sieve}
\maketitle{}

\begin{center}
\author{Alec Shute}\\
\vspace{7pt}
\address{School of Mathematics\\ University of Bristol}\\Bristol, BS8 1UG\\
\email{alec.shute@bristol.ac.uk}
\end{center}

\begin{abstract}
A central question in Arithmetic geometry is to determine for which polynomials $f \in \Z[t]$ and which number fields $K$ the Hasse principle holds for the affine equation $f(t) = \bN_{K/\Q}(\bx) \neq 0$. Whilst extensively studied in the literature, current results are largely limited to polynomials and number fields of low degree. In this paper, we establish the Hasse principle for a wide family of polynomials and number fields, including polynomials that are products of arbitrarily many linear, quadratic or cubic factors. The proof generalises an argument of Irving \cite{irving2017cubic}, which makes use of the beta sieve of Rosser and Iwaniec. As a further application of our sieve results, we prove new cases of a conjecture of Harpaz and Wittenberg on locally split values of polynomials over number fields, and discuss consequences for rational points in fibrations. 
\end{abstract}

\tableofcontents

\section{Introduction}

Let $K$ be a number field of degree $n$, and let $f \in \Z[t]$ be a polynomial. A central problem in Arithmetic geometry is to determine under what conditions $f$ can take values equal to a norm of an element of $K$. In order to address this question, we take an integral basis $\omega_1, \ldots, \omega_n$ for $K$, viewed as a vector space over $\Q$, and define the \textit{norm form} as $\bN(\bx) = N_{K/\Q}(\omega_1x_1+ \cdots +\omega_nx_n)$, where $N_{K/\Q}(\cdot)$ is the field norm. We then seek to understand when the equation 
\begin{equation}\label{object of study}
f(t) = \bN(\bx) \neq 0
\end{equation}
has a solution with $(t,x_1, \ldots, x_n) \in \Q^{n+1}$. A necessary condition for solubility over $\Q$ is that $(\ref{object of study})$ must have solutions in $\R^{n+1}$ and in $\Q_p^{n+1}$ for every prime $p$. We say that the \textit{Hasse principle} holds if this condition alone is sufficient to guarantee existence of a solution to (\ref{object of study}) over $\Q$. 

Local to global questions for (\ref{object of study}) have received much attention over the years. The first case to consider is when $f$ is a non-zero constant polynomial. Here, the Hasse principle for (\ref{object of study}) is known as the \textit{Hasse norm principle}. More precisely, we say that the \textit{Hasse norm principle} holds for the extension $K/\Q$ if $\Q^{\times}\cap N_{K/\Q}(I_K) = N_{K/\Q}(K^{\times})$, where $I_K = \A_K^{\times}$ is the group of ideles of $K$. The Hasse norm principle has been extensively studied, beginning with the work of Hasse himself, who established that it holds for cyclic extensions $K/\Q$ (a result known as the \textit{Hasse norm theorem}), but does not hold for certain biquadratic extensions, such as $K = \Q(\sqrt{13},\sqrt{-3})$. The Hasse norm principle is also known to hold if the degree of $K$ is prime (Bartels \cite{bartels1981arithmetik}), or if the normal closure of $K$ has Galois group $S_n$ (Kunyavski\u{\i} and Voskresenski\u{\i} \cite{HNPforSnextensions}) or $A_n$ for $n \neq 4$ (Macedo \cite{macedo2020hasse}).

When $[K:\Q]=2$ and $f$ is irreducible of degree $3$ or $4$, (\ref{object of study}) defines a Ch\^{a}telet surface. There are now many known counterexamples to the Hasse principle for Ch\^{a}telet surfaces, including one by Iskovskikh  \cite{iskovskikh1971counterexample}, which we discuss in more detail in Example \ref{iskovskikh example}. However,  Colliot-Th{\'e}l{\`e}ne, Sansuc and Swinnerton-Dyer \cite{colliot1987intersections} prove that the Brauer--Manin obstruction accounts for all failures of the Hasse principle. A similar result holds when $f$ is an irreducible polynomial of degree at most $3$ and $[K:\Q]=3$, as proved by Colliot-Th{\'e}l{\`e}ne and Salberger \cite{colliot1989arithmetic}. Both of these results make use of fibration and descent methods. 

In the case when $f$ is an irreducible quadratic and $K$ is a quartic extension containing a root of $f$, the Hasse principle and weak approximation are known to hold for (\ref{object of study}) thanks to the work of Browning and Heath-Brown \cite{browning2012quadratic}. This result was generalised by Derenthal, Smeets and Wei \cite[Theorem 2]{derenthal2015universal}, to prove that the Brauer--Manin obstruction is the only obstruction to the Hasse principle and weak approximation for irreducible quadratics $f$ and arbitrary number fields $K$. Moreover, in \cite[Theorem 4]{derenthal2015universal} they give an explicit description of the Brauer groups that can be obtained in this family. 

Results when $f$ is not irreducible have so far been limited to products of linear polynomials. Suppose that $f$ takes the form
\begin{equation}\label{norms as products of lin polys}
    f(t) = c\prod_{i=1}^r (t-e_i)^{m_i},
\end{equation}
for some $c\in \Q^*, e_1, \ldots, e_r \in \Q$ and $m_1, \ldots, m_r \in \N$. 
When $r=1$, the Brauer--Manin obstruction is the only obstruction to the Hasse principle and weak approximation for any smooth projective model of (\ref{object of study}). This is a special case of the work of Colliot-Th{\'e}l{\`e}ne and Sansuc \cite{colliot1977r} on principal homogeneous spaces under algebraic tori. Heath-Brown and Skorobogatov \cite{heathskoro2002polynorms} treat the case $r=2$ by combining descent methods with the Hardy--Littlewood circle method, under the assumption that $\gcd(m_1,m_2,\deg K)=1$. This assumption was later removed by Colliot-Th{\'e}l{\`e}ne, Harari and Skorobogatov \cite{colliotharariskoro2003valeurs}. Thanks to the work of Browning and Matthiesen \cite{browning2017norm}, it is now settled that for any number field $K$ and and polynomial $f$ of the form (\ref{norms as products of lin polys}) (for arbitrary $r \geq 1$), the Brauer--Manin obstruction is the only obstruction to the Hasse principle and weak approximation for any smooth projective model of (\ref{object of study}). Their result is inspired by additive combinatorics results of Green, Tao and Ziegler \cite{green2010linear}, \cite{green2012inverse}, combined with ``vertical" torsors introduced by Schindler and Skorobogatov \cite{schindlerskoro2014norms}. 

In general, it has been conjectured by Colliot-Th{\'e}l{\`e}ne \cite{colliot2003conj} that all failures of the Hasse principle for any smooth projective model of (\ref{object of study}) are explained by the Brauer--Manin obstruction. Assuming Schinzel's hypothesis, this holds true for $f$ an arbitrary polynomial and $K/\Q$ a cyclic extension, as demonstrated by work of Colliot-Th{\'e}l{\`e}ne and Swinnerton-Dyer on pencils of Severi--Brauer varieties \cite{CLTpencils}. Recently, Skorobogatov and Sofos also establish unconditionally that when $K/\Q$ is cyclic, (\ref{object of study}) satisfies the Hasse principle for a positive proportion of polynomials $f$ of degree $d$, when their coefficients are ordered by height \cite[Theorem 1.3]{skorobogatov2020schinzel}.  

In \cite{irving2017cubic}, Irving introduces an entirely new approach to studying the Hasse principle for (\ref{object of study}), which rests on sieve methods. Irving's main result \cite[Theorem 1.1]{irving2017cubic} states that if $f\in \Z[t]$ is an irreducible cubic, then the Hasse principle holds for (\ref{object of study}) under the following assumptions: 
\begin{enumerate}
    \item $K$ satisfies the Hasse norm principle.
    \item There exists a prime $q\geq 7$, and a finite set of primes $S$, such that for all $p\notin S$, either $p\equiv 1 \Mod{q}$ or the inertia degrees of $p$ in $K/\Q$ are coprime. 
    \item The number field generated by $f$ is not contained in the cyclotomic field $\Q(\zeta_q)$.
\end{enumerate}
An example provided by Irving in \cite{irving2017cubic} is the number field $\Q(\alpha)$, where $\alpha$ is a root of $x^q-2$ and $q \geq 7$ is prime. We shall comment on this further in Example \ref{irvings example}.

In this paper, we generalize Irving's arguments to establish the Hasse principle for a wide new family of polynomials and number fields. Our results cover for the first time polynomials of arbitrarily large degree which are not a product of linear factors. In fact, under suitable assumptions on $K$, we can deal with polynomials that are products of arbitrarily many linear, quadratic and cubic factors. 

Throughout this paper, we let $\hat{K}$ denote the Galois closure of $K$, and we let $G = \Gal(\hat{K}/\Q)$, viewed as a permutation group on $n$ letters. We define
\begin{equation}\label{definition of T(G)}
T(G) = \frac{1}{\#G}\#\{\sigma \in G: \textrm{ the cycle lengths of }\sigma\textrm{ are not coprime}\}.
\end{equation}
We now state our main results. 
\begin{theorem}\label{main HP result}
Let $K$ be a number field satisfying the Hasse norm principle. Let $f\in \Z[t]$ be a polynomial, all of whose irreducible factors have degree at most 2. Let $k$ denote the number of distinct irreducible factors of $f$, and let $j$ denote the number of distinct irreducible quadratic factors of $f$ which generate a quadratic field contained in $\hat{K}$. Suppose that $T(G) \leq \frac{0.39000}{k+j+1}$. Then the Hasse principle holds for (\ref{object of study}).
\end{theorem}

In practice, the constant $0.39000$ can be improved slightly, particularly when the majority of the factors of $f$ are linear, although it will always be less than $1/2$. The precise optimal constant is obtained by finding the maximal value of $\kappa$ such that (\ref{maximising kappa}) holds. 

When $G=S_n$, we shall see in Lemma \ref{neat lemma about symmetric group} that $T(S_n) \ra 0$ as $n \ra \infty$, so in the setting of Theorem \ref{main HP result} we can establish the Hasse principle provided $n$ is sufficiently large in terms of the degree of $f$. We illustrate this by treating the case when $f$ is a product of two irreducible quadratics. 

\begin{corollary}\label{a tale of two quadratics}
Let $f \in \Z[t]$ be a product of two quadratic polynomials such that the number field $L$ generated by $f$ is a biquadratic extension of $\Q$. Let $K$ be a number field of degree $n$ with $G=S_n$. Suppose that $L\cap \hat{K} = \Q$. Then the Hasse principle holds for (\ref{object of study}), provided that 
$$n\not\in\{2,3,\ldots, 10, 12,14,15,16,18,20,22,24,26,28,30,36,42,48\}.$$
\end{corollary}
We remark that without the assumption $L\cap \hat{K} = \Q$, a similar result to Corollary \ref{a tale of two quadratics} still holds, although a larger list of degrees $n$ would need to be excluded. For example, if $L\cap \hat{K}$ is quadratic, then the Hasse principle holds for (\ref{object of study}) for all primes $n\geq 11$ and all integers $n>90$, whilst if $L\cap \hat{K} = L$, then the Hasse prinicple holds for all primes $n\geq 13$ and all integers $n>150$.

We cannot hope to deal with all small values of $n$ in Corollary \ref{a tale of two quadratics}. For example, the work of Iskovskikh \cite{iskovskikh1971counterexample} shows that the Hasse principle can fail when $n=2$ (see Example \ref{iskovskikh example}). However, as we shall discuss in Appendix \ref{section:Brauer group calculation}, in the case $n\geq 3$, there is no Brauer--Manin obstruction to the Hasse principle, and so according to the conjecture of Colliot-Th{\'e}l{\`e}ne mentioned above we should expect the Hasse principle to hold.

Our second main result allows $f$ to contain irreducible cubic factors, but requires more restrictive assumptions on the number field $K$, more similar to Irving's setup in \cite{irving2017cubic}. 

\begin{theorem}\label{HP with cubic factors}
Let $f\in \Z[t]$ be a polynomial, all of whose irreducible factors have degree at most $3$. Then the Hasse principle holds for (\ref{object of study}) under the following assumptions for $K$.
\begin{enumerate}
    \item $K$ satisfies the Hasse norm principle.
    \item The set $\mc{P}$ of primes $p$ for which the inertia degrees of $p$ in $K/\Q$ are not coprime satisfies Assumption \ref{assumtion on special P}.
\end{enumerate}
\end{theorem}
As an example, Assumption \ref{assumtion on special P} is satisfied if there exists a prime $q$ such that $\frac{\deg f +1}{q-1}\leq 0.32380$, and such that for all but finitely many primes $p\not\equiv 1 \Mod{q}$, the inertia degrees of $p$ in $K/\Q$ are coprime. The constant $0.32380$ appearing in Assumption \ref{assumtion on special P} could likely be improved with more work, and in specific examples, the required bounds could be computed more precisely using (\ref{annoying numeric optimisation}). We remark that we have also dropped the assumption made in \cite{irving2017cubic} that the number field generated by $f$ is not contained in $\Q(\zeta_q)$. This assumption is not essential to Irving's argument, but allows for the treatment of smaller values of $q$. Reinserting this assumption and optimising (\ref{annoying numeric optimisation}), we could recover Irving's result from our work.

We prove Theorems \ref{main HP result} and \ref{HP with cubic factors} by applying the beta sieve of Rosser and Iwaniec \cite[Theorem 11.13]{friedlander2010opera}. The main sieve results we obtain are stated in Theorems \ref{main sieve result} and \ref{main sieve result for cubic factors}. These results in fact prove the existence of a solution to (\ref{object of study}) with $t$ arbitrarily close to a given adelic solution. Consequently, the above results could be extended to prove weak approximation for (\ref{object of study}), provided that weak approximation holds for the norm one torus $\bN(\bx)=1$. For example, the work of Kunyavski\u{\i} and Voskresenski\u{\i} \cite{HNPforSnextensions} and Macedo \cite{macedo2020hasse} demonstrates that weak approximation for the norm one torus holds when $G=S_n$ or $G=A_n$ and $n \neq 4$, and so weak approximation holds in the setting of Corollary \ref{a tale of two quadratics}.

\medskip

In Section \ref{section: the Harpaz-Wittenberg conjecture}, we find a second application of Theorem \ref{main sieve result} to a conjecture of Harpaz and Wittenberg \cite[Conjecture 9.1]{harpaz2016fibration}, which we restate in Conjecture \ref{Harpaz--Wittenberg 9.1} and henceforth refer to as the \textit{Harpaz--Wittenberg conjecture}. The conjecture concerns a collection of number field extensions $L_i/k_i/k$, $i\in \{1, \ldots, n\}$, where $k_i \isom k[t]/(P_i(t))$ for monic irreducible polynomials $P_i \in k[t]$. Roughly speaking, the conjecture predicts, under certain hypotheses, the existence of an element $t_0\in k$ such that $P_1(t_0), \ldots,P_n(t_0)$ are locally split, i.e., each place in $k_i$ dividing $P_i(t_0)$, has a degree $1$ place of $L_i$ above it. 

A major motivation for the conjecture is the development of the theory of rational points in fibrations. Given a fibration $\pi:X\ra \PP^1_{k}$, a natural question is to what extent we can deduce arithmetic information about $X$ from arithmetic information about the fibres of $\pi$. A famous conjecture of Colliot-Th{\'e}l{\`e}ne \cite[p.174]{colliot2003conj} predicts that for any smooth, proper, geometrically irreducible, rationally connected variety $X$ over a number field $k$, the rational points $X(k)$ are dense in the Brauer--Manin set $X(\A_k)^{\op{Br}}$. (In other words, the Brauer--Manin obstruction is the only obstruction to weak approximation.) Applied to this conjecture, the above question becomes whether density of $X(k)$ in $X(\A_k)^{\op{Br}}$ follows from density of $X_c(k)$ in $X_c(\A_k)^{\op{Br}}$ for a general fibre $X_c \colonequals \pi^{-1}(c)$ of $\pi$ (see \cite[Question 1.2]{harpaz2021rational}). Applications of the Harpaz--Wittenberg conjecture to this question are studied in \cite{harpaz2016fibration} and \cite{harpaz2021rational}. 

Harpaz and Wittenberg \cite[Section 9.2]{harpaz2016fibration} demonstrate that their conjecture follows from the homogeneous version of Schinzel's hypothesis (commonly reffered to as $(\textrm{HH}_1)$) in the case of abelian extensions $L_i/k_i$, or more generally, almost abelian extensions (see \cite[Definition 9.4]{harpaz2016fibration}). Examples of almost abelian extensions include cubic extensions, and extensions of the form $k(c^{1/p})/k$ for $c\in k$ and $p$ prime. The work of Heath-Brown and Moroz  \cite{heathmoroz2002primes} establishes $(\textrm{HH}_1)$ for primes represented by binary cubic forms, from which the Harpaz--Wittenberg conjecture can be deduced in the case $k=\Q, n=1$ and $\deg P_1 = 3$. Using a geometric reformulation of \cite[Conjecture 9.1]{harpaz2016fibration}, the authors establish their conjecture in low degree cases, namely when $\sum_{i=1}^n [k_i:k] \leq 2$ or $\sum_{i=1}^n [k_i:k]=3$ and $[L_i:k_i]=2$ for all $i$. 

The Harpaz--Wittenberg conjecture is related to the study of polynomials represented by norm forms. As a consequence of the work of Matthiesen \cite{matthiesen2018square} on norms as products of linear polynomials, the Harpaz--Wittenberg conjecture holds in the case $k_1=\cdots = k_n = k =\Q$ \cite[Theorem 9.14]{harpaz2016fibration}. Similarly, we can deduce from \cite[Theorem 1.1]{irving2017cubic} that the Harpaz--Wittenberg conjecture holds in the case $n=2, k=\Q, k_1=K,k_2=\Q, L_1 = K(2^{1/q})$ and $L_2 = \Q(2^{1/q})$, where $q\geq 7$ is a prime such that $K\not\subseteq \Q(\zeta_q)$ \cite[Theorem 9.15]{harpaz2016fibration}. 

Besides the work of Matthiesen \cite{matthiesen2018square} for $k_1 = \cdots = k_n = k = \Q$, the aforementioned results apply only to the case $n\leq 2$. In Section \ref{section: the Harpaz-Wittenberg conjecture}, we prove the following theorem, which establishes the Harpaz--Wittenberg conjecture in a new family of extensions $k_1/\Q, \ldots, k_n/\Q$, where $n$ may be arbitrarily large, and each extension $k_i/\Q$ may have degree up to $3$. 

\begin{theorem}\label{HW result}
Let $n\geq 1$. Let $k=\Q$, and for $i\in \{1, \ldots, n\}$, let $k_i, M_i$ be linearly disjoint number fields over $\Q$. Let $L_i = M_ik_i$ be the compositum of $k_i$ and $M_i$. Define
\begin{equation}\label{fixed point density}T_i = \f{1}{\#\Gal(\hat{M_i}/\Q)}\#\{\sigma \in \Gal(\hat{M}_i/\Q): \sigma \textrm{ has no fixed point}\}.
\end{equation}
Let $d = \sum_{i=1}^n [k_i:\Q]$. Then the Harpaz--Wittenberg conjecture holds in the following cases.
\begin{enumerate}
    \item $[k_i:\Q]\leq 2$ for all $i\in \{1, \ldots, n\}$ and $\sum_{i=1}^n T_i \leq 0.39000/d$.
    \item $[k_i:\Q]\leq 3$ for all $i\in \{1, \ldots, n\}$, and there exist primes $q_i$ satisfying $\sum_{i=1}^n 1/(q_i-1)\leq 0.32380/d$, and integers $t_i$ coprime to $q_i$, such that for all but finitely many primes $p \not\equiv t_i\Mod{q_i}$, there is a place of degree $1$ in $M_i$ above $p$. 
\end{enumerate}
\end{theorem} 
\begin{corollary}\label{HW corollary}
Let $q_1, \ldots, q_n$ be distinct primes, and let $r_1, \ldots, r_n \in \N$ be such that $g_i(x) = x^{q_i}-r_i$ is irreducible for all $i$. Let $M_i = \Q[x]/(g_i)$ and let $k_i,L_i$ and $d$ be as in Theorem \ref{HW result}. Suppose that one of the following holds:
\begin{enumerate}
\item $[k_i:\Q] \leq 2$ for all $i\in\{1, \ldots, n\}$ and $\sum_{i=1}^n 1/q_i \leq 0.39000/d$,
\item $[k_i:\Q] \leq 3$ for all $i \in\{1,\ldots,n\}$ and $\sum_{i=1}^n 1/(q_i-1) \leq 0.32380/d$.
\end{enumerate}
Then the Harpaz--Wittenberg conjecture holds for $k=\Q$ and for such choices of $k_i$ and $L_i$.
\end{corollary}
We remark that when applied to the setting of \cite[Theorem 9.15]{harpaz2016fibration}, the above result requires a stronger bound on $q$. However, with a more careful optimisation of (\ref{cubic saving q-1}), it should be possible to recover \cite[Theorem 9.15]{harpaz2016fibration} from our approach. 

By combining Theorem \ref{HW result} with \cite[Theorem 9.17]{harpaz2016fibration} (with the choice $B=0, M'' = \emptyset$ and $M' = \PP^1_k\bs U$), we obtain the following result about rational points in fibrations. 

\begin{theorem}\label{application to fibrations}
Let $X$ be a smooth, proper, geometrically irreducible variety over $\Q$. Let $\pi:X \ra \PP^1_{\Q}$ be a dominant morphism whose general fibre is rationally connected. Let $k_1, \ldots, k_n$ denote the residue fields of the closed points of $\PP^1_{\Q}$ above which $\pi$ has nonsplit fibres, and let $L_i/k_i$ be finite extensions which split these nonsplit fibres. Assume that
\begin{enumerate}
    \item The smooth fibres of $\pi$ satisfy the Hasse principle and weak approximation.
    \item The hypotheses of Theorem \ref{HW result} hold.
\end{enumerate}
Then $X(\Q)$ is dense in $X(\A_{\Q})^{\Br(X)}$. 
\end{theorem}
It would be interesting to investigate whether Condition $(1)$ in Theorem \ref{application to fibrations} could be relaxed to the assumption that the smooth fibres $X_c(\Q)$ are dense in $X_c(\A_{\Q})^{\Br(X_c)}$, as in the setting of \cite[Question 1.2]{harpaz2021rational} discussed above. This would require an extension of Theorem \ref{HW result} to cover a stronger version of the Harpaz--Wittenberg conjecture, involving strong approximation of an auxiliary variety $W$ off a finite set of places \cite[Proposition 6.1]{harpaz2021rational}. Strong approximation of $W$ was studied by Browning and Schindler \cite{browningschindler2019strongW} for example, who established \cite[Question 1.2]{harpaz2021rational} in the case when the rank of $\pi$ is at most $3$, and at least one of its nonsplit fibres lies above a rational of $\PP^1_{\Q}$. 

\hfill\break
\nid \textbf{Acknowledgements.} 
The author is grateful to Jean-Louis $\CT$ for providing the statements and proofs in Appendix \ref{section:Brauer group calculation}, and to Julian Lyczak and Alexei Skorobogatov for helpful discussions on Brauer groups. The author would like to thank Olivier Wittenberg for many useful comments on the Harpaz--Wittenberg conjecture, including suggesting the statement and proof of Proposition \ref{Wittenberg's email}, and to Florian Wilsch for helpful discussions and feedback on an earlier version of Section \ref{section:HW hypoth}. The author would like to thank Tim Browning for valuable feedback and guidance during the development of this work. Finally, the author is grateful to the anonymous referee for many useful comments on an earlier version of this paper. The author is supported by the University of Bristol and the Heilbronn Institute for Mathematical Research.

\section{Main sieve results for binary forms}\label{section: main sieve results for binary forms}

Let $f \in \Z[x,y]$ be a non-constant binary form with nonzero discriminant. We write $f(x,y)$ as a product of distinct irreducible factors 
\begin{equation}\label{factorisation of f}
f(x,y) = \prod_{i=0}^m f_i(x,y)\prod_{i=m+1}^k f_i(x,y), 
\end{equation}
where $f_i(x,y)$ are linear forms for $1\leq i\leq m$, and forms of degree $k_i\geq 2$ for $m+1\leq i \leq k$. If $y\mid f(x,y)$, then we define $f_0(x,y)=y$, and otherwise we let $f_0(x,y)=1$. Hence we have $y\nmid f_i(x,y)$ for all $i\geq 1$.

For $i\in\{0,\ldots, k\}$, we define
\begin{align}
\nu_i(p) &= \#\{[x:y]\in \PP^1(\F_p): f_i(x,y) \equiv 0 \Mod{p}\},\label{definition of nui}\\
\nu(p) &=\#\{[x:y]\in \PP^1(\F_p): f(x,y) \equiv 0 \Mod{p}\}.\label{definition of nu}
\end{align}
Let $\mc{P}$ be a set of primes, and let $\mc{P}_{\leq x} = \{p\in \mc{P}:p\leq x\}$. We denote by $\pi(x)$ the number of primes less than $x$. For all $i\in \{0,\ldots, k\}$, we need to assume $\mc{P}$ has the following properties, for some $\alpha, \theta_i >0$ and any $A\geq 1$:

\begin{align}
    \sum_{p \in \mc{P}_{\leq x}} 1 &= \alpha \pi(x)\l(1+O_A\l((\log x)^{-A}\r)\r),\label{density 1}\\
    \sum_{p \in \mc{P}_{\leq x}} \nu_i(p) &= \alpha\theta_i \pi(x)\l(1+O_A\l((\log x)^{-A}\r)\r).\label{density 3}
\end{align}
The reason we require explicit error terms in (\ref{density 1}) and (\ref{density 3}) is so that the \textit{sieve dimensions}, introduced in Section \ref{subsection: sieve dims}, exist. We note that for $i=0$, we have $\theta_0=1$ if $f_0(x,y) = y$, and $\theta_0=0$ if $f_0(x,y)=1$. Additionally, from (\ref{density 3}), it follows that
\begin{equation}
\sum_{p \in \mc{P}_{\leq x}} \nu(p) = \alpha\theta \pi(x)\l(1+O_A\l((\log x)^{-A}\r)\r)\label{density 2},
\end{equation}
where $\theta = \theta_0+ \cdots + \theta_k$. 

Let $\mc{B}\subseteq [-1,1]^2$ denote the compact region enclosed by a piecewise continuous simple closed curve of finite length. The perimeter of $\mc{B}$ will always be assumed to be bounded by some absolute constant $C$. In the applications in Sections \ref{section: Application to HP} and \ref{section: the Harpaz-Wittenberg conjecture}, we shall make the choice 
\begin{equation}\label{choice of region}\mc{B} = \l\{(x,y) \in (0,1]^2: \l|\f{x}{y}-r\r|<\xi\r\},\end{equation}
for a fixed real number $r>0$ and a small parameter $\xi >0$, and so we may choose $C=4$, for example. We also define $\mc{B}N=\{(Nx,Ny): (x,y) \in \mc{B}\}.$ 

Let $\Delta$ be an integer and let $a_0,b_0 \in \Z/\Delta\Z$. We now state the main sieve results which will be used in the proof of Theorem \ref{main HP result} and Theorem \ref{HP with cubic factors}. They concern the sifting function
\begin{equation}\label{sieves:main sieve problem}
    S(\mc{P}, \mc{B},N) = \#\l\{(a,b)\in \mc{B}N\cap \Z^2:
    \begin{aligned}
    &a\equiv a_0, b\equiv b_0 \Mod{\Delta}\\
    &p\mid f(a,b) \implies p\notin \mc{P}
    \end{aligned}
    \r\}.
\end{equation}

\begin{theorem}\label{main sieve result}
Let $f(x,y)$ be a binary form consisting of distinct irreducible factors, all of degree at most $2$. Then there exists a finite set of primes $S_0$, depending on $f$, such that the following holds:

Let $S$ be a finite set of primes containing $S_0$. Let $\Delta$ be an integer with only prime factors in $S$, and let $a_0, b_0 \in \Z/\Delta\Z$. Let $\mc{P}$ be a set of primes disjoint from $S$ and satisfying (\ref{density 1}) and (\ref{density 3}) for some $\alpha, \theta_i>0$. Assume that $\alpha\theta\leq 0.39000$. Then $S(\mc{P},\mc{B},N)>0$ for $N$ sufficiently large. 
\end{theorem}

We also have a similar sieve result when $f$ may contain irreducible factors of degree up to $3$, but with a less general sifting set $\mc{P}$, satisfying the following assumption.  

\begin{assumption}\label{assumtion on special P}
   There exists a positive integer $n$, a finite set of primes $S$, primes $q_1, \ldots, q_n$, and integers $t_1, \ldots, t_n$ with $q_j \nmid t_j$ for all $j \in \{1, \ldots, n\}$, such that 
   \begin{equation}\label{assumption equation thingy}
   \mc{P}\bs S \subseteq \bigcup_{j=1}^n \{p \equiv t_j \Mod{q_j}\}
   \end{equation}
and
 \[
   \deg f \sum_{j=1}^n \f{1}{q_j-1} \leq 0.32380.
   \]
\end{assumption}
 
\begin{theorem}\label{main sieve result for cubic factors}
Let $f(x,y)$ be a binary form consisting of distinct irreducible factors, all of degree at most $3$. Then there exists a finite set of primes $S_0$, depending on $f$, such that the following holds:

Let $S$ be a finite set of primes containing $S_0$. Let $\Delta$ be an integer with only prime factors in $S$, and let $a_0, b_0 \in \Z/\Delta\Z$. Let $\mc{P}$ be a set of primes satisfying Assumption \ref{assumtion on special P}. Then $S(\mc{P}, \mc{B},N)>0$ for $N$ sufficiently large.
\end{theorem}

For brevity, in the remainder of the paper we shall denote the condition $a\equiv a_0 \Mod{\Delta}, b\equiv b_0 \Mod{\Delta}$ by $C(a,b)$.

\section{Levels of distribution}\label{section: levels of distribution}

Crucial to the success of the beta sieve in proving Theorems \ref{main sieve result} and \ref{main sieve result for cubic factors} is a good level of distribution result, which provides an approximation of the quantities 
$$\#\{(a,b) \in \mc{B}N\cap \Z^2: p\mid f_i(a,b), d \mid f(a,b)\}$$
by multiplicative functions, at least on average over $p$ and $d$. (Here, and throughout this section, we keep the notation from Section \ref{section: main sieve results for binary forms}.) In this section, we provide such an estimate, following similar arguments developed by Daniel \cite[Lemma 3.3]{daniel1999divisor}. We slightly generalise the setup as follows:

Let $g_1,g_2$ be binary forms with nonzero discriminants. Throughout this section, we fix $S, \Delta$, and $C(a,b)$, and assume that $S$ contains all primes dividing the discriminants of $g_1$ and $g_2$. We allow all implied constants to depend only the degrees of $g_1$ and $g_2$ and a small positive constant $\eps$, which for convenience we allow to take different values at different points in the argument. 

Let $\mc{R}$ be the compact region of $\R^2$ enclosed by a piecewise continuous simple closed curve of finite length. We denote by $\op{Vol}(\mc{R})$ and $P(\mc{R})$ the volume and perimeter of $\mc{R}$ respectively. Let
\begin{align}
    A(d_1,d_2) &= \#\{(a,b)\in \mc{R}\cap \Z^2:C(a,b), d_1\mid g_1(a,b), d_2\mid g_2(a,b)\},\label{lod counting problem}\\
    \rho(d_1,d_2)&=\#\{(a,b) \Mod{d_1d_2}: d_1\mid g_1(a,b), d_2\mid g_2(a,b)\}\label{definition of multiplicative function rho}.
\end{align}
We define
\begin{equation}\label{def: r(d_1,d_2)}
    r(d_1,d_2) = A(d_1,d_2)-\frac{\rho(d_1,d_2)\op{Vol}(\mc{R})}{d^2\Delta^2}.
\end{equation}

In what follows, we let $d=d_1d_2$, and we assume that $\gcd(d_1,d_2)=\gcd(d,\Delta)=1$. The main aim of this section is to prove the following proposition.
\begin{proposition}\label{lod result}
Suppose that $g_1$ does not contain any linear factors. Then for any $D_1,D_2 >0$ and any $\eps>0$, we have 
\begin{equation*}
\sum_{\substack{d_1\leq D_1, d_2\leq D_2\\
\gcd(d_1,d_2)=\gcd(d,\Delta)=1}}\sup_{P(\mc{R})\leq N}\left|r(d_1, d_2)\right|\ll (D_1D_2)^{\eps}(D_1D_2+N(D_1D_2)^{1/2}+ND_2).
\end{equation*}
\end{proposition}
As a corollary, we obtain the following level of distribution result.
\begin{corollary}\label{lod}
Suppose that $g_1$ does not contain any linear factors. Let $\mc{B} \subseteq [-1,1]^2$ be as in Section \ref{section: main sieve results for binary forms}, and let $\mc{R} = \mc{B}N$. Then for any $\eps >0$, there exists $\delta>0$ such that for any $D_1,D_2>0$ with $D_2 \ll N^{1-\eps}$ and $D_1D_2 \ll N^{2-\eps}$ , we have 
\begin{equation}\label{fff}
\sum_{\substack{d_1\leq D_1, d_2\leq D_2\\ \gcd(d_1,d_2)=\gcd(d,\Delta)=1}}\left|A(d_1,d_2)-\frac{N^2\rho(d_1,d_2)\op{Vol}(\mc{B})}{d^2\Delta^2}\right|\\
\ll N^{2-\delta}.
\end{equation}
\end{corollary}

Proposition \ref{lod result} and Corollary \ref{lod} are generalisations of Irving's results from \cite[Section 3]{irving2017cubic}, which can be recovered by taking $g_1(x,y) = f(x,y)$ to be the cubic form Irving considered, and $g_2(x,y)=yf(x,y)$. The method of proof is inspired by the pioneering work of Daniel on the divisor-sum problem for binary forms, which requires a similar level of distribution result (see \cite[Lemma 3.3]{daniel1999divisor}). Daniel's argument is more delicate, keeping track of powers of $\log N$ in place of factors of $N^{\eps}$, and Corollary \ref{lod} could be similarly refined, but this yields no advantage for our applications. 

Before proceeding with the proof of Proposition \ref{lod result}, we recall the following standard lattice point counting result.
\begin{lemma}\label{standard lattice point counting result}
    Let $\Lambda \subseteq \R^2$ be a full-rank lattice, and let $\mc{R}\subseteq \R^2$ be as defined before Proposition \ref{lod result}. Then
    \[
    \#(\mc{R}\cap \Lambda) = \f{\Vol(\mc{R})}{\det \Lambda} + O\l(\f{P(\mc{R})}{\lambda_1} + 1\r),
    \]
    where $\lambda_1$ is the length of a shortest nonzero vector in $\Lambda$. 
\end{lemma}

\begin{proof}
    Let $\mc{F}$ be a fundamental domain of $\Lambda$. The translates $v + \mc{F}$ for $v \in \Lambda$ tile $\R^2$. Define sets 
    \[
    S^- = \{v \in \Lambda: (v + \mc{F}) \subseteq \mc{R}\}, \qquad S^+ = \{v \in \Lambda: (v + \mc{F})\cap \mc{R} \neq \emptyset\}.
    \]
    Then 
    \[
    \f{\Vol(S^-+ \Lambda)}{\det(\Lambda)} = \#S^- \leq \#(\mc{R}\cap \Lambda) \leq \#S^+ = \f{\Vol(S^++ \Lambda)}{\det(\Lambda)}.
    \]
    Moreover, $S^- + \Lambda \subseteq \mc{R} \subseteq S^+ + \Lambda$, so $\Vol(S^- + \Lambda) \leq \Vol(\mc{R}) \leq \Vol( S^+ + \Lambda)$. Therefore, 
    \[
    \l|\#(\mc{R}\cap \Lambda) - \f{\Vol(\mc{R})}{\det(\Lambda)}\r| \leq \#S^+ - \#S^-.
    \]
    However, $S^+ - S^- = \{v \in \Lambda: (v + \mc{F}) \cap \del \mc{R} \neq \emptyset\},$
    where $\del\mc{R}$ denotes the boundary of $\mc{R}$. Each segment of $\del\mc{R}$ of length $\lambda_1$ can intersect at most four translates of $\mc{F}$. Therefore, $S^+ - S^- \ll P(\mc{R})/\lambda_1 + 1$, as required.
\end{proof}

We now commence with the proof of Proposition \ref{lod result}. We introduce the quantities $R^*(d_1,d_2), \rho^*(d_1,d_2)$ and $r^*(d_1,d_2)$ which are defined similarly to $A(d_1,d_2), \rho(d_1,d_2)$ and $r(d_1,d_2)$ but with the added condition $\gcd(a,b,d)=1$. 
We write $(a_1,b_1) \sim (a_2,b_2)$ if there exists an integer $\lambda$ such that $(a_1,b_1)\equiv \lambda(a_2,b_2)\Mod{d}$. This forms an equivalence relation on points $(a,b)\in \Z^2$ with $\gcd(a,b,d)=1$. Moreover, the properties $g_1(a,b) \equiv 0 \Mod{d_1}$ and  $g_2(a,b)\equiv 0 \Mod{d_2}$ are preserved under this equivalence. We may therefore define 
$$\mc{U}(d_1,d_2) = \left.\left\{a,b \Mod d:
\begin{tabular}{l l}
     &  $ \gcd(a,b,d)=1$\\ 
     & $d_1\mid g_1(a,b), d_2\mid g_2(a,b)$
\end{tabular}\right\}\middle/\sim.\right.$$

For $\mc{C}\in \mc{U}(d_1,d_2)$, we define
$$ \Lambda(\mc{C}) = \{y \in \Z^2: y \equiv \lambda(a,b) \Mod{d}\textrm{ for some }(a,b)\in \mc{C} \textrm{ and some } \lambda \in \Z\}.$$
It is easy to check that $\Lambda(\mc{C})$ is a lattice in $\Z^2$, and its set of primitive points is $\mc{C}$. For $e \in \Z$, we define 
$$ \Lambda(\mc{C},e) = \{(a,b) \in \Lambda(\mc{C}): e\mid \gcd(a,b)\}.$$
By M\"{o}bius inversion, we have
$$R^*(d_1,d_2) = \sum_{\mc{C} \in \mc{U}(d_1,d_2)}\sum_{e\mid d}\mu(e)\#\{(a,b) \in \mc{R}\cap \Lambda(\mc{C},e): C(a,b)\}.$$
Since $\gcd(d,\Delta)=1$, the set $\{(a,b) \in \Lambda(\mc{C},e):C(a,b)\}$ is a coset of the lattice $\Lambda(\mc{C},e\Delta)$, which has determinant $de\Delta^2$. Therefore, by Lemma \ref{standard lattice point counting result},  
\begin{equation}\label{perimeter and volume}
R^*(d_1,d_2) = \sum_{\mc{C} \in \mc{U}(d_1,d_2)}\sum_{e\mid d}\mu(e) \left(\frac{\op{Vol}(\mc{R})}{de\Delta^2}+O\left(1+\frac{P(\mc{R})}{\lambda_1(\mc{C})}\right)\right),
\end{equation}
where $\lambda_1(\mc{C})$ denotes the length of the shortest nonzero vector in $\Lambda(\mc{C})$. 

Each equivalence class $\mc{C} \in \mc{U}(d_1,d_2)$ consists of $\phi(d)$ elements, and so 
$$\sum_{\mc{C} \in \mc{U}(d_1,d_2)}\sum_{e\mid d}\frac{\mu(e)}{e} = \sum_{\mc{C} \in \mc{U}(d_1,d_2)}\frac{\phi(d)}{d} = \frac{\rho^*(d_1,d_2)}{d}.$$
Moreover, we have $\#\mc{U}(d_1,d_2)\ll d^{\eps}$, as we now explain. We observe that $\#\mc{U}(d_1,d_2) =\rho^*(d_1,d_2)/\phi(d)$, and $\rho^*$ is multiplicative by the Chinese remainder theorem. For primes $p\notin S$, we may apply Hensel's lemma to show that $\rho^*(p^e,1),\rho^*(1,p^e) = O(p^e)$ for any integer $e \geq 1$. Therefore, by the trivial bound for the divisor function \cite[Section 18.1]{HardyWrightintro}, we conclude that  
\begin{equation}
\label{not many equiv classes}\#\mc{U}(d_1,d_2) = \frac{\rho^*(d_1,d_2)}{\phi(d)}\ll \frac{d^{1+\eps}}{\phi(d)}\ll d^{\eps}.
\end{equation}
Applying (\ref{perimeter and volume}), and (\ref{not many equiv classes}), we obtain 
\begin{equation}\label{R1x}
\begin{split}
&\sum_{\substack{d_1\leq D_1, d_2\leq D_2\\ \gcd(d_1,d_2)=\gcd(d,\Delta)=1}}\sup_{P(\mc{R})\leq N}\left|r^*(d_1,d_2)\right|\\
&\ll_{\eps}(D_1D_2)^{\eps}\left(D_1D_2+N\sum_{\substack{d_1\leq D_1, d_2\leq D_2\\ \gcd(d_1,d_2)= \gcd(d,\Delta)=1}}\sum_{\mc{C} \in \mc{U}(d_1,d_2)}\lambda_1(\mc{C})^{-1}\right).
\end{split}
\end{equation}

Let $v_1(\mc{C})$ denote a shortest nonzero vector of $\Lambda(\mc{C})$, and let $\|\cdot\|$ be the usual Euclidean norm. Then $\|v_1(\mc{C})\|^2 \ll |\det \Lambda(\mc{C})|=d \leq D_1D_2$. Therefore
\begin{equation}\label{MAB}
\sum_{\substack{d_1\leq D_1, d_2\leq D_2\\ \gcd(d_1,d_2)=\gcd(d,\Delta)=1}}\sum_{\mc{C} \in \mc{U}(d_1,d_2)}\lambda_1(\mc{C})^{-1} \ll \sum_{0<a^2+b^2\ll D_1D_2}\frac{M(a,b)}{\sqrt{a^2+b^2}},
\end{equation}
where
$$M(a,b) = \#\l\{\begin{tabular}{l l}
     &  $d_1\leq D_1, d_2 \leq D_2, \mc{C} \in \mc{U}(d_1,d_2): $\\
     &  $\gcd(d_1,d_2)= \gcd(d,\Delta)=1,v_1(\mc{C})=(a,b)$
\end{tabular} \r\}.$$
For any $d_1,d_2$ enumerated by $M(a,b)$, we have $d_1\mid g_1(a,b)$ and $d_2\mid g_2(a,b)$, so
$$M(a,b) \leq \#\{d_1\leq D_1, d_2\leq D_2: d_1\mid g_1(a,b), d_2\mid g_2(a,b)\}.$$
Since $g_1$ contains no linear factors, we know that $g_1(a,b)\neq 0$ whenever $(a,b) \neq (0,0)$. Suppose in addition that $g_2(a,b) \neq 0$. Then by the trivial bound for the divisor function we have $M(a,b) \ll (D_1D_2)^{\eps}$. We deduce that 
\begin{align*}
\sum_{\substack{0<a^2+b^2\ll D_1D_2\\g_2(a,b)\neq 0}}\frac{M(a,b)}{\sqrt{a^2+b^2}}&\ll (D_1D_2)^{\eps}\sum_{0<a^2+b^2\ll D_1D_2}\frac{1}{\sqrt{a^2+b^2}}\\
&\ll (D_1D_2)^{1/2 + \eps}.
\end{align*}

Now suppose that $g_2(a,b)=0$. Then as above, we have $O(D_1^{\eps})$ choices for $d_1$, but now $D_2$ choices for $d_2$, so that $M(a,b) \ll D_1^{\eps}D_2$. We obtain 
\begin{align*}
\sum_{\substack{0<a^2+b^2\ll D_1D_2\\g_2(a,b)= 0}}\frac{M(a,b)}{\sqrt{a^2+b^2}}&\ll D_1^{\eps}D_2\sum_{\substack{0<a^2+b^2\ll D_1D_2\\ g_2(a,b)=0}}\frac{1}{\sqrt{a^2+b^2}}.
\end{align*}
For a fixed $b\neq 0$, $g_2(a,b)$ is a nonzero polynomial in $a$, and so has $O(1)$ roots. Therefore 
\begin{align*}
    \sum_{\substack{0<a^2+b^2\ll D_1D_2\\ g_2(a,b)=0}}\frac{1}{\sqrt{a^2+b^2}}&= \sum_{
    \substack{0<a^2+b^2\ll D_1D_2\\ b \neq 0\\g_2(a,b)= 0}}\frac{1}{\sqrt{a^2+b^2}}+\sum_{\ss{0<a^2 \ll D_1D_2\\ g_2(a,0)=0}}\frac{1}{a}\\
    &\ll \sum_{b \ll \sqrt{D_1D_2}}\frac{1}{b}+\sum_{a \ll \sqrt{D_1D_2}}\frac{1}{a}\\
    &\ll (D_1D_2)^{\eps}.
\end{align*}

To summarize, we have established the following generalisation of \cite[Lemma 3.2]{irving2017cubic}.

\begin{lemma}\label{lemma 3.2*}
Suppose that $g_1$ does not contain any linear factors. Then for any $D_1,D_2>0$ and any $\eps >0$, we have
\begin{equation*}
\sum_{\substack{d_1\leq D_1, d_2\leq D_2\\ \gcd(d_1,d_2)= \gcd(d,\Delta)=1}}\sup_{P(\mc{R})\leq N}\left|r^*(d_1,d_2)\right|\ll (D_1D_2)^{\eps}(D_1D_2+N(D_1D_2)^{1/2}+ND_2).
\end{equation*}
\end{lemma}

Now we remove the restriction $\gcd(a,b,d)=1$. Below we write $A(d_1,d_2)=A(\mc{R},d_1,d_2; C(a,b))$ in order to make the dependence on $\mc{R}$ and $C(a,b)$ clear. Let $k_1=\deg g_1$ and $k_2=\deg g_2$. We work with multiplicative functions $\psi_k$ for $k=k_1$ and $k=k_2$, which map prime powers $p^r$ to $p^{\ceil{r/k}}$. We follow the same argument as Irving, but with $\psi_{k_1},\psi_{k_2}$ in place of $\psi_3,\psi_4$. The motivation for this definition of $\psi_k$ comes from the fact that for any integers $d,e,k \geq 1$ with $e\mid \psi_k(d)$, and for any prime $p$, we have
\begin{equation}\label{motivation for psi}
    p\mid \f{\psi_k(d)}{e} \iff p \mid \f{d}{\gcd(d,e^k)}.
\end{equation}

Since $\gcd(d_1,d_2)=1$, we have 
\begin{equation}\label{sum over e}
A(\mc{R},d_1,d_2;C(a,b)) = \sum_{\ss{e_1\mid \psi_{k_1}(d_1)\\e_2\mid \psi_{k_2}(d_2)}}N(d_1,d_2,e_1,e_2),
\end{equation}
where 
\begin{equation}\label{nde}
N(d_1,d_2,e_1,e_2) = \#\l\{(a,b) \in \mc{R}\cap \Z^2: 
\begin{tabular}{l l}
     & $C(a,b), d_1\mid g_1(a,b), d_2\mid g_2(a,b),$\\
     &  $\gcd(a,b,\psi_{k_1}(d_1)\psi_{k_2}(d_2)) = e_1e_2$
\end{tabular}
\r\}.
\end{equation}
We make a change of variables $a' = a/e_1e_2, b' = b/e_1e_2$ in (\ref{nde}). Let $\bar{e_1e_2}$ denote the multiplicative inverse of $e_1e_2$ modulo $\Delta$, which exists due to the assumption $\gcd(d_1d_2, \Delta)=1$. The congruence condition $C(a,b)$ is equivalent to the congruence condition $a' \equiv \bar{e_1e_2}a_0 \Mod{\Delta}$ and $b' \equiv \bar{e_1e_2}b_0\Mod{\Delta}$, which we denote by $C'(a',b')$. We have
\begin{align*}
d_1\mid g_1(a,b) &\iff d_1 \mid (e_1e_2)^{k_1} g_1(a',b') \\
&\iff d_1 \mid e_1^{k_1}g_1(a',b') \\
&\iff \f{d_1}{\gcd(d_1,e_1^{k_1})}\mid g_1(a',b'),
\end{align*}
and similarly for $d_2 \mid g_2(a,b)$. For convenience, we define 
$$ f_1 = \f{d_1}{\gcd(d_1,e_1^{k_1})}, \quad f_2 =\f{d_2}{\gcd(d_2,e_2^{k_2})}.$$
changing notation from $a',b'$ back to $a,b$, we deduce that $N(d_1,d_2,e_1,e_2)$ can be rewritten as
\begin{align}
&\#\l\{(a,b) \in \mc{R}/(e_1e_2)\cap \Z^2: 
\begin{tabular}{l l}
     & $C'(a,b), f_1\mid g_1(a,b), f_2\mid g_2(a,b),$\nonumber\\
     &  $\gcd(a,b,\psi_{k_1}(d_1)\psi_{k_2}(d_2)/e_1e_2) = 1$
\end{tabular}
\r\}\\
&=\#\l\{(a,b) \in \mc{R}/(e_1e_2)\cap \Z^2: 
\begin{tabular}{l l}
     & $C'(a,b), f_1\mid g_1(a,b), f_2\mid g_2(a,b),$\\
     &  $\gcd(a,b,f_1f_2)= 1$
\end{tabular}
\r\}.\nonumber\\
&=R^*\l(\mc{R}/(e_1e_2), f_1,f_2;C'(a,b)\r)\label{R stuff}.
\end{align}
The above arguments, but with the congruence conditions removed, and with the specific choice $\mc{R} = [0,d_1d_2]^2$ also demonstrate that 
\begin{align}
    \rho(d_1,d_2) &=\sum_{\ss{e_1\mid \psi_{k_1}(d_1)\\e_2\mid \psi_{k_2}(d_2)}}\#\l\{(a,b) \in \mc{R}/(e_1e_2)\cap \Z^2:\begin{tabular}{l l}
     & $f_1\mid g_1(a,b), f_2\mid g_2(a,b),$\nonumber\\
     &  $\gcd(a,b,\psi_{k_1}(d_1)\psi_{k_2}(d_2)/e_1e_2) = 1$
\end{tabular}
\r\}\\
&= \sum_{\ss{e_1\mid \psi_{k_1}(d_1)\\e_2\mid \psi_{k_2}(d_2)}}\l(\f{d_1d_2}{e_1e_2f_1f_2}\r)^2\rho^*(f_1,f_2)\label{rho stuff}.
\end{align}
We denote the quantity 
$$R^*(\mc{R}/(e_1e_2), f_1,f_2; C'(a,b)) - \f{\op{Vol}(\mc{R}/(e_1e_2))\rho^*(f_1,f_2)}{(f_1f_2\Delta)^2}$$
by $E(e_1,e_2,f_1,f_2)$. Combining (\ref{sum over e}), (\ref{R stuff}) and (\ref{rho stuff}), we have
\begin{align}
    &\sum_{\ss{d_1\leq D_1,d_2\leq D_2\\ (d_1,d_2)=(d_1d_2,\Delta)=1}}\sup_{P(\mc{R})\leq N}\l|r(d_1,d_2)\r|\nonumber\\
    &=\sum_{\ss{d_1\leq D_1,d_2\leq D_2\\ (d_1,d_2)=(d_1d_2,\Delta)=1}}\sup_{P(\mc{R})\leq N}\sum_{\ss{e_1\mid \psi_{k_1}(d_1)\\e_2\mid \psi_{k_2}(d_2)}}|E(e_1,e_2,f_1,f_2)|\\
    &\leq \sum_{e_1\leq D_1,e_2\leq D_2}\sum_{\ss{f_1\leq D_1/e_1, f_2 \leq D_2/e_2\\ (f_1,f_2) = (f_1f_2,\Delta) =1}}\delta(e_1,f_1)\delta(e_2,f_2)\sup_{P(\mc{R})\leq N}\l|E(e_1,e_2,f_1,f_2)\r|\label{ready to use stars},
\end{align}
where for integers $e,f,k,D \geq 1$, we have defined
$$\delta(e,f) = \#\l\{d\leq D: e\mid \psi_{k}(d), f = \f{d}{\gcd(d,e^{k})}\r\}.$$

We claim that $\delta(e,f) \ll e^{\eps}$. To see this, suppose that $p$ is a prime and let $r=\nu_p(d), s=\nu_p(e)$ and  $t = \nu_p(f)$. There is a unique choice of $r$ for a given $k,s$ and $t$ provided that $t>0$, namely, $r=ks+t$. If $t=0$, then we deduce from $f = d/\gcd(d,e^{k})$ that $r\leq ks$. Taking a product over primes, we conclude that each $d$ enumerated by $\delta(e,f)$ is a divisor of $e^k$ multiplied by a quantity that is uniquely determined by $e$ and $f$. The claim follows, since the number of divisors of $e^k$ is $O(e^{\eps})$. In our situation, where $e_1\leq D_1$ and $e_2 \leq D_2$, we obtain $\delta(e_1,f_1)\delta(e_2,f_2) \ll (D_1D_2)^{\eps}.$
Therefore, applying Lemma \ref{lemma 3.2*} for each choice of $e_1,e_2$ in (\ref{ready to use stars}), we conclude that
\begin{align*}
    &\sum_{\ss{d_1\leq D_1,d_2\leq D_2\\ (d_1,d_2)=(d_1d_2,\Delta)=1}}\sup_{P(\mc{R}\leq N)}\l|r(d_1,d_2)\r|\\
    &\ll (D_1D_2)^{\eps}\sum_{\ss{e_1\leq D_1\\e_2\leq D_2}}\l(\f{D_1D_2}{e_1e_2}+N\l(\f{D_1D_2}{e_1e_2}\r)^{1/2}+\f{ND_2}{e_2}\r)\\
    &\ll (D_1D_2)^{\eps}\l(D_1D_2 + N(D_1D_2)^{1/2}+ND_2\r),
\end{align*}
which completes the proof of Proposition \ref{lod result}.

\begin{remark}\label{fun fact}
If $g_2(a,b)\neq 0$ for all $(a,b) \neq (0,0)$, then we do not need to consider the case $g_2(a,b)=0$ in the analysis of the sum in (\ref{MAB}), and so in our final level of distribution result, we do not require the assumption $D_2 \ll N^{1-\eps}$. 
\end{remark}

When $g_1(a,b)$ does contain linear factors, we can still obtain a basic level of distribution result from the above argument using the trivial estimate $\lambda_1(\mc{C})^{-1}\leq 1$ in (\ref{R1x}). This establishes the following lemma.

\begin{lemma}\label{linear lod}
Let $g_1,g_2$ be arbitrary binary forms with nonzero discriminant. Then for any $\eps>0$, there exists $\delta>0$ such that for any $D_1,D_2>0$ with $D_1D_2 \ll N^{1-\eps}$ , we have 
$$\sum_{\substack{d_1\leq D_1, d_2\leq D_2\\ \gcd(d_1,d_2)=\gcd(d,\Delta)=1}}\sup_{P(\mc{R})\leq N}\left|r(d_1,d_2)\right|\\
\ll N^{2-\delta}.$$
\end{lemma}

\section{Application of the beta sieve}\label{section: application of the beta sieve}
In this section, we prove Theorem \ref{main sieve result} and Theorem \ref{main sieve result for cubic factors} by combining the level of distribution results from Section \ref{section: levels of distribution} with the beta sieve of Rosser and Iwaniec \cite[Theorem 11.12]{friedlander2010opera}. We state the precise version of this theorem we need in Theorem \ref{taylored verison of beta sieve}.

We recall some of the notation from Section \ref{section: main sieve results for binary forms}. We fix a region $\mc{R}=\mc{B}N$ for some $\mc{B}\subseteq [-1,1]^2$ as in Section \ref{section: main sieve results for binary forms}. Then $\mc{R}$ has volume $\gg N^2$ and perimeter $\ll N$. Let $d$ denote the largest degree among the irreducible factors of $f$. (We specialise to the cases $d=2$ and $d=3$ later. The reason our methods are unable to deal with larger values of $d$ is explained in Remark \ref{why we can't handle quartics}.) Then there exists $x \ll N^d$ such that the largest prime factor of $f(a,b)$ for $(a,b)\in \mc{R}\cap \Z^2$ is strictly less than $x$. Let $S$ be a finite set of primes, including all primes dividing the discriminant of $f(x,y)$. In Section \ref{section: the sums s4i}, we also append to $S$ all primes bounded by some constant $P_1$. Let $\Delta$ be a an integer with only prime factors in $S$. Without loss of generality, we may assume every prime in $S$ divides $\Delta$, because taking a multiple of $\Delta$ can only decrease the sifting function $S(\mc{P}, \mc{B}, N)$ from (\ref{def of sifting function}). Additionally, by taking an appropriate multiple of $\Delta$ and appropriate lifts of $a_0, b_0$, we may assume that $\nu_p(f(a_0, b_0)) < \nu_p(\Delta)$ for all $p\in S$. 

All implied constants in this section are allowed to depend on $\Delta$. Let $\mc{P}$ be a set of primes disjoint from $S$ satisfying (\ref{density 1}) and (\ref{density 3}). We also define $\mc{P'}$ to be the set of primes not in $\mc{P}\cup S$. Let $P(x)$ denote the product of primes in $\mc{P}_{< x}$, and similarly for $P'(x)$. We also define
$X = \op{Vol}(\mc{R})/\Delta^2$.

For a sequence of non-negative real numbers $\mc{A} = (a_n)$, and a parameter $z\geq 1$, we define the sifting function
$$ S(\mc{A}, \mc{P}, z) = \sum_{\gcd(n,P(z))=1}a_n.$$
We make the choice
\begin{equation}\label{choice of an}
a_n = \#\{(a,b) \in \mc{R}\cap \Z^2: C(a,b), f(a,b) = n\},
\end{equation}
so that 
\begin{equation}\label{def of sifting function}
\begin{split}
S(\mc{A},\mc{P},x)&=\#\{(a,b)\in \mc{R}\cap \Z^2: C(a,b), \gcd(f(a,b), P(x))=1\}\\
&=S(\mc{B},\mc{P},N),
\end{split}
\end{equation}
where $S(\mc{B},\mc{P},N)$ is as defined in (\ref{sieves:main sieve problem}). Our aim is to prove that $S(\mc{A}, \mc{P}, x) >0$ for sufficiently large $N$ (which may depend on $\Delta$).
For a prime $p\in \mc{P}$ and for any $i\in \{0,\ldots, k\}$, we also consider the sequences $\mc{A}_p, \mc{A}_p^{(i)}$ defined similarly to (\ref{choice of an}) but with the additional conditions $p\mid f(a,b), p\mid f_i(a,b)$ respectively, so that
\begin{align*}
    S(\mc{A}_p, \mc{P}, p) &=\#\{(a,b)\in \mc{R}\cap \Z^2: C(a,b), p\mid f(a,b), \gcd(f(a,b), P(p))=1\},\\
    S(\mc{A}_p^{(i)}, \mc{P},p) &=\#\{(a,b)\in \mc{R}\cap \Z^2: C(a,b), p\mid f_i(a,b), \gcd(f(a,b), P(p))=1\}.
\end{align*}

Using the Buchstab identity, we have
$$S(\mc{A},\mc{P},x) = S(\mc{A},\mc{P},N^{\gamma})-\sum_{\substack{N^{\gamma} \leq p < x\\p \in \mc{P}}}S(\mc{A}_p, \mc{P},p),$$
for a parameter $\gamma\in (0,1)$ to be chosen later. We denote $S(\mc{A},\mc{P},N^{\gamma})$ by $S_1$. If $p\mid f(a,b)$ then $p\mid f_i(a,b)$ for some $i$. Therefore, we have the decomposition
$$S(\mc{A},\mc{P},x) \geq S_1 -\sum_{i=0}^m S_2^{(i)}- \sum_{i=m+1}^k \left(S_3^{(i)}+S_4^{(i)}\right),$$
where
\begin{equation}\label{sieve decomposition}
\begin{split}
S_2^{(i)} &=\sum_{\substack{N^{\gamma}\leq p \ll N\\p \in \mc{P}}}S(\mc{A}_p^{(i)},\mc{P},p),\\ S_3^{(i)}&=\sum_{\substack{N^{\gamma}\leq p < N^{\beta_i}\\p \in \mc{P}}}S(\mc{A}_p^{(i)}, \mc{P},p), \\
S_4^{(i)}&=\sum_{\substack{N^{\beta_i}\leq p< x\\p \in \mc{P}}}S(\mc{A}_p^{(i)}, \mc{P},p),
\end{split}
\end{equation}
for parameters $\beta_i\geq \gamma$ to be chosen later.  

\subsection{The beta sieve}\label{section: intro to beta} Like most combinatorial sieves, the beta sieve provides a mechanism to estimate sifting functions of the form $S(\mc{A},\mc{P},z)$ given arithmetic information about the related quantities
$|\mc{A}_d| := \sum_{d\mid n}a_n$ for squarefree integers $d$. More specifically, we require an approximation $|\mc{A}_d| = |\mc{A}_1|g(d) + r(d)$, where $g(d)$ is a multiplicative function supported on squarefree integers and 
 \[
    R(z) := \sum_{\ss{d\leq z\\ d \textrm{ squarefree}}}|r(d)|
    \]
is small. Define 
$$V(z) = \sum_{d \mid P(z)}\mu(d)g(d) = \prod_{p \in \mc{P}_{\leq z}}(1-g(p)).$$ 
We shall assume that for some $\kappa, L \geq 0$, we have
\begin{equation}\label{sieves: stronger assumption on g}
    V(w) \leq \l(\f{\log z}{\log w}\r)^{\kappa}\l(1+\f{L}{\log w}\r)V(z)
\end{equation}
for all $2 \leq w \leq z$. 

For some choice of sieve weights $\Lambda^{\pm} =(\lambda_d^{\pm})_{d \geq 1}$, define
\[
V^{\pm}(z) = \sum_{d\mid P(z)}\lambda_d^{\pm}g(d). 
\]
We recall that if $\Lambda^{\pm}$ are upper and lower bound sieves of level $z$, i.e., if $\lambda^{\pm}_d$ are supported on squarefree integers $d<z$ and 
\[
\sum_{d\mid m}\lambda^-_d \leq \sum_{d\mid m}\mu(d) \leq \sum_{d\mid m}\lambda^+_d
\]
for all integers $m$, then we have
    \begin{equation}\label{sieve setup}
        \begin{split}
        S(\mc{A}, \mc{P}, z) &\leq |\mc{A}_1|V^+(z) + R(z),\\
        S(\mc{A}, \mc{P}, z) &\geq |\mc{A}_1|V^-(z) - R(z).
        \end{split}
    \end{equation}

The main theorem of the beta sieve we apply is given in \cite[Theorem 11.12]{friedlander2010opera}. We record this theorem here for convenience, in the special case $s=1$. 

\begin{theorem}\label{taylored verison of beta sieve}
    Suppose $\kappa, L$ are such that the assumption (\ref{sieves: stronger assumption on g}) holds. Then there is a choice of upper and lower bound sieve weights $\Lambda^{\pm}$ (taking values in $\{-1,0,1\}$), and explicit constants $A(\kappa), B(\kappa) \geq 0$ such that, as $z\ra \infty$, we have
\begin{equation}\label{Thm 11.12}
    \begin{split}
    V^+(z) &\leq  \l(A(\kappa) + o(1)\r)V(z),\\
    V^-(z) &\geq \l(B(\kappa) + o(1)\r)V(z).
    \end{split}
\end{equation}
\end{theorem}

\begin{notation}
    Throughout the remainder of this chapter, for a sequence $\mc{A}$, a set of primes $\mc{P}$, a multiplicative function $g$, and a sifting level $z\geq 1$, we define $\Lambda^{\pm}(\mc{A},\mc{P},g,z)$ to be the corresponding upper and lower bound beta sieves with these parameters. We sometimes apply (\ref{Thm 11.12}) directly without reference to a sequence, in which case the parameter $\mc{A}$ is omitted from the notation.   
\end{notation}

In our applications of the beta sieve, the required bounds on $R(z)$ are provided by Corollary \ref{lod} and Lemma \ref{linear lod}. For $i\in\{0,\ldots, k\}$, we define multiplicative functions 
\begin{align*}
    \rho_i(d_1,d_2) &=\#\{a,b \Mod{d_1d_2}: d_1\mid f_i(a,b), d_2\mid f(a,b)\},\\
    \rho_i(d) &=\#\{a,b \Mod d: f_i(a,b) \equiv 0 \Mod{d}\},\\
    \rho(d) &= \#\{a,b \Mod d: f(a,b) \equiv 0 \Mod{d}\}.
\end{align*}
We note that the function $\rho_i(d_1,d_2)$ is the same as the function $\rho(d_1,d_2)$ from (\ref{definition of multiplicative function rho}) with $g_1(x,y) = f_i(x,y)$ and $g_2(x,y) = f(x,y)$, but in this section we add a subscript to keep track of the dependence on $i$. When $\gcd(d_1,d_2) = 1$, we have $\rho_i(d_1,d_2) = \rho_i(d_1)\rho(d_2)$. Moreover, for any $i\in \{1,\ldots, k\}$ and any prime $p\notin S$, we have 
\begin{align}
\rho_i(p) &= \nu_i(p)(p-1)+1,\label{rhoi to nui}\\ 
\rho(p) &= \nu(p)(p-1)+1\label{rho to nu},
\end{align}
where $\nu_i(p)$ and $\nu(p)$ are as defined in (\ref{definition of nui}) and (\ref{definition of nu}). We define multiplicative functions
\[
g(p) := \f{\rho(p)}{p^2}, \qquad g_i(p) := \f{\rho_i(p)}{p^2}
\]
and define
\begin{equation}\label{choice of V and Vi}
     V(x)= \prod_{\ss{p\in \mc{P}_{\leq z}}}\l(1-g(p)\r), \qquad
     V_i(x)= \prod_{\ss{p\in \mc{P}'_{\leq z}}}\l(1-g_i(p)\r)
\end{equation}
for $i\in \{1,\ldots, k\}$. 

\subsection{Sieve dimensions}\label{subsection: sieve dims} 
We prove in Lemma \ref{sieve dims} that the functions $V,V_i$ defined above satisfy the hypothesis (\ref{sieves: stronger assumption on g}) with the sieve dimensions 
$$\kappa \colonequals \alpha\theta, \qquad \kappa_i \colonequals 1-\alpha\theta_i,$$
where $\alpha, \theta_i$ and $\theta$ are as defined in (\ref{density 1}), (\ref{density 3}) and (\ref{density 2}). 

In the notation of Theorem \ref{taylored verison of beta sieve}, we write $A=A(\kappa), B=B(\kappa)$ and $A_i = A(\kappa_i)$. We assume throughout this section that $\kappa <1/2$, and so $\kappa_i>1/2$. Then $A$ and $B$ are defined in \cite[Equations (11.62), (11.63)]{friedlander2010opera} (see also Section \ref{section:Details of the numerical computations}), and $A_i$ is defined in \cite[Equations (11.42), (11.57)]{friedlander2010opera}. A table of numerical values of these constants can be found in \cite[Section 11.19]{friedlander2010opera}.

\begin{lemma}\label{sieve dims}
Let $x\geq 1$. For $i\in \{1,\ldots,k\}$, and $V,V_i$ as in (\ref{choice of V and Vi}), there exist constants $c,c_i>0$ such that
\begin{align}
     V(x)&= \f{c}{(\log x)^{\kappa}}\l(1+O((\log x)^{-1})\r),\label{small sieve dim}\\
     V_i(x)&= \f{c_i}{(\log x)^{\kappa_i}}\l(1+O((\log x)^{-1})\r)\label{big sieve dim}.
\end{align}
The asymptotic in (\ref{big sieve dim}) also holds for $i=0$ when $f_0 \not\equiv 1$. 
\end{lemma}

\begin{proof}
We follow a similar approach to \cite[Lemma 4.2]{irving2017cubic}. Below, we denote by $C$ a constant which is allowed to vary from line to line. We have
\begin{align}
    \log V(x)&=-\sum_{p \in \mc{P}_{\leq x}}\l(\sum_{m=1}^{\infty}\f{\rho(p)^m}{mp^
{2m}}\r)\label{log the prod}\\
    &=-\sum_{p\in \mc{P}_{\leq x}}\f{\nu(p)(p-1)+1}{p^2} + C + O((\log x)^{-1}).\\
    &=-\sum_{p\in \mc{P}_{\leq x}}\f{\nu(p)}{p} + C + O((\log x)^{-1}),\label{prelim sieve dim}
\end{align}
where in (\ref{prelim sieve dim}) we have used that $\nu(p) \leq \deg f$ for all but finitely many primes $p$. 
To estimate the sum in (\ref{prelim sieve dim}), we apply partial summation, together with our assumption (\ref{density 2}). For $t\geq 2$, we define
\begin{equation}\label{at}
A_t = \sum_{p \in \mc{P}_{\leq t}}\nu(p).
\end{equation}
Then
\begin{align}
    \sum_{p \in \mc{P}_{\leq x}}\f{\nu(p)}{p}&=\f{A_x}{x} + \int_{2}^x \f{A_t}{t^2}\textrm{d}t\nonumber\\
    &=\kappa\int_{2}^x \f{\pi(t)\l(1+O\l((\log t)^{-1}\r)\r)}{t^2} \textrm{d}t + O((\log x)^{-1})\nonumber\\
    &=\kappa\int_{2}^x \f{\textrm{d}t}{t\log t} + C+O((\log x)^{-1})\nonumber\\
    &=\kappa\log\log x + C + O((\log x)^{-1})\label{exponentiate me}.
\end{align}
We deduce (\ref{small sieve dim}) by taking the exponential of (\ref{exponentiate me}). 

We can prove (\ref{big sieve dim}) in a similar way. When $i=0$ and $f_0 \not\equiv 1$, we have $\rho_0(p) = p$, and so the result is a consequence of Mertens' theorem \cite[Equation (2.16)]{iwanieckowalski2021analytic}. For any $i\in \{1,\ldots, k\}$, we have 
\begin{align}
    \log V_i(x)&=\sum_{p\in \mc{P'}_{\leq x}}\f{\nu_i(p)}{p} + C + O((\log x)^{-1})\nonumber\\
    &=\sum_{p\leq x \textrm{ prime}}\f{\nu_i(p)}{p} - \sum_{p\in \mc{P}_{\leq x}}\f{\nu_i(p)}{p}  + C + O((\log x)^{-1})\label{partial summation pieces}.
\end{align}

Similarly to above, using partial summation and (\ref{density 3}), we have 
\begin{equation}\label{the nui piece}\sum_{p\in \mc{P}_{\leq x}}\f{\nu_i(p)}{p} = \alpha\theta_i \log\log x + C+ O((\log x)^{-1}).\end{equation}
To treat the first sum in (\ref{partial summation pieces}), we define $L_i$ to be the number field generated by $f_i$. For all but finitely many primes $p$, the quantity $\nu_i(p)$ is equal to the number of degree one prime ideals $\mf{p}$ in $L_i$ above $p$. Let $\pi_{L_i}(x)$ denote the number of prime ideals $\mf{p}$ in $L$ of norm at most $x$. This count is dominated by degree one ideals. In fact, the number of prime ideals of degree at least $2$ enumerated by $\pi_{L_i}(x)$ is $O(x^{1/2})$, because such an ideal must lie over a rational prime $p \leq x^{1/2}$ and each rational prime $p$ has at most $[L_i:\Q]$ prime ideals in $L_i$ lying above it. Therefore, 
$$\sum_{p\leq x \textrm{ prime}}\nu_i(p) = \pi_{L_i}(x) + O(x^{1/2}).$$
Using partial summation, as above, together with the Prime ideal theorem \cite{mitsui1968prime}, we deduce that
\begin{equation}\label{nui over all primes}
\sum_{p\leq x \textrm{ prime}}\f{\nu_i(p)}{p} = \log\log x + C + O((\log x)^{-1}).
\end{equation}
Combining this with (\ref{the nui piece}) and taking exponentials, we deduce the asymptotic in (\ref{big sieve dim}). \end{proof}

In the following lemma, we record three more useful estimates following similar arguments to Lemma \ref{sieve dims}. 

\begin{lemma}\label{useful facts about rhoi} There exists constants $C,C_i>0$ such that
\begin{align}
    \sum_{p \in \mc{P}_{\leq x}} \f{\rho_i(p)}{p^2} &=(1-\kappa_i)\log\log x + C + O((\log x)^{-1}),\label{rhoi over P}\\
    \sum_{p \in \mc{P}'_{\leq x}} \f{\rho_i(p)}{p^2} &=\kappa_i\log\log x + C_i + O((\log x)^{-1})\label{rhoi over P'},\\
    \sum_{p \in \mc{P}'_{\leq x}} \f{\rho_i(p)}{p^2}\log p &=\kappa_i\log x +O(1)\label{rhoi log}.
\end{align}
\end{lemma}

\begin{proof}
The estimates (\ref{rhoi over P}) and (\ref{rhoi over P'}) are immediate consequences of (\ref{exponentiate me}), (\ref{the nui piece}) and (\ref{nui over all primes}), together with fact that
$$\f{\rho_i(p)}{p^2} =\f{(p-1)\nu_i(p)+1}{p^2}= \f{\nu_i(p)}{p} + O(p^{-2}).$$ To prove (\ref{rhoi log}), we proceed via partial summation in a very similar manner to (\ref{exponentiate me}). We recall from the Prime number theorem that
\begin{equation}\label{PNT}
\pi(t) = \f{t}{\log t}+\f{t}{(\log t)^2} + O\l(\f{t}{(\log t)^3}\r).
\end{equation}
For $A_t$ as defined in (\ref{at}), we have
\begin{align*}
    \sum_{p \in \mc{P}'_{\leq x}}\f{\rho_i(p)}{p^2}\log p &= \sum_{p \in \mc{P}'_{\leq x}}\f{\nu_i(p)}{p}\log p + O(1)\\
&=\f{A_x\log x}{x} - \int_{2}^x A_t\l(\f{\log t}{t}\r)' \textrm{d}t + O(1)\\
&=\kappa_i\int_2^x \f{(\log t -1)\pi(t)(1+O((\log t)^{-A})}{t^2} \textrm{d}t + O(1)\\
&=\kappa_i\int_{2}^x \f{1}{t} + O\l(\f{1}{t(\log t)^{2}}\r) \textrm{d}t + O(1) \quad \textrm{ (from } (\ref{PNT}))\\
&=\kappa_i\log x + O(1),
\end{align*}
as required.
\end{proof}

\subsection{The sum $S_1$}
We apply the lower bound sieve $\Lambda^{-}(\mc{A}, \mc{P}, g, N^{\gamma})$ and the level of distribution result from Corollary \ref{lod} with $g_1(x,y)=1, g_2(x,y)=f(x,y), D_1=1$ and $D_2= N^{\gamma}$. The hypotheses of Corollary \ref{lod} require that $\gamma<1$. We obtain
\begin{equation}
S_1 \geq (B+o(1))XV(N^{\gamma}).
\end{equation}
By Lemma \ref{sieve dims}, we have 
$$V(N^{\gamma})\sim \f{c}{(\log N^{\gamma})^{\kappa}}.$$
For any $\eps>0$, taking $\gamma$ sufficiently close to $1$, we obtain
\begin{equation}\label{s1}
    S_1 \geq \frac{(cB-\eps+o(1))X}{(\log N)^{\kappa}}.
\end{equation}

\subsection{The sums $S_2^{(i)}$}\label{section: the sums s2i}
We write $f_i(a,b) = pr$, for $p \in \mc{P}$ and $N^{\gamma}< p \ll N$. We apply the \textit{switching principle}, which transforms the sum over $p$ defining $S_2^{(i)}$ into a much shorter sum over the variable $r$. 

Let $R=N^{1-\gamma}$. The sums $S_2^{(i)}$ only involve linear factors $f_i(x,y)$, since we assume $i\in \{0,\ldots, m\}$. Therefore, for $(a,b) \in \mc{R}\cap \Z^2$ we have $f_i(a,b) \ll N$, and so $|r|\ll R$. Let $z = N^{1/3}$. We shall take $\gamma$ arbitrarily close to 1; for now, we assume that $\gamma>2/3$. Then by definition of $S(\mc{A}_p^{(i)},\mc{P},p)$, we know that $\gcd(r, P(R))=1$ and $\gcd(f(a,b),P(z))=1$. 

Let $r' = |r|/\gcd(r,\Delta)$. We now explain why $r'$ only has prime factors in $\mc{P}'$. Fix $l \in S$. Recall the assumption that $\nu_l(f(a_0,b_0)) < \nu_l(\Delta)$. Since $f_i(a_0,b_0) \mid f(a_0,b_0)$ and $f_i(a,b) \equiv f_i(a_0,b_0) \Mod{\Delta}$ (by the condition $C(a,b)$), we have 
\begin{equation}\label{l-adic valuations}
\nu_l(r) = \nu_l(f_i(a, b)) \leq \nu_l(f(a,b)) =  \nu_l(f(a_0,b_0)) < \nu_l(\Delta). 
\end{equation}
Therefore, $\gcd(r',\Delta) =1$, and by the assumption that every prime in $S$ divides $\Delta$, this implies $r'$ has no prime factors in $S$. Moreover, $p$ is the smallest prime in $\mc{P}$ dividing $f_i(a,b)$. Since $r<p$, this means that $r'$ has no prime factors in $\mc{P}$, and hence only has prime factors in $\mc{P}'$.

On the other hand, $f_i(a,b)/r'$ has no prime factors in $\mc{P}'$. Therefore, 
\begin{equation}\label{switching decomp 1}
    S_2^{(i)} \leq \sum_{\ss{r'\ll R\\ \gcd(r',P(R)\Delta)=1}}S_2^{(i)}(r'),
\end{equation}
where
\begin{equation}\label{switching decomp 2}
S_2^{(i)}(r') = \#\left\{(a,b) \in \mc{R}\cap \Z^2: 
\begin{tabular}{l l l}
     & $C(a,b), r'\mid f_i(a,b),$ \\
     & $\gcd(f_i(a,b)/r', P'(z))=1$\\  
     & $\gcd(f(a,b),P(z))=1$
\end{tabular} \right\}.
\end{equation}
Below, for convenience, we change notation from $r'$ back to $r$. 

Let $\Lambda_1^+, \Lambda_2^+$ be upper bound sieves of level $z$. Defining $A(dr,e), r(dr,e)$ as in Section \ref{section: levels of distribution} with $g_1(x,y) = f_i(x,y)$ and $g_2(x,y) = f(x,y)$, and writing
\[
m_1 = \gcd\l(\f{f_i(a,b)}{r}, P'(z)\r), \qquad m_2= \gcd(f(a,b), P(z)),
\]
we obtain
\begin{align}
    S_2^{(i)}(r) &\leq \sum_{\ss{(a,b) \in \mc{R}\cap \Z^2\\ C(a,b)}}\sum_{d \mid m_1}\mu(d)\sum_{e\mid m_2}\mu(e)\nonumber\\
    &\leq \sum_{\ss{(a,b) \in \mc{R}\cap \Z^2\\ C(a,b)}}\sum_{d \mid m_1}\lambda_1^+(d)\sum_{e\mid m_2}\lambda_2^+(e)\nonumber\\
    &\leq\sum_{\ss{d \mid P'(z)\\ \gcd(d,r)=1}}\lambda_1^+(d)\sum_{e\mid P(z)}A(dr,e)\lambda_2^+(e)\nonumber\\
    &=\sum_{\ss{d \mid P'(z)\\ \gcd(d,r)=1}}\lambda_1^+(d)\sum_{e\mid P(z)}\lambda_2^+(e)\l(\f{X\rho(dr,e)}{(dre)^2} + r(dr,e)\r)\nonumber\\
    &\leq\l(\f{X\rho_i(r)}{r^2}\sum_{\ss{d \mid P'(z)\\ \gcd(d,r) = 1}}\lambda_1^+(d)g_i(d)\sum_{\ss{e \mid P(z)}}\lambda_2^+(e)g(d)\r) + \sum_{\ss{d,e \leq z\\ \gcd(dr,e) = 1\\ \gcd(dre, \Delta) = 1}}|r(dr,e)|.\label{general UB sieve manipulation}
\end{align}
Choose $\lambda_1^+,\lambda_2^+$ to be the beta sieves $\Lambda^+(\mc{P}'\bs\{p\mid r\}, g_i, N^{1/3}), \Lambda^+(\mc{P}, g, N^{1/3})$ respectively. Then the remainder term in (\ref{general UB sieve manipulation}) can be bounded by 
$$ \sum_{\ss{d,e \leq N^{1/3}\\ \gcd(dr,e) = 1\\ \gcd(dre, \Delta) = 1}}|r(dr,e)| \leq \sum_{\ss{d \ll N^{1/3}R\\e \ll N^{1/3}\\ \gcd(dr,e) = 1\\ \gcd(dre, \Delta) = 1}}|r(d,e)|,$$
which is negligible by Lemma \ref{linear lod} because $N^{2/3}R = N^{5/3-\gamma} \ll N^{1-\eps}$. 

Define a multiplicative function $h_i(r)$ supported on squarefree integers $r$, which is zero unless all prime factors of $r$ are in $\mc{P}'$ and  
\begin{equation}\label{def gir}
h_i(r)=\frac{\rho_i(r)}{r^2}\prod_{p\mid r}\l(1-\f{\rho_i(p)}{p^2}\r)^{-1}
\end{equation}
otherwise. Applying Theorem \ref{taylored verison of beta sieve}, we conclude that 
$$ S_2^{(i)}(r) \leq Xh_i(r)V(N^{1/3})V_i(N^{1/3})(AA_i + o(1)).$$
From Lemma \ref{sieve dims}, we obtain
\[
S_2^{(i)}\leq \frac{(cc_iAA_i+o(1))X}{(\log N^{1/3})^{\kappa + \kappa_i}}\sum_{\ss{r \ll R}}h_i(r).
\]
To deal with the sum over $h_i(r)$, we note that
\begin{equation}\label{crude estimate for gir sum}
\sum_{\ss{r \ll R}}h_i(r) \leq \prod_{\ss{p \in \mc{P}'\\p\ll R}}\l(1+\sum_{m=1}^{\infty} h_i(p^m)\r)= \prod_{\ss{p \in \mc{P}'\\p \ll R}}\l(1-g_i(p)\r)^{-1}(1+O(p^{-2})).
\end{equation}
Applying Lemma \ref{sieve dims} to the product, we obtain
$$\sum_{\ss{r \ll R}}h_i(r) \ll (\log R)^{\kappa_i}.$$
Since $R=N^{1-\gamma}$, we deduce that
$$S_2^{(i)} \ll  \f{cc_iAA_iX}{(\log N)^{\kappa}}\l(\f{(1-\gamma)^{\kappa_i}}{(1/3)^{\kappa_i+\kappa}}\r).$$
Therefore, $S_2^{(i)}$ can be made negligible compared to $S_1$ by taking $\gamma$ arbitrarily close to 1. 

\subsection{The sums $S_3^{(i)}$}\label{section: the sums s3i}
For a fixed prime $p$ and an upper bound sieve $\lambda^+$ of level $z$, we have 
$$S(\mc{A}_p^{(i)},\mc{P}, z) \leq \sum_{d\mid P(z)}\lambda^+(d)A(p,d),$$
where $A(p,d)$ is as in Section \ref{section: levels of distribution} with $g_1(x,y) = f_i(x,y)$ and $g_2(x,y) = f(x,y)$. 

We first deal with the primes $p$ in the interval $I:=(N^{\gamma}, N^{2-\gamma-\eps}]$. For any $p\in I$, let $\lambda^+$ be the beta sieve $\Lambda^+(\mc{A}_p^{(i)}, \mc{P}, g_i(p)g, N^{\gamma})$. We obtain
\begin{align}
    S(\mc{A}_p^{(i)}, \mc{P},p) &\leq S(\mc{A}_p^{(i)}, \mc{P},N^{\gamma})\nonumber\\
    &\leq (A+o(1))XV(N^{\gamma})g_i(p) + \sum_{\ss{d \leq N^{\gamma}\\ \gcd(d,p\Delta)=1}}|r(p,N^{\gamma})|.\label{estimating S2i}
\end{align}
We apply Corollary \ref{lod} with $D_1 = N^{2-\gamma - \eps}$ and $D_2 = N^{\gamma}$. These choices of $D_1,D_2$ satisfy the hypotheses of Corollary \ref{lod}. Taking a sum over $p\in I$, the contribution from the remainder term in (\ref{estimating S2i}) is negligible. We obtain
\begin{equation}
    \sum_{\substack{p \in I\cap \mc{P}}}S(\mc{A}_p^{(i)}, \mc{P},p)\leq (A+o(1)) XV(N^{\gamma})\sum_{p \in I\cap \mc{P}}g_i(p).
\end{equation}
It follows from Lemma \ref{useful facts about rhoi} that
\begin{align*}
    \sum_{\substack{N^{\gamma}<p\leq N^{2-\gamma-\eps}}}g_i(p) &= \log\log N^{2-\gamma-\eps} - \log\log N^{\gamma} + o(1)\\
    &= \log(2-\gamma - \eps) - \log \gamma +o(1). 
\end{align*}
Therefore, the contribution to $S_3^{(i)}$ from this range is negligible if we take $\gamma$ arbitrarily close to 1. 

In the remaining range $N^{2-\gamma -\eps}< p <N^{\beta_i}$, we split into dyadic intervals $(R,2R]$. Note that for $p\in (R,2R]$, the assumption $\gamma <1$ implies that $N^{2-\eps}/R <p$, so $S(\mc{A}_p^{(i)}, \mc{P}, p) \leq S(\mc{A}_p^{(i)}, \mc{P}, N^{2-\eps}/R)$. For each dyadic interval, we apply the beta sieve $\Lambda^+(\mc{A}_p^{(i)}, \mc{P}, g_i(p)g, N^{2-\eps}/R)$ and the level of distribution result from Corollary \ref{lod} with $D_1=2R$ and $D_2  = N^{2-\eps}/R$.  At this point we need to assume that $\beta_i<2$ for all $i$, so that $D_2\geq 1$. We obtain
\begin{align}
    &\sum_{\substack{N^{2-\gamma-\eps}<p< N^{\beta_i}\\p \in \mc{P}}}S(\mc{A}_p^{(i)}, \mc{P},p)\\ &\leq \sum_{\ss{R \textrm{ dyadic}\\N^{2-\gamma -\eps}<R < N^{\beta_i}}}\sum_{\substack{p \in (R,2R]\\p \in \mc{P}}}S(\mc{A}_p^{(i)}, \mc{P},N^{2-\eps}/R)\nonumber\\
    &\leq (A+o(1))X\sum_{\ss{R \textrm{ dyadic}\\ N^{2-\gamma -\eps}<R < N^{\beta_i}}}V(N^{2-\eps}/R)\sum_{\substack{p \in (R,2R]\\p \in \mc{P}}}g_i(p)\nonumber\\
    &\leq \frac{(cA+o(1))X}{(\log N)^{\kappa}}\sum_{\substack{ N^{2-\gamma-\eps} \leq p< N^{\beta_i}\\p \in \mc{P}}}\frac{g_i(p)}{(2-\eps - \f{\log p}{\log N})^{\kappa}}\label{tbetaikappa},
\end{align}
where the last line follows from Lemma \ref{sieve dims} and the fact that $V(N^{2-\eps}/R)<V(N^{2-\eps-\log p/\log N})$ for all $p\in (R,2R]$. 

We denote the sum in (\ref{tbetaikappa}) by $T(\beta_i, \kappa)$. Since $\gamma<1$, we have $2-\gamma-\eps >1$ for sufficiently small $\eps$, and so we may upper bound $T(\beta_i,\kappa)$ by enlarging its range of summation to $N<p< N^{\beta_i}$. Define
$$ A(t) = \sum_{p \in \mc{P}_{\leq t}}g_i(p), \qquad h(t) = \l(2-\eps - \f{\log t}{\log N}\r)^{-\kappa}.$$
From Lemma \ref{useful facts about rhoi}, we have $A(t)= (1-\kappa_i)\log\log t + C+o(1)$ for some constant $C$. In particular, for any $r>0$, we have $A(N^r)-A(N) = (1-\kappa_i)\log r + o(1)$. Applying summation by parts, followed by the substitution $t=N^s$, we obtain 
\begin{align}
T(\beta_i,\kappa) &\leq (A(N^{\beta_i})-A(N))h(N^{\beta_i}) - \int_{N}^{N^{\beta_i}}(A(t)-A(N))h'(t) \textrm{d}t \nonumber \\
& =(A(N^{\beta_i})-A(N))h(N^{\beta_i}) - \int_{1}^{\beta_i}(A(N^s)-A(N))\f{\del h(N^s)}{\del s}\textrm{d}s \nonumber \\
&=(1-\kappa_i)\log \beta_i h(N^{\beta_i}) -(1-\kappa_i)\int_{1}^{\beta_i}(\log s) \f{\del h(N^s)}{\del s}\textrm{d}s + o(1) \nonumber \\
&=(1-\kappa_i)\int_{1}^{\beta_i}\f{h(N^s)}{s}\textrm{d}s + o(1).\label{fiddling with eps}
\end{align}
Taking $\eps$ sufficiently small and redefining it appropriately, we may replace $h(N^s)$ by $(2-s)^{-\kappa}$ at the cost of adding $\eps$ in (\ref{fiddling with eps}). Combining with (\ref{tbetaikappa}), we conclude that for any $\eps >0$, 
\begin{equation}\label{sum s3i}
   S_3^{(i)}\leq \frac{(cA+\eps+o(1))X}{(\log N)^{\kappa}}\cdot (1-\kappa_i) \int_{1}^{\beta_i}(2-s)^{-\kappa}\f{\textrm{d}s}{s}. 
\end{equation}

Due to the factor $1-\kappa_i = \alpha\theta_i$ appearing in the above estimate, $S_3^{(i)}$ becomes negligible compared to $S_1$ as $\alpha \ra 0$. We perform a more precise quantitative comparison in Section \ref{section: proof of quadratic sieve result}. 

\subsection{The sums $S_4^{(i)}$}\label{section: the sums s4i}
We begin in a similar manner to the sums $S_2^{(i)}$, by writing $f_i(a,b) = pr$, for $p \in \mc{P}$, where now $N^{\beta_i}\leq  p < x$ and $R=x/N^{\beta_i}$. Let $D_1 = N^{\eta_1}$ and $D_2 = N^{\eta_2}$ for parameters $\eta_1, \eta_2 >0$ (which may depend on $r$ and $i$) to be chosen later. We assume $\eta_2 < \beta_i$ so that the condition $\gcd(f(a,b), P(p)) =1$ can be replaced by the weaker condition $\gcd(f(a,b), P(N^{\eta_2}))=1$. Proceeding as in Section \ref{section: the sums s2i}, we have 
\begin{equation}
    S_4^{(i)} \leq \sum_{\ss{r\ll R\\ \gcd(r,P(N^{\beta_i})\Delta)=1}}S_4^{(i)}(r),
\end{equation}
where
\begin{equation}
S_4^{(i)}(r) = \#\left\{(a,b) \in \mc{R}\cap \Z^2: 
\begin{tabular}{l l l}
     & $C(a,b), r\mid f_i(a,b),$ \\
     & $\gcd(f_i(a,b)/r, P'(N^{\eta_1}))=1$\\  
     & $\gcd(f(a,b),P(N^{\eta_2}))=1$
\end{tabular} \right\}.
\end{equation}
Similarly to (\ref{general UB sieve manipulation}), for any upper bound sieves $\lambda_1^+, \lambda_2^+$ of levels $N^{\eta_1}, N^{\eta_2}$, the quantity $S_4^{(i)}(r)$ is bounded above by 
\[
\l(\f{X\rho_i(r)}{r^2}\sum_{\ss{d \mid P'(N^{\eta_1})\\ \gcd(d,r) = 1}}\lambda_1^+(d)g_i(d)\sum_{\ss{e \mid P(N^{\eta_2})}}\lambda_2^+(e)g(d)\r) + \sum_{\ss{d\leq N^{\eta_1}, e \leq N^{\eta_2}\\ \gcd(dr,e) = 1\\ \gcd(dre, \Delta) = 1}}|r(dr,e)|.
\]
In order to ensure the error terms from applying the level of distribution result from Corollary \ref{lod} are negligible after summing over $r$, we need $\eta_1,\eta_2$ to satisfy
\begin{equation}\label{optimisation problem}
  \eta_1, \eta_2 >0, \quad \eta_2 \leq 1-\delta \qquad (\textrm{for all }r\leq R),\\
    \end{equation}
\begin{equation}\label{optimisation problem 2}
     \sum_{r\leq R}N^{\eta_1+\eta_2}\leq N^{2-\delta}.
\end{equation}

\begin{remark}\label{why we can't handle quartics} There are no $\eta_1,\eta_2$ satisfying (\ref{optimisation problem}) and (\ref{optimisation problem 2}) unless $R\leq N^{2-\delta}$, i.e., $x \leq N^{2-\delta + \beta_i}$. Since we had to assume $\beta_i<2$ in the treatment of the sums $S_3^{(i)}$ in Section \ref{section: the sums s3i}, this means our approach cannot handle the case $d \geq 4$, in which $f(x,y)$ has an irreducible factor of degree $\geq 4$. Therefore, we proceed with the additional assumption that $d\leq 3$.
\end{remark}

Choose $\lambda^+_1, \lambda^+_2$ to be the beta sieves $\Lambda^+(\mc{P}'\bs\{p\mid r\}, g_i, N^{\eta_1}), \Lambda^+(\mc{P}, g, N^{\eta_2})$ respectively. Recalling the definition of $h_i(r)$ from (\ref{def gir}), we obtain
\begin{equation}\label{s3ir}
S_4^{(i)}(r) \leq (AA_i+o(1))Xh_i(r)V_i(N^{\eta_1})V(N^{\eta_2})
\end{equation}

Using Lemma \ref{sieve dims} to estimate the products in (\ref{s3ir}), and taking a sum over $r$, we obtain
\begin{equation}\label{ready to optimise etas}
S_4^{(i)}\leq \frac{(cc_iAA_i+o(1))X}{(\log N)^{\kappa_i+\kappa}}\sum_{r\leq R}\f{h_i(r)}{\eta_1^{\kappa_i}\eta_2^{\kappa}}.
\end{equation}

We divide the sum over $r$ into dyadic intervals $r \in (R_1,2R_1]$, and take $\eta_1,\eta_2$ depending only on $R_1$ and $i$. To obtain a good estimate for (\ref{ready to optimise etas}), we maximise $\eta_1^{\kappa_i}\eta_2^{\kappa}$ subject to the constraints
$$ \eta_1,\eta_2 >0, \qquad \eta_2\leq 1-\delta, \qquad \eta_1+\eta_2 \leq 2-\delta - \f{\log R_1}{\log N}.$$
By a similar computation to \cite[Section 6.5]{irving2017cubic}, the optimal solution is 
\begin{align*}
   \eta_1 &= \f{\kappa_i}{\kappa+\kappa_i}\l(2-\delta-\f{\log R_1}{\log N}\r),\\
    \eta_2 &= \f{\kappa}{\kappa+\kappa_i}\l(2-\delta-\f{\log R_1}{\log N}\r).
\end{align*}
We note that for $\delta>0$ sufficiently small, this solution satisfies $\eta_2 \leq 1-\delta$ due to the assumption $\kappa <1/2$. Substituting this choice of $\eta_1,\eta_2$ into (\ref{ready to optimise etas}), we obtain
\begin{equation}\label{partially sum me}
\sum_{r\leq R}\f{h_i(r)}{\eta_1^{\kappa_i}\eta_2^{\kappa}}\leq \sum_{r\leq R}w(r,\delta)h_i(r),
\end{equation}
where
\begin{equation}\label{the function wrdelta}
w(r,\delta) = \l(\f{\kappa}{\kappa+\kappa_i}\r)^{-\kappa}\l(\f{\kappa_i}{\kappa+\kappa_i}\r)^{-\kappa_i}\l(2-\delta -\f{\log r}{\log N}\r)^{-(\kappa+\kappa_i)}.
\end{equation}
We treat the sum in (\ref{partially sum me}) using partial summation, for which we require estimates for $\sum_{r\leq t}h_i(r)$. For the case $d=2$, the estimate already found in (\ref{crude estimate for gir sum}) is sufficient. Below, we find a more refined estimate which we use in the case $d=3$. 

We first consider the contribution to (\ref{partially sum me}) from squarefree values of $r$. We would like to apply \cite[Theorem A.5]{friedlander2010opera}, which states that under certain hypothesis on the function $h_i(r)$, we have
\begin{equation}\label{squarefree sum of gir}
\sum_{\ss{m \leq x\\ \mu^2(m)=1}}h_i(m)= c_{h_i}(\log x)^{\kappa_i} + O((\log x)^{\kappa_i -1}),
\end{equation}
where
$$ c_{h_i} = \f{1}{\Gamma(\kappa_i+1)}\prod_p \l(1-\f{1}{p}\r)^{\kappa_i}(1+h_i(p)).$$
In the following lemma, we verify that the function $h_i(r)$ satisfies the required hypotheses for \cite[Theorem A.5]{friedlander2010opera}.
\begin{lemma}\label{properties of gir}
For any $x\geq 1$ and any $2\leq  w< z$, the function $h_i(r)$ satisfies the following estimates.
\begin{align}
    \prod_{w\leq p <z}(1+h_i(p)) &\ll \l(\f{\log z}{\log w}\r)^{\kappa_i},\label{gir property 2}\\
    \sum_p h_i(p)^2\log p &< \infty, \label{gir property 3}\\
     \sum_{p \leq x}h_i(p)\log p &= \kappa_i\log x + O(1)\label{gir property 1}.
\end{align}
\end{lemma}
\begin{proof}

To prove (\ref{gir property 2}), we note that $1+h_i(p) = \l(1-\f{\rho_i(p)}{p^2}\r)^{-1}$ for all $p\in \mc{P}'$. The result is then immediate from Lemma \ref{sieve dims}. To prove (\ref{gir property 3}), we recall that $\rho_i(p) \ll p$, and so $h_i(p) \ll p^{-1}$. Therefore
$$\sum_{p}h_i(p)^2\log p \ll \sum_{p}\f{\log p}{p^2}< \infty.$$
Finally, we note that 
\begin{equation*}
    \sum_{p\leq x}h_i(p)\log p =\sum_{p \in \mc{P}'_{\leq x}}\f{\rho_i(p)}{p^2}\log p + O(1),
\end{equation*}
so that (\ref{gir property 1}) follows by applying Lemma \ref{useful facts about rhoi}. 
\end{proof}

We can now evaluate the sum in (\ref{partially sum me}) using partial summation. We obtain 
\begin{align*}
    \sum_{\ss{r\leq R\\ \mu^2(r) = 1}}w(r,\delta)h_i(r)&= w(R,\delta)\sum_{\ss{r\leq R\\ \mu^2(r)=1}}h_i(r)-\int_{1}^R \l(\sum_{\ss{r\leq t\\ \mu^2(r)=1}}h_i(r)\r)w'(t,\delta) \textrm{d}t\\
    &=c_{h_i}\l[(\log R)^{\kappa_i}w(R,\delta)-\int_{1}^R w'(t,\delta)(\log t)^{\kappa_i}\textrm{d}t\r]\\
    &\quad +o(1)\l[(\log R)^{\kappa_i}w(R,\delta)+\int_{1}^R w'(t,\delta)(\log t)^{\kappa_i}\textrm{d}t\r]\\
    &= c_{h_i}\kappa_i\int_{1}^R w(t,\delta)(\log t)^{\kappa_i-1}t^{-1}\textrm{d}t+o\l((\log R)^{\kappa_i}\r)\\
    &=c_{h_i}\kappa_i(\log N)^{\kappa_i}\int_{0}^{\log R/\log N}w(N^s, \delta)s^{\kappa_i -1}\textrm{d}s + o\l((\log R)^{\kappa_i}\r)\\
    &\leq (c_{h_i}\kappa_i +o(1))(\log N)^{\kappa_i}\int_{0}^{d-\beta_i}W(s)\textrm{d}s,
\end{align*}
where $W(s) = w(N^s, 0)s^{\kappa_i-1}.$

We now consider the contribution to (\ref{def gir}) from those $r$ which are not squarefree. Let $d_i$ denote the degree of $f_i$. From \cite[Lemma 3.1]{daniel1999divisor}, we have $\rho_i(p^{\alpha}) \ll p^{2\alpha(1-1/d_i)}$ for all primes $p \notin S$ and any positive integer $\alpha$. By the multiplicativity of $h_i(r)$, it follows that $h_i(r) \ll r^{-2/d_i +\eps}$. 

We recall that a positive integer $n$ is \textit{squareful} if for any prime $p\mid n$, we also have $p^2\mid n$. Since $d_i\leq 3$, we have
$$ \sum_{r \textrm{ squareful}}h_i(r) \ll \sum_{r \textrm{ squareful}}r^{-2/d_i + \eps} \leq \sum_{r \textrm{ squareful}}r^{-2/3 + \eps}< \infty,$$
where the last inequality follows from partial summation together with the fact that there are $O(M^{1/2})$ squareful positive integers less than $M$. Since $h_i(r)$ is supported on integers with no prime factors in $S$, for any $\eps>0$ there exists a set of primes $S_0$, depending only on $\eps$ and $f$, such that for any $S\supseteq S_0$, we have 
$$ \sum_{\ss{r \textrm{ squareful}\\ r>1}}h_i(r) <\eps.$$ 
For the remainder of this section, we assume that $S\supseteq S_0$. Proceeding as in \cite[Lemma 6.2]{irving2017cubic}, we use that $w(r,\delta) \ll 1$, and decompose each non-squarefree $r$ into $r=r_1r_2$, where $r_1$ is squarefree and $r_2>1$ is squareful. We have
$$\sum_{\ss{r \leq R\\ \mu^2(r)=0}}w(r,\delta)h_i(r) \ll \sum_{\ss{r_1 \leq R\\ \mu^2(r_1)=1}}h_i(r_1) \sum_{\ss{r_2 \textrm{ squareful}\\r_2>1}}h_i(r_2).$$
Combining with (\ref{squarefree sum of gir}), we deduce that for any $\eps>0$, we have
$$\sum_{\ss{r \leq R\\ \mu^2(r)=0}}w(r,\delta)h_i(r) \leq (\eps + o(1))(\log N)^{\kappa_i}.$$

In conclusion, we have the upper bound
\begin{equation}\label{final estimate for S4}
S_4^{(i)} \leq \f{(cc_iAA_ic_{h_i}\kappa_i + \eps + o(1))X}{(\log N)^{\kappa}}\int_{0}^{d-\beta_i}W(s)\textrm{d}s,
\end{equation}
where
\begin{equation}\label{definition of W(s)}
W(s)= \l(\f{\kappa}{\kappa+\kappa_i}\r)^{-\kappa}\l(\f{\kappa_i}{\kappa+\kappa_i}\r)^{-\kappa_i}(2-s)^{-(\kappa+\kappa_i)}s^{\kappa_i-1}.
\end{equation}

\begin{remark}
    Whilst $c,c_i, c_{h_i}$ all depend on $S$, the ratio $\kappa_ic_ic_{h_i}AA_i/B$ of the constants in the bounds for $S_4^{(i)}$ and $S_1$ is independent of $S$, as seen below in (\ref{cicgi}). Therefore, we may indeed take $\epsilon$ to be arbitrarily small in (\ref{final estimate for S4}) without sacrificing anything in our comparison of $S_1$ and $S_4^{(i)}$. The $o(1)$ terms in (\ref{final estimate for S4}) and in the bounds for the other sums $S_1$, $S_2^{(i)}, S_3^{(i)}$ also depend on $S$ (via the quantity $L$ in (\ref{sieves: stronger assumption on g})), but this does not matter because we may take $N$ to be sufficiently large in terms of $S$.
\end{remark}

\subsection{Proof of Theorem \ref{main sieve result}}\label{section: proof of quadratic sieve result}
We now suppose that $d=2$, i.e., $f(x,y)$ consists of irreducible factors of degree at most $2$. We first obtain a qualitative result by considering the limit as $\alpha \ra 0$. As $\alpha \ra 0$, we have $\kappa \ra 0$ and $\kappa_i \ra 1$. Therefore, we have $A_i \ra A(1)$, which is equal to $2e^{\gamma}$, where $\gamma = 0.57721...$ is the Euler--Mascheroni constant. Moreover, by \cite[Equation (11.62)]{friedlander2010opera}, we have $A(\kappa),B(\kappa) \ra 1$ as $\kappa \ra 0$. By a very similar computation to \cite[p. 248]{irving2017cubic}, we have
\begin{equation}\label{cicgi}
c_ic_{h_i} = \frac{e^{-\gamma \kappa_i}}{\Gamma(1+\kappa_i)}.
\end{equation}
Therefore, the ratio of the constants in the bounds for $S_1$ and $S_4^{(i)}$ is $\kappa_ic_icc_{h_i}AA_i/cB$, which is independent of $S$ and tends to zero as $\alpha \ra 0$. Also,
$$ \lim_{\alpha \ra 0}W(s) = (2-s)^{-1}.$$
Therefore, 
\begin{equation}\label{qualitative S4}
\lim_{\alpha \ra 0}S_4^{(i)} \ll \f{(B+o(1))X}{(\log N)^{\kappa}}\l(\log 2-\log \beta_i\r).
\end{equation}
For all $\eps>0$, by choosing $\beta_i$ sufficiently close to $2$, we have the bound
\begin{equation}
\lim_{\alpha \ra 0}S_4^{(i)} \leq \f{(\eps B+o(1)) X}{(\log N)^{\kappa}}\ll \eps S_1.
\end{equation}
We recall from Sections \ref{section: the sums s2i} and \ref{section: the sums s3i} that $S_2^{(i)}$ and $S_3^{(i)}$ are also negligible compared to $S_1$ as $\alpha \ra 0$. Therefore, we see that $S(\mc{A},\mc{P},x)>0$ for sufficiently small $\alpha$ and sufficiently large $N$. 

To obtain the best quantitative bounds, the choices of $\beta_i \in (1,2)$ should be optimised so as to minimise $S_3^{(i)}+S_4^{(i)}$. However, in practice, numerical computations suggest that the optimal choices for $\beta_i$ are extremely close to $2$, and little is lost in taking them arbitrarily close to 2, as above. In this case, the contributions from the sums $S_4^{(i)}$ are negligible. Taking a sum over $i$ of the estimates from (\ref{sum s3i}) and combining with (\ref{s1}), for any $\eps>0$ we have
$$S(\mc{A},\mc{P},x) \geq 
\l(\f{(c-\eps+o(1))X}{(\log N)^{\kappa}}\r)\l(B-A\alpha \sum_{i=m+1}^k \theta_i \int_{1}^{2}(2-s)^{-\kappa}\f{\textrm{d}s}{s}\r).$$
 Let $r(\kappa) = B/A$. Then we have established that $S(\mc{A}, \mc{P}, x) >0$ provided that 
\begin{equation}\label{maximising kappa}
r(\kappa) - \alpha\sum_{i=m+1}^k \theta_i\int_{1}^{2}(2-s)^{-\kappa}\f{\textrm{d}s}{s} >0.
\end{equation}
We recall that $\theta = \theta_0 + \cdots +\theta_k$ and $\kappa = \alpha\theta$. Therefore, we may replace $\alpha\sum_{i=m+1}^k\theta_i$ with the trivial upper bound $\kappa$, after which we find by numerical computations (see Section \ref{section:Details of the numerical computations} for details) that the largest value of $\kappa$ we can take in (\ref{maximising kappa}) is $\kappa = 0.39000...$. Thus the condition $\alpha\theta \leq 0.39000$ is enough to ensure that $S(\mc{A}, \mc{P},x) >0$ for sufficiently large $N$. This completes the proof of Theorem \ref{main sieve result}. 

\subsection{Proof of Theorem \ref{main sieve result for cubic factors}}\label{proof of cubic sieve result}
We now discuss the case $d=3$, where $f(x,y)$ may contain irreducible factors of degree up to $3$. We recall in the case $d=2$, the sums $S_4^{(i)}$ could be made negligible compared to $S_1$ by choosing $\beta_i$ arbitrarily close to $2$, due to the factor $\log 2- \log \beta_i$ appearing in (\ref{qualitative S4}). When $d=3$, we obtain the same bound as in (\ref{qualitative S4}), but with $\log 2-\log \beta_i$ replaced with $\log 2 - \log(\beta_i-1)$. Consequently, $S_4^{(i)}$ is no longer negligible, even when $\alpha \ra 0$ and $\beta_i$ is arbitrarily close to $2$. In fact, it can be checked that its limit as $\alpha \ra 0$ and $\beta_i \ra 2$ is larger than $S_1$, and so the above methods break down in this case.

However, we recall the additional hypothesis in Theorem \ref{main sieve result for cubic factors} that Assumption \ref{assumtion on special P} holds. By enlarging $\mc{P}$ if necessary, we may assume that equality holds in (\ref{assumption equation thingy}), namely
\[
\mc{P}\bs S = \bigcup_{j=1}^n \{p \equiv t_i \Mod{q_i}\}.
\]
By Dirichlet's theorem on primes in arithmetic progressions, (\ref{density 1}) holds with $\alpha \leq \sum_{j=1}^n \f{1}{q_j-1} \leq 0.32380$. Moreover, it follows from Lemma \ref{Chebotarev for non-Galois extensions} (a version of the Chebotarev density theorem) that (\ref{density 3}) holds for some $\theta_i \leq 3$. We now explain why this choice of $\mc{P}$ is easier to handle than arbitrary choices of $\mc{P}$ of the same density $\alpha$. 

When applying the switching principle for $S_4^{(i)}$, we wrote $f_i(a,b) = pr$ for $p \in \mc{P}$. Since we may assume $q_1, \ldots, q_n \in S_0$, the congruence condition $C(a,b)$ forces $f_i(a,b)$ to lie in a particular congruence class modulo $q_1, \ldots, q_n$. Combined with the fact that $p \equiv t_j \Mod{q_j}$ for some $j$, we see that $r$ lies in a particular congruence class modulo $q_j$ for some $j\in \{1, \ldots, n\}$, and this congruence class depends only on $t_j, C(a,b)$ and $f$. Moreover, by (\ref{l-adic valuations}), we have that $\gcd(r, \Delta) = \gcd(f_i(a_0,b_0),\Delta)$ depends only on $C(a,b)$ and $f$. We deduce that there exist $r_1,\ldots, r_n$, depending only on $t_1, \ldots, t_n, C(a,b)$ and $f$, such that $r':= |r|/\gcd(r,\Delta)$ lies in the set
\[
\mc{T} := \{r' \in \N: r'\equiv r_j \Mod{q_j} \textrm{ for some }j\in\{1, \ldots, n\}\}.
\]
Adding this condition into (\ref{ready to optimise etas}) and (\ref{partially sum me}) and changing notation from $r'$ back to $r$, we obtain
\begin{equation}\label{s4i in Irving's setup}
S_4^{(i)}\leq \frac{(cc_iAA_i+o(1))X}{(\log N)^{\kappa_i+\kappa}}\sum_{r \in \mc{T}_{\leq x}}w(r,\delta)h_i(r).
\end{equation}
As demonstrated by Irving \cite[Lemma 6.1]{irving2017cubic}, the argument based on \cite[Theorem A.5]{friedlander2010opera} we used to obtain (\ref{squarefree sum of gir}) can be generalised to give 

\begin{equation}\label{Irving's squarefree sum of gir}
\sum_{\ss{r \in \mc{T}_{\leq x}\\ \mu^2(r)=1}}h_i(r) \leq \sum_{j=1}^n \sum_{\ss{r \leq x\\ r\equiv r_j \Mod{q_j}}}h_i(r) = c_{h_i}(\log x)^{\kappa_i}\sum_{j=1}^n \f{1}{q_j-1} + O((\log x)^{\kappa_i-1}).
\end{equation}
Proceeding as before, we deduce the same estimate for $S_4^{(i)}$ as in (\ref{final estimate for S4}), but with an additional factor $\alpha_0 := \sum_{j=1}^n \f{1}{q_j-1}$. It is now clear qualitatively that for sufficiently large $q_1, \ldots, q_n$ (i.e., as $\alpha_0 \ra 0$), the sums $S_4^{(i)}$ are once again negligible compared to $S_1$. 

We now make the above discussion more quantitative in order to complete the proof of Theorem \ref{main sieve result for cubic factors}. Combining everything, we see that $S(\mc{A},\mc{P},x)>0$ for sufficiently large $N$ provided that 
\begin{equation}\label{cubic saving q-1}r(\kappa)-\sum_{i\,:\, \deg f_i =2}\alpha \theta_i \int_{1}^2(2-s)^{-\kappa}\f{\textrm{d}s}{s}-
\sum_{\ss{i \,:\, \deg f_i =3}}\alpha_0 A_i\kappa_ic_ic_{h_i}\int_{0}^1 W(s) \textrm{d}s >0.
\end{equation}

We denote the left hand side of (\ref{cubic saving q-1}) by $F(\boldsymbol{\theta})$. Recalling (\ref{definition of W(s)}) and (\ref{cicgi}), we have
\begin{equation}\label{annoying numeric optimisation}
A_i\kappa_ic_ic_{h_i}\int_{0}^1 W(s) \textrm{d}s = \f{A_i\kappa_ie^{-\gamma \kappa_i}\kappa^{-\kappa}\kappa_i^{-\kappa_i}(\kappa+\kappa_i)^{\kappa+\kappa_i}}{\Gamma(1+\kappa_i)}\int_0^1 (2-s)^{-(\kappa+\kappa_i)}s^{\kappa_i-1}\textrm{d}s.
\end{equation}

For a fixed choice of $\kappa<1/2$, the integrand is a decreasing function of $\kappa_i$, because $0\leq \f{s}{(2-s)}\leq 1$ for any $s\in [0,1]$. The functions $\Gamma(1+\kappa_i)^{-1}$, $\kappa_i^{-\kappa_i}$ and $e^{-\gamma\kappa_i}$ are also decreasing in $\kappa_i$ in the range $\kappa_i\in (1/2,1)$. Let $t = \alpha_0 \deg f$. Since $\kappa_i \geq 1-\kappa \geq 1-t$, we therefore obtain an upper bound by replacing $\kappa_i$ with $1-t$ in all these terms. The remaining terms in (\ref{annoying numeric optimisation}) are all increasing in $\kappa_i$, and we apply the trivial bound $\kappa_i \leq 1$. Finally, we note that $\kappa^{-\kappa}(\kappa+\kappa_i)^{\kappa+\kappa_i}$ is an increasing function in $\kappa$ for $\kappa <1/2$, and so we may replace $\kappa$ by $t$ in this expression. Therefore, (\ref{annoying numeric optimisation}) can be bounded by 

\begin{equation}\label{annoying numeric optimisation 2}
     \f{A(1)e^{\gamma(t-1)}t^{-t}(t+1)^{t+1}}{\Gamma(2-t)}\int_{0}^1 (2-s)^{-1}s^{-\kappa}\textrm{d}s.
\end{equation}
We denote the factor outside the integral in (\ref{annoying numeric optimisation 2}) by $H(t)$. The integral in (\ref{annoying numeric optimisation 2}) is equal to the first integral in (\ref{cubic saving q-1}), as can be seen by making the substitution $u=2-s$. Therefore, 
\begin{align*}
    F(\boldsymbol{\theta})&\geq r(\kappa)-\l(\sum_{i\,:\,\deg f_i =2}\alpha\theta_i + \sum_{i\,:\, \deg f_i = 3} \alpha_0 H(t)\r) \int_0^1 (2-s)^{-1}s^{-\kappa}\textrm{d}s.
\end{align*}
Recalling that $A(1)=2e^{\gamma}$, it can be checked that $H(t)>4$ for all $t>0.2$, and so in particular $H(t)>2\theta_i$ whenever $\deg f_i =2$. Moreover, $r$ is a decreasing function of $\kappa$, so $r(\kappa)\geq r(t)$. We obtain
\begin{align}   
F(\boldsymbol{\theta})&\geq r(t)-H(t)\int_{0}^1 (2-s)^{-1}s^{-t}\textrm{d}s\l(\sum_{i\,:\,\deg f_i=2}\alpha/2 +\sum_{\ss{i:\deg f_i =3}}\alpha_0\r)\nonumber\\
    &\geq r(t)-\f{tH(t)}{3}\int_{0}^1 (2-s)^{-1}s^{-t}\textrm{d}s\label{maximising t}
\end{align}
for any $t>0.2$. We find by numerical computations that the above expression is positive provided that $t\leq 0.32380.$. Since $t=  \alpha_0\deg f$ this is implied by the condition (\ref{assumption equation thingy}) assumed in Theorem \ref{main sieve result for cubic factors}. 

\subsection{Details of the numerical computations}\label{section:Details of the numerical computations}
We now explain how we obtained numerical values with guaranteed error bounds in the proofs of Theorem \ref{main sieve result} and Theorem \ref{main sieve result for cubic factors}. From \cite[Equations (11.62), (11.63)]{friedlander2010opera}, we have for $0< \kappa < 1/2$ that 
\[
A(\kappa) = \f{2e^{\gamma\kappa}}{\Gamma(1-\kappa)}\f{q(\kappa)}{p(\kappa)+ q(\kappa)}, \qquad B(\kappa) = \f{2e^{\gamma \kappa}}{p(\kappa) + q(\kappa)},
\]
where
\begin{align*}
p(\kappa) &= \int_{0}^{\infty}z^{-\kappa}\exp(-z+\kappa\op{Ei}(-z))\op{d}z \\
q(\kappa) &=\int_{0}^{\infty}z^{-\kappa}\exp(-z-\kappa\op{Ei}(-z))\op{d}z 
\end{align*}
and where $\op{Ei}(-z) = -\int_{z}^{\infty}\f{e^{-u}}{u}\op{d}u$ is the standard exponential integral. For any $0<\kappa < 1/2$, the integrand defining $p(\kappa)$ is monotonically decreasing, decays exponentially as $z \ra \infty$, and is bounded above by $2$. Similarly, the integrand defining $q(\kappa)$ is monotonically decreasing and decays exponentially as $z \ra \infty$, but diverges as $z \ra 0$. To analyse its behaviour near $0$, we note that for $\delta \in (0,1)$ and $\kappa < 0.4$, we have
\begin{align*}
    &\int_{0}^{\delta}z^{-\kappa}e^{-z}\exp(-\kappa\op{Ei}(-z))\op{d}z\\
    &\leq \int_{0}^{\delta}z^{-0.4}\op{exp}\l(0.4\int_{z}^{\infty}\f{e^{-u}}{u}\op{d}u\r)\op{d}z\\
    &\leq \int_{0}^{\delta}z^{-0.4} \op{exp}\l(0.4\int_{z}^{1}\f{\op{d}u}{u} + 0.4\int_{1}^{\infty}e^{-u}\op{d}u\r)\op{d}z\\
    &\leq \int_{0}^{\delta}z^{-0.4}\exp(-0.4\log(z) + 0.4e^{-1})\op{d}z\\
    &\leq 2 \int_{0}^{\delta}z^{-0.8}\op{d}z\\
    &\leq 10\delta^{0.2}.
\end{align*}
Therefore, we can obtain estimates for $p(\kappa), q(\kappa)$ (and hence also $A(\kappa), B(\kappa)$ with provable error bounds using standard numerical integration methods such as the trapezoid rule. The functions from (\ref{maximising kappa}) and (\ref{maximising t}) are bounded, continuous, and monotonically decreasing in $\kappa$, and so it is straightforward to find their roots to the desired precision. This was implemented in Sage, which produced the values $0.39000, 0.32380$ in $12.3$ seconds on a standard laptop with $2$ cores.

\section{Application to the Hasse principle}\label{section: Application to HP}

In this section, we apply the main sieve results (Theorem \ref{main sieve result} and Theorem \ref{main sieve result for cubic factors}) obtained in Section \ref{section: application of the beta sieve} in order to prove Theorem \ref{main HP result} and Theorem \ref{HP with cubic factors}. 

\subsection{Algebraic reduction of the problem}Let $K$ be a number field of degree $n$ satisfying the Hasse norm principle. In \cite[Lemma 2.6]{irving2017cubic}, Irving turns the problem of establishing the Hasse principle for 
\begin{equation}\label{HP for general products of quadratics}
    f(t) = \bN(x_1, \ldots, x_n) \neq 0
\end{equation}
into a sieve problem. Irving assumes that $f(t) \in \Z[t]$ is an irreducible cubic polynomial. However, in the following result, we demonstrate that Irving's strategy can be applied to establish a similar result for an arbitrary polynomial $f\in \Z[t]$. We recall that $f(x,y)$ denotes the homogenisation of $f$. Throughout this section, we make the choice 
\begin{equation}\label{specific choice of sifting set p}
    \mc{P} = \{p \notin S: \textrm{ the inertia degrees of }p\textrm{ in }K/\Q\textrm{ are not coprime}\}
\end{equation}
for a finite set of primes $S$ containing all ramified primes in $K/\Q$. 
\begin{proposition}\label{algebraic reduction}
Suppose that (\ref{HP for general products of quadratics}) has solutions over $\Q_p$ for every $p$ and over $\R$. Let $\mc{P}$ and $S$ be as in (\ref{specific choice of sifting set p}). Then there exists $\Delta \in \N$, divisible only by primes in $S$, integers $a_0,b_0$, and real numbers $r,\xi >0$ such that the following implication holds:

Suppose that $a,b$ are integers for which 
\begin{enumerate}
    \item $a \equiv a_0 \Mod{\Delta}$ and $b \equiv b_0 \Mod{\Delta}$,
    \item $|a/b - r|<\xi$,
    \item $bf(a,b)$ has no prime factors in $\mc{P}$.
\end{enumerate}
Then $(\ref{HP for general products of quadratics})$ has a solution over $\Q$.
\end{proposition}

By multiplicativity of norms, it suffices to find integers $a,b$ such that $b$ and $f(a,b)$ are in $N_{K/\Q}(K^*)$. Since $K$ satisfies the Hasse norm principle, we have $\Q^*\cap N_{K/\Q}(I_K) = N_{K/\Q}(K^*)$, where $I_K$ denotes the ideles of $K$. Consequently, to show that $c\in \Q^*$ is a norm from $K$, it suffices to find elements $x_v \in K_v^{*}$ for each place $v$ of $K$, such that 
\begin{enumerate}
    \item For all but finitely many places $v$ of $K$, we have $x_v\in \OO_v^*$ (this ensures that $(x_v) \in I_K$).
    \item For all places $w$ of $\Q$, we have 
    \begin{equation}\label{idelic condition}
    \prod_{v\mid w}N_{K_v/\Q_w}(x_v)=c.
    \end{equation}
\end{enumerate}
The arguments in \cite[Lemma 2.2, Lemma 2.3, Lemma 2.4]{irving2017cubic} go through without changes. We summarise them below.
\begin{lemma}\label{bundle of Irving lemmas} Suppose $c\neq 0$ is an integer. Then
\begin{enumerate} 
\item If $p\nmid c$ and $K/\Q$ is unramified at $p$, then there exist $x_v \in \OO_v^*$ for each place $v$ of $K$ above $p$, such that $\prod_{v\mid p}N_{K_v/\Q_p}(x_v)=c$.  
\item Suppose that $K/\Q$ is unramified at $p$, and that the inertia degrees above $p$ in K are coprime. Then there exist $x_v \in K_v^*$ for each place $v$ of $K$ above $p$, such that $\prod_{v\mid p}N_{K_v/\Q_p}(x_v)=c$.  
\item Let $p$ be a place of $\Q$. Suppose that there exists $x_1, \ldots, x_n \in \Q_p$ such that $c =\bN(x_1, \ldots ,x_n)$. Then there exists $x_v \in K_v^*$ for each place $v$ of $K$ above $p$, such that $\prod_{v\mid p}N_{K_v/\Q_p}(x_v)=c$.  
\end{enumerate}
\end{lemma}
We now give a slight generalisation of \cite[Lemma 2.5]{irving2017cubic} to the case of an arbitrary polynomial $f$. 
\begin{lemma}\label{Irving 2.5}
Let $p$ be a prime for which $f(t) = \bN(x_1, \ldots, x_n) \neq 0$ has a solution over $\Q_p$. Then there exists $a_0,b_0 \in \Z$ and $l \in \N$ (all depending on $p$) such that for any $a,b \in \Z$ satisfying $a\equiv a_0 \Mod{p^l}, b\equiv b_0 \Mod{p^l}$, we have $b,f(a,b) \in \bN(\Q_p^n)\bs\{0\}$. 
\end{lemma}

\begin{proof}
We define $N = \bN(\Q_p^n)\bs\{0\}$. Let $t_1 \in \Q_p$ be such that $f(t_1) = \bN(x_1, \ldots, x_n) \neq 0$ has a solution over $\Q_p$. Choose $a_1, b_1 \in \Z_p$ such that $\nu_p(b_1)$ is a multiple of $n$, and $a_1/b_1 = t_1$. Then $b_1 \in N$ and $f(a_1,b_1) \in N$. 

The set $N \subseteq \Q_p$ is open, and so $N\times N \subseteq \Q_p^2$ is open. Moreover, the map $\phi: \Q_p^2 \ra \Q_p^2$ sending $(a,b)$ to $(f(a,b),b)$ is continuous in the $p$-adic topology. Therefore, the set $\phi^{-1}(N\times 
N)$ is open, and contains the element $(a_1,b_1)$. Hence there is a small $p$-adic ball with centre $(a_1,b_1)$, all of whose elements $(a,b)$ satisfy $b, f(a,b) \in N$. This ball can be described by congruence conditions $a\equiv a_0 \Mod{p^l}, b\equiv b_0 \Mod{p^l}$ for a sufficiently large integer $l$, as claimed in the lemma.
\end{proof}

\begin{proof}[Proof of Proposition \ref{algebraic reduction}]
The proof closely follows the argument in \cite[Lemma 2.6]{irving2017cubic}. By the Hasse norm principle, to find solutions to (\ref{HP for general products of quadratics}) it suffices to find integers $a,b$ such that properties (1) and (2) stated before Lemma \ref{bundle of Irving lemmas} hold with $c=b$ and $c=f(a,b)$. We divide the places of $\Q$ into four sets:
\begin{enumerate}
    \item $p\in S$. Here, Lemma \ref{Irving 2.5} gives congruence conditions $a \equiv a_{0,p} \Mod{p^l}$, $b \equiv b_{0,p} \Mod{p^l}$ which ensure that $b, f(a,b) \in N(\Q_p^n)\bs\{0\}$. By part (3) of Lemma \ref{bundle of Irving lemmas}, we deduce that property (2) stated before Lemma \ref{bundle of Irving lemmas} holds with $c=b$ and $c=f(a,b)$. The congruence conditions at each prime $p \in S$ can be merged into one congruence condition $a \equiv a_0 \Mod{\Delta}, b \equiv b_0 \Mod{\Delta}$ using the Chinese remainder theorem. 
    \item $p\notin S$ and $p\notin \mc{P}$. If $p \nmid b$, we apply part (1) of Lemma \ref{bundle of Irving lemmas} with $c=b$. If $p\mid b$ we apply part (2) of Lemma \ref{bundle of Irving lemmas} with $c=b$. The same argument works for $f(a,b)$ by choosing $c=f(a,b)$. 
    \item $p \in \mc{P}$. Since we are assuming that $bf(a,b)$ has no prime factors in $\mc{P}$, for these primes, part (1) of Lemma \ref{bundle of Irving lemmas} applies with $c=b$ and $c=f(a,b)$. 
    \item $p=\infty$. We follow a similar argument to Lemma \ref{Irving 2.5}. We may assume that (\ref{HP for general products of quadratics}) is everywhere locally soluble, so in particular there exists $r \in \R$ such that $f(r) \in \bN(\R^n)\bs\{0\}$. Since $f$ is continuous and $\bN(\R^n)\bs\{0\}$ is open in the Euclidean topology, we can find $\xi >0$ such that $f(t) \in \bN(\R^n)\bs\{0\}$ whenever $|t-r|<\xi$. Clearly solubility of $f(t) = \bN(\bx) \neq 0$ and $f(-t) = \bN(\bx) \neq 0$ over $\Q$ are equivalent; consequently, we may assume $r>0$. Suppose in addition that $t\in \Q$, and write $t=a/b$ for $a,b \in \N$. Since $b$ is positive, it is automatically in $\bN(\R^n)\bs\{0\}$. By multiplicativity of norms, we conclude that $b,f(a,b) \in \bN(\R^n)\bs\{0\}$. The condition (\ref{idelic condition}) now follows from part (3) of Lemma \ref{bundle of Irving lemmas}. \qedhere  
\end{enumerate}
\end{proof}

\begin{example}\label{iskovskikh example} Let $f(t) = (t^2-2)(-t^2+3)$ and $K = \Q(i)$, so that 
\begin{align*}
    f(x,y) &= (x^2-2y^2)(-x^2+3y^2)\\
    N_{K/\Q}(u,v) &= u^2+v^2.
\end{align*}
It is known that there is a Brauer--Manin obstruction to the Hasse principle for the equation $(t^2-2)(-t^2+3) = u^2+v^2 \neq 0$ by the work of Iskovskikh \cite{iskovskikh1971counterexample}. However, Proposition \ref{algebraic reduction} still applies. 

When $p\equiv 1 \Mod{4}$, the prime $p$ splits in $K/\Q$, and so the inertia degrees of $p$ in $K/\Q$ are coprime. On the other hand, when $p \equiv 3 \Mod{4}$, the prime $p$ is inert in $K/\Q$ and has degree 2, so the inertia degrees are not coprime. Therefore, we have $\mc{P} = \{p: p\equiv 3 \Mod{4}\}$. We choose $S=\{2\}$. In this example, it can be checked that the congruence conditions $a \equiv 8 \Mod{16}$ and $b \equiv 1 \Mod{16}$ are sufficient. Finally, for the infinite place, we just have the condition $f(a,b) >0$. The sieve problem we obtain is to find integers $a,b$ such that
\begin{enumerate}
    \item $a \equiv 8 \Mod{16}, b \equiv 1 \Mod{16}$,
    \item $f(a,b)>0$,
     \item $f(a,b)$ has no prime factors $p \equiv 3 \Mod{4}$.
\end{enumerate}
We remark that $f(a/b) = b^{-4}f(a,b)$, and since $2=[K:\Q]$ divides $4$, $b^{-4}$ is automatically a norm from $K$. This explains why in (3) above, we can consider prime factors of $f(a,b)$ rather than $bf(a,b)$.  

An integer is the sum of two squares if and only if it is non-negative and all prime factors $p \equiv 3 \Mod{4}$ occur with an even exponent. The above conditions are stronger than this, so the algebraic reduction performed in Proposition \ref{algebraic reduction} is consistent with what we already knew for this example. 

Since the Hasse principle fails for this example, we know that the above sieve problem cannot have a solution. This is indeed the case, since condition (1) implies that $-a^2+3b^2 \equiv 3 \Mod{4}$, and so $f(a,b)$ must contain a prime factor $p \equiv 3 \Mod{4}$. 

The fact that the above sieve problem has no solutions also does not contradict Theorem \ref{main sieve result}. For a prime $p>3$, we have that $\nu_1(p)$ is 2 if $p \equiv \pm 1 \Mod{8}$ and zero otherwise, and $\nu_2(p)$ is 2 if $p \equiv \pm 1\Mod{12}$ and zero otherwise. Consequently, by Dirichlet's theorem on primes in arithmetic progressions, for $i\in\{1,2\}$, asymptotically one half of the primes $p\in \mc{P}$ have $\nu_i(p) = 2$ and half have $\nu_i(p) = 0$. We conclude that $\theta_1 = \theta_2 = 1$, so that $\theta = 2$. Theorem \ref{main sieve result} therefore requires the density of $\mc{P}$ to be less than $0.39000.../\theta =0.19503...$. However, the density of $\mc{P}$ here is $1/2$, and so Theorem \ref{main sieve result} does not apply to this example. 
\end{example}

\subsection{The Chebotarev density theorem}
In order to apply the sieve results from Section \ref{section: main sieve results for binary forms} to prove Theorem \ref{main HP result} and Theorem \ref{HP with cubic factors}, we need to compute the densities (\ref{density 1}) and (\ref{density 3}) for the choice of $\mc{P}$ given in (\ref{specific choice of sifting set p}). For more complicated number fields $K$, such as when $K/\Q$ is non-abelian, it is not possible to write down the set $\mc{P}$ explicitly in terms of congruence conditions. However, we can still compute the densities (\ref{density 1}) and (\ref{density 3}) by appealing to the \textit{Chebotarev density theorem}, which can be viewed as a vast generalisation of Dirichlet's theorem on primes in arithmetic progressions. 

Let $K/\Q$ be a number field of degree $n$. Let $p$ be a prime, unramified in $K/\Q$. We can factorise the ideal $(p)$ as $(p) = \mathfrak{p}_1\cdots \mathfrak{p}_r$, where $\mf{p}_1, \ldots, \mf{p}_r$ are distinct prime ideals in $\OOK$. The \textit{splitting type} of $p$ in $K/\Q$ is the partition $(a_1, \ldots ,a_r)$ of $n$, where $a_i$ is the inertia degree of $\mf{p}_i$, i.e., $N(\mf{p}_i) = p^{a_i}$. Equivalently, the splitting type is the list of degrees of the irreducible factors of the minimum polynomial of $K/\Q$, when factorised modulo $p$.

Suppose first that $K/\Q$ is Galois, with Galois group $G$. Then $G$ acts transitively on $\{\mf{p}_1, \ldots, \mf{p}_r\}$. Fix $i \in \{1, \ldots, r\}$. The \textit{Decomposition} group $D_{\mf{p}_i}$ is the stabilizer of $\mf{p}_i$ under this action. Note that $D_{\mf{p}_i}$ is cyclic, and there is an isomorphism 
$$\psi_i : D_{\mf{p}_i} \ra \Gal((\OOK/\mf{p}_i)/(\Z/(p))).$$
The group $\Gal((\OOK/\mf{p}_i)/(\Z/(p)))$ is generated by the $\textit{Frobenius element}$ defined by $x\mapsto x^p$, which has order $a_i$. Let $\sigma_i$ denote the preimage of the Frobenius element under $\psi_i$. We define the \textit{Artin symbol}
$$ \left[\f{K/\Q}{p}\right] = \{\sigma_1, \ldots, \sigma_r\}.$$
The Artin symbol is a conjugacy class of $G$. Indeed, all the $\mf{p}_i$'s lie in the same orbit of $G$ (there is only one orbit as $G$ acts transitively). Stabilisers of points in the same orbit of an action are conjugate, and so all the $D_{\mf{p}_i}$ are conjugate. 

We now come to the statement of the Chebotarev density theorem. For a conjugacy class $C$ of $G$, we let $\pi_C(x)$ denote the number of primes $p\leq x$ whose Artin symbol is equal to $C$. We define the (natural) \textit{density} of a set of primes $\mc{P}$ to be 
$$\lim_{x \ra \infty}\l( \f{\#\{p \in \mc{P}: p \leq x\}}{\pi(x)} \r),$$
if such a limit exists. 

\begin{theorem}[Chebotarev density theorem]\label{Chebotarev for Galois extensions} 
Let $C$ be a conjugacy class of $G$. The density of primes $p$ for which the Artin symbol is equal to $C$ is $\#C/\#G$. 
\end{theorem}

In order to obtain the explicit error terms from (\ref{density 1}) and (\ref{density 3}), we need an effective version of the Chebotarev density theorem. The following lemma is a straightforward consequence of a more refined result due to Lagarias and Odlyzko \cite[Theorem 1]{lagarias1977effective}.

\begin{theorem}[Effective Chebotarev density theorem]\label{effective galois chebotarev}
For any $A\geq 1$, we have
$$\pi_C(x) =\pi(x)\l(\f{\#C}{\#G} + O_A((\log x)^{-A})\r),$$
where the implied constant may depend on $K,C$ and $A$.
\end{theorem}

We now consider the non-Galois case. As above, let $(p) = \mf{p}_1\cdots \mf{p}_r$ be the factorisation of $(p)$ in $\OOK$. Let $\widehat{K}$ denote the Galois closure of $K$, and let $G= \Gal(\widehat{K}/\Q)$. This time, there is no action of $G$ on $\{\mf{p}_1, \ldots, \mf{p}_r\}$, because the $\mf{p}_i$'s could split further in $\widehat{K}$, and elements of $G$ could permute prime factors which are above different $\mf{p}_i$'s.

To get around this, we define $H=\Gal(\widehat{K}/K)$, and instead consider the action of $G$ on the set $X$ of left cosets of $H$ in $G$. For an element $\sigma\in G$, the cyclic group $\langle \sigma \rangle$ generated by $\sigma$ acts by left multiplication on $X$. The sizes of the orbits of this action form a partition of $[G:H] =[K:\Q] = n$. Moreover, it can be checked that conjugate elements of $G$ give the same orbit sizes, so we can associate a single partition of $n$ with each Artin symbol $\l[\f{\hat{K}/\Q}{p}\r]$. The following fact relating this partition with the splitting type of $p$ can be found in \cite[Ch. 3, Proposition 2.8]{janusz1996algebraic}.

\begin{lemma}\label{non Galois Chebotarev}
Let $\sigma \in \left[\f{\widehat{K}/\Q}{p}\right]$. Then $p$ has splitting type $(a_1,\ldots, a_r)$ in $K/\Q$ if and only if the action of $\langle \sigma \rangle$ on $X$ has orbit sizes $(a_1, \ldots, a_r)$. 
\end{lemma}

Let $K=\Q(\alpha)$, and let $g$ be the minimum polynomial of $\alpha$. We can also view $\langle\sigma\rangle$ as acting on the set of $n$ roots of $g$ in $\hat{K}$.  By definition of $H$, we have that $\sigma \alpha = \sigma'\alpha$ if and only if $\sigma H = \sigma' H$. It follows that the orbit sizes of $\langle \sigma \rangle$ acting on $X$ are the same as the orbit sizes of $\langle \sigma \rangle$ acting on the roots of $g$, which in turn are the cycle lengths of $\sigma$ viewed as a permutation on the $n$ roots of $g$ in $\hat{K}$. The set of $\sigma \in G$ with cycle lengths $(a_1, \ldots, a_r)$ is a union $\bigcup_{i=1}^s C_i$ of conjugacy classes $C_i$. We may now apply Theorem \ref{effective galois chebotarev} to each of these conjugacy classes separately. Putting everything together, we have the following result on densities of splitting types in non-Galois extensions.

\begin{lemma}\label{Chebotarev for non-Galois extensions}
Let $K$ be a number field of degree $n$ over $\Q$, and let $\hat{K}$ denote its Galois closure. Let $G= \Gal(\hat{K}/\Q)$, viewed as a permutation group on the $n$ roots of the minimum polynomial of $K$ in $\hat{K}$. For a partition $\ba = (a_1, \ldots, a_r)$ of $n$, let $\mc{P}(\ba)$ denote the set of primes with splitting type $\ba$ in $K/\Q$, and let $T(\ba)$ denote the proportion of elements of $G$ with cycle shape $\ba$. Then for any $A\geq 1$, 
$$\#\{p \in \mc{P}(\ba): p \leq x\}= \pi(x)\l(T(\ba) + O_A((\log x)^{-A})\r).$$
\end{lemma}

\subsection{Proof of Theorem \ref{main HP result} and Theorem \ref{HP with cubic factors}}

\begin{proof}[Proof of Theorem \ref{main HP result}]
Since $K$ satisfies the Hasse norm principle, we may apply the algebraic reduction from Proposition \ref{algebraic reduction}. Therefore, it suffices to show that conditions (1)--(3) from Proposition $\ref{algebraic reduction}$ hold for $\mc{P}$ as defined in (\ref{specific choice of sifting set p}), and for some choice of $S$ containing all the ramified primes in $K/\Q$. Let $F(x,y)$ be the binary form obtained from $yf(x,y)$ after removing any repeated factors. Clearly to prove $bf(a,b)$ is free from prime factors in $\mc{P}$, it suffices to prove $F(a,b)$ is, and so we may replace $bf(a,b)$ with $F(a,b)$ in Condition (3) of Proposition \ref{algebraic reduction}. We apply Theorem \ref{main sieve result} to the binary form $F(x,y)$, which has nonzero discriminant. We choose $S$ to be the union of the ramified primes in $K/\Q$ and the set $S_0$ from Theorem \ref{main sieve result}, and we make the choice $\mc{R}=\mc{B}N$, where $\mc{B}$ takes the form (\ref{choice of region}) for the parameters $r,\xi>0$ coming from the application of Proposition \ref{algebraic reduction}.

It remains only to check that (\ref{density 1}) and (\ref{density 3}) hold with $\alpha\theta \leq  0.39000$. By Lemma \ref{Chebotarev for non-Galois extensions}, (\ref{density 1}) holds with $\alpha = T(G)$. We now compute the quantity $\theta$. We claim that $\theta \leq k+j+1$. We recall that $\theta$ is a sum over the quantities $\theta_i$ associated to each irreducible factor of $f$, plus an additional term $\theta_0=1$ coming from the homogenising factor $f_0(x,y) = y$. We may therefore reduce to the case where $f$ is itself irreducible of degree at most $2$, with the goal of proving that
$$ \theta \leq 
\begin{cases}
3, & \textrm {if }f\textrm{ is quadratic and }L\subseteq \hat{K},\\
2, & \textrm{if }f\textrm{ otherwise},
\end{cases}
$$ 
where $L$ denotes the number field generated by $f$. 

If $f$ is linear, then $\nu_f(p)=1$ for all $p\notin S$, and so $\theta = \theta_0+ 1 = 2$, as required. We now consider the case where $f$ is an irreducible quadratic.  Let 
$$\nu_f(p) = \#\{t \Mod{p}: f(t) \equiv 0 \Mod{p}\}.$$
If $L\subseteq \hat{K}$, then Lemma \ref{Chebotarev for non-Galois extensions} could be applied to compute $\theta$, with the desired error terms from (\ref{density 2}). However, we apply the trivial bound $\nu_f(p)\leq 2$ for $p \notin S$, since it is not possible to improve on this in general. We therefore obtain $\theta \leq \theta_0+ 2 = 1+2 = 3$, which is satisfactory. 

We now assume that $L\not\subseteq \hat{K}$. We want to show that $\theta =2$, or equivalently that $\nu_f(p)=1$ on average over $p \in \mc{P}$. Let $M = \hat{K}L$ be the compositum of $\hat{K}$ and $L$. Since $\hat{K}\cap L =\Q$, we have by \cite[Ch. VI, Theorem 1.14]
{serge2002algebra} that $M/\Q$ is Galois, and
$$\Gal(M/\Q) \isom \Gal(L/\Q) \times \Gal(\hat{K}/\Q) \isom \Z/2\Z \times \Gal(\hat{K}/\Q).$$ 
We have $\nu_f(p) = 2$ if $p$ is split in $L$, and $\nu_f(p) = 0$ if $p$ is inert in $L$, and so
\begin{equation}\label{computing theta}
\theta = 1+ 2\lim_{x \ra \infty}\l(\f{\#\{p\leq x: p \in \mc{P}, \,p \textrm{ split in }L \}}{\#\{p\leq x: p \in \mc{P}\}}\r).
\end{equation}
Let $\sigma'=(\tau, \sigma)$ be an element of $\Gal(M/\Q)$, where $\tau \in \Gal(L/\Q)$ and $\sigma \in \Gal(\hat{K}/\Q)$. Applying Lemma \ref{non Galois Chebotarev}, the primes $p\in \mc{P}$ correspond to the $\sigma'$ for which $\sigma$ has non-coprime cycle lengths (so these primes have density $T(G)$ as mentioned above). If in addition, the prime $p$ is split in $L$, then we require that $\tau = \textrm{id}$. Therefore, by Lemma \ref{Chebotarev for non-Galois extensions}, asymptotically as $x\ra \infty$, one half of the primes in $\mc{P}$ are also split in $L$. We conclude from (\ref{computing theta}) that (\ref{density 2}) holds with $\theta = 1+ 2(1/2) = 2$. 
\end{proof}

\begin{proof}[Proof of Theorem \ref{HP with cubic factors}]
We begin in the same manner as in the proof of Theorem \ref{main HP result}, by appealing to Proposition \ref{algebraic reduction} to reduce to a sieve problem. The binary form $F(x,y)$ has degree at most one higher than the degree of $f$. Therefore, we deduce from Theorem \ref{main sieve result for cubic factors} that the Hasse principle holds for (\ref{object of study}) provided that $(\deg f +1)\sum_{j=1}^n \f{1}{q_j-1}\leq 0.32380$.
\end{proof}

\begin{example}\label{irvings example}
We consider the example $K=\Q(2^{1/q})$ discussed by Irving \cite{irving2017cubic}, where $q$ is prime. Since $K$ has prime degree, it satisfies the Hasse norm principle by work of Bartels \cite{bartels1981arithmetik}. We now compute $G=\Gal(\hat{K}/\Q)$. The minimum polynomial of $K/\Q$ is $x^q-2$, which has roots $\{\beta, \beta\omega, \ldots, \beta \omega^{q-1}\}$, where $\omega$ is a primitive $q$th root of unity and  $\beta$ is the real root of $x^q-2$. We identify these roots with $\{0,\ldots, q-1\}$ in the obvious way. An element $\sigma \in G$ is determined by the image of $0$ and $1$, since $\beta, \beta\omega$ multiplicatively generate all the other roots. Therefore, $\sigma$ takes the form $\sigma_{a,b}:x \mapsto ax+b$ for some $a \in \F_q^{\times}, b \in \F_q$. Conversely, the maps $\sigma_{1,b}$ correspond to the $q$ different embeddings $K$ into $\hat{K}$, and the maps $\sigma_{a,0}$ for $a \in \F_q^{\times}$ are elements of $\Gal(\hat{K}/K) \leq G$. Combining these, we see that $\sigma_{a,b}\in G$ for any $a \in \F_q^{\times}, b \in \F_q$. We conclude that $G \isom \op{AGL}(1,q)$, the group of affine linear transformations on $\F_q$. 

When $a=1$ and $b\neq 0$, $\sigma_{a,b}$ is a $q$-cycle. For all other choices of $a,b$, the equation $ax+b=x$ has a solution $x \in \F_q$, and so $\sigma_{a,b}$ has a fixed point. Therefore, $T(G) = (q-1)/\#G = 1/q$. From this, we see that Irving's choice of $\mc{P} = \{p \notin S: p \equiv 1 \Mod{q}\}$ is not quite optimal, because it has density $\alpha = 1/(q-1)$, whereas the set of primes we actually need to sift out has density $T(G) = 1/q$. In fact, we can see directly that even when $p\equiv 1 \Mod{q}$, there is sometimes a solution to $x^q -2 \equiv 0 \Mod{p}$ (e.g. $q=3$, $p = 31$, $x= 4$). However, it can be checked that even with this smaller sieve dimension, we are still not able to handle the cases $q=5$ or $q=3$ when $f$ is an irreducible cubic. 

We remark that we can replace the number $2$ by any positive integer $r$ such that $x^q-r$ is irreducible in the above example. (A necessary and sufficient condition for irreducibility of $x^q-r$ is given in \cite[Theorem 8.16]{karpilovsky1989topics}.) We can still take a sifting set $\mc{P}$ contained in $\{p\notin S: p\equiv 1\Mod{q}\}$ and with density $1/q$, and the Galois group is still $\op{AGL}(1,q)$, and so generalising to $x^q-r$ does not affect the analysis. 
\end{example}

\subsection{Proof of Corollary \ref{a tale of two quadratics}}
We now consider the case when $[K:\Q]=n$ and $G = S_n$, with a view to proving Corollary \ref{a tale of two quadratics}. Such number fields automatically satisfy the Hasse principle by the work of Kunyavski\u{\i} and Voskresenski\u{\i} \cite{HNPforSnextensions}. To ease notation, we shall write $T(n)$ in place of $T(S_n)$. In the following lemma, we find an estimate for $T(n)$.
\begin{lemma}\label{neat lemma about symmetric group} 
For all $n\geq 1$, we have
\begin{align}
 T(n) &= 1-\sum_{k\mid n}\f{\mu(k)\Gamma((n+1)/k)}{\Gamma(1/k)\Gamma(n/k+1)},\label{exact formula for T(n)}\\
T(n) &< \f{2}{\sqrt{\pi}}n^{1/r-1}\omega(n),\label{upper bound for T(n)}
    \end{align}
where $r$ is the smallest prime factor of $n$ and $\omega(n)$ is the number of prime factors of $n$.
\end{lemma}

\begin{proof}
Define 
$$T_k(n) = \frac{1}{n!}\#\{\sigma \in S_n: \textrm{ the cycle lengths of }\sigma \textrm{ are all divisible by $k$}\}.$$
By M\"{o}bius inversion, we have $T(n) = 1-\sum_{k\mid n}\mu(k)T_k(n)$. We now find an explicit formula for $T_k(n)$. For $j\geq 1$, let $a_{jk}$ denote the number of cycles of length $jk$ in $\sigma$. The cycle lengths of $\sigma$ are all a multiple of $k$ if and only if $\sum_{j=1}^{n/k} jka_{jk} = n$. We apply the well-known formula for the number of permutations of $S_n$ with a given cycle shape to obtain

\begin{align}T_k(n) &= \frac{1}{n!}\sum_{\substack{a_{k}, a_{2k}, \ldots, a_{n}\\ \sum_{j=1}^{n/k}jka_{jk}=n}}\frac{n!}{\prod_{j=1}^{n/k}(jk)^{a_{jk}}a_{jk}!}\nonumber\\
&=\sum_{\substack{b_1,\ldots, b_{n/k}\\ \sum_{j=1}^{n/k}jb_j=n/k}}\frac{1}{\prod_{j=1}^{n/k}(jk)^{b_j}b_j!}\nonumber\\
&=\sum_{i=1}^{n/k}k^{-i}\sum_{\substack{b_1, \ldots, b_{n/k}\\ \sum_{j=1}^{n/k}jb_j=n/k\\ \sum_{j=1}^{n/k}b_j = i}}\frac{1}{\prod_{j=1}^{n/k}j^{b_j}b_j!}\nonumber\\
&= \frac{1}{m!}\sum_{i = 1}^{m}k^{-i}c(m,i)\label{expression for Tkn},
\end{align}
where $m = n/k$ and $c(m,i)$ is the number of $\sigma' \in S_m$ with exactly $i$ cycles. The quantity $c(m,i)$ is called the \textit{Stirling number of the first kind}. In order to evaluate (\ref{expression for Tkn}), we follow the argument from \cite[Example II.12]{flajolet2009analytic}. We define a bivariate generating function 
\begin{equation}\label{bivariate GF}
P(w,z) \colonequals \sum_{i=0}^{\infty}w^i\sum_{m=0}^{\infty}\frac{z^m}{m!}c(m,i).
\end{equation}
By \cite[Proposition II.4]{flajolet2009analytic}, we have 
$$ \sum_{m=0}^{\infty}\frac{z^m}{m!}c(m,i) = \f{1}{i!}\l(\log\l(\f{1}{1-z}\r)\r)^i.$$
Therefore, 
$$P(w,z) = \sum_{i=0}^{\infty}\f{w^i}{i!}\l(\log\l(\f{1}{1-z}\r)\r)^i = \exp\l(w\log\l(\f{1}{1-z}\r)\r)=(1-z)^{-w}.$$
Applying the Binomial theorem, we find that the $z^m$ coefficient of $P(w,z)$ is equal to $w(w+1)\cdots (w+m-1)/m!$. On the other hand, if we substitute $w= 1/k$, the $z^m$ coefficient of (\ref{bivariate GF}) is precisely (\ref{expression for Tkn}). We conclude that 
\begin{align*}
    \frac{1}{m!}\sum_{i=1}^m k^{-i}c(m,i) &= (1/k)(1+1/k)\cdots (m-1+1/k)/m!\\
    &=\frac{\Gamma(m+1/k)}{\Gamma(1/k)\Gamma(m+1)},
\end{align*}
which completes the proof of (\ref{exact formula for T(n)}).

We now establish the upper bound in (\ref{upper bound for T(n)}). A basic bound on the gamma function is that for any real $s\in (0,1)$ and any positive real number $x$, we have 
$$ x^{1-s} < \frac{\Gamma(x+1)}{\Gamma(x+s)} < (1+x)^{1-s}.$$
Applying this with $x=m$ and $s=1/k$, we have 
$$\frac{\Gamma(m + 1/k)}{\Gamma(m+1)}< m^{1/k - 1}.$$
Moreover, we have 
$$\frac{1}{\Gamma(1/k)} = \frac{1}{\Gamma(1+1/k)}\frac{\Gamma(1+1/k)}{\Gamma(1/k)}\leq \frac{2}{\sqrt{\pi}k},$$
since $\Gamma(1+1/k)$ for integers $k\geq 2$ achieves its minimum at $k=2$, where we have $\Gamma(1+1/k)=\Gamma(3/2) = \sqrt{\pi}/2$. We conclude that 
$$T_k(n)< \frac{2}{\sqrt{\pi}k}(n/k)^{1/k-1}= \frac{2}{\sqrt{\pi}}n^{1/k-1}k^{-1/k}<\frac{2}{\sqrt{\pi}}n^{1/k-1}.$$
Taking a sum over $k=p$ prime, we obtain
$$T(n) < \sum_{p\mid n}T_p(n) \leq \frac{2}{\sqrt{\pi}}n^{1/r-1}\omega(n),$$
as required.
\end{proof}

\begin{proof}[Proof of Corollary \ref{a tale of two quadratics}]
We recall the setting of Corollary \ref{a tale of two quadratics}. We assume that $G=S_n$, and $f$ is a product of two quadratics generating a biquadratic extension $L$ of $\Q$. We apply Theorem \ref{main sieve result} to the binary form $\prod_{i=0}^2 f_i(x,y)$, where $f_0(x,y) = y$ and $f_1(x,y), f_2(x,y)$ are the homogenisations of the two quadratic factors of $f$. We also assume that $L\cap \hat{K} = \Q$, and so by the proof of Theorem \ref{main HP result}, we have $\theta_0= \theta_1 = \theta_2 = 1$, and $\theta = 3$. By maximising the value of $\kappa$ in (\ref{maximising kappa}) directly, we find by numerical computations that the largest value of $\kappa$ we can take here is $0.42214...$. (The slight improvement over $\kappa \leq 0.39000...$ for the general case comes from computing $\alpha\sum_{i=m+1}^k \theta_i = 2\alpha$ in our example, whilst in the proof of Theorem \ref{main HP result}, we applied the trivial bound $\alpha\sum_{i=m+1}^k \theta_i \leq \kappa = 3\alpha$.) Hence the Hasse principle holds for $f(t) = \bN(\bx)\neq 0$ provided that $T(n) \leq 0.42214.../3 = 0.14071...$. 

We use the upper bound (\ref{upper bound for T(n)}) from Lemma \ref{neat lemma about symmetric group} to reduce the $n$ for which $T(n)\geq 0.14071...$ to finitely many cases, and then the exact formula (\ref{exact formula for T(n)}) to find precisely which $n$ satisfy $T(n)\geq 0.14071...$. We find that $T(n)\leq 0.14071...$ unless $n\in\{2,3,\ldots, 10, 12,14,15,16,18,20,22,24,26,28,30,36,42,48\}.$
\end{proof}

\section{Application to the Harpaz--Wittenberg conjecture}\label{section: the Harpaz-Wittenberg conjecture}

In this section, we apply the sieve result from Theorem \ref{main sieve result} to prove Theorem \ref{HW result}. We recall the statement of \cite[Conjecture 9.1]{harpaz2016fibration} (we shall only work with the ground field $\Q$).
\begin{conjecture}[Harpaz, Wittenberg]\label{Harpaz--Wittenberg 9.1}
Let $P_1, \ldots, P_n \in \Q[t]$ be pairwise distinct irreducible monic polynomials. Let $k_i=\Q[t]/(P_i(t))$ be the corresponding number fields. Let $a_i\in k_i$ denote the class of $t$. For each $i\in \{1,\ldots, n\}$, let $L_i/k_i$ be a finite extension, and let $b_i \in k_i^*$. Let $S_0$ be a finite set of places of $\Q$ including the archimedean place, and all finite places above which, for some $i$, either $b_i$ is not a unit or $L_i/k_i$ is ramified. For each $v\in S_0$, fix an element $t_v \in \Q_v$. Suppose that for every $i\in \{1,\ldots, n\}$ and every $v\in S_0$, there exists $x_{i,v} \in (L_i \tensor_{\Q} \Q_v)^*$ such that $b_i(t_v-a_i) = N_{L_i\tensor_{\Q} \Q_v/k_i \tensor_{\Q} \Q_v}(x_{i,v})$ in $k_i \tensor_{\Q} \Q_v$. Then there exists $t_0 \in \Q$ satisfying the following conditions.
\begin{enumerate}
    \item $t_0$ is arbitrarily close to $t_v$ for all $v \in S_0$.
    \item For every $i\in \{1,\ldots, n\}$ and every place $\mf{p}$ of $k_i$ with $\ord_{\mf{p}}(t_0-a_i)>0$, either $\mf{p}$ lies above a place of $S_0$ or the field $L_i$ possesses a place of degree $1$ over $\mf{p}$. 
\end{enumerate}
\end{conjecture}
We remark that below, the $b_i$ and $x_{i,v}$ appearing in Conjecture \ref{Harpaz--Wittenberg 9.1} do not play a role, and so in the cases that Theorem \ref{HW result} applies, it establishes a stronger version of Conjecture \ref{Harpaz--Wittenberg 9.1}, where the assumption on the existence of the elements $x_{i,v}$ is removed. We discuss this further in Section \ref{section:HW hypoth}.

We can reduce Conjecture \ref{Harpaz--Wittenberg 9.1} to a sieve problem as follows. Let $f_i(x,y) = c_iN_{k_i/\Q}(x-a_iy)$, where $c_i \in \Q$ is chosen such that the coefficients of $f_i(x,y)$ are coprime integers. Then $f_i(x,y)$ is an irreducible polynomial in $\Z[x,y]$. 

Below, we suppose that $S$ is a finite set of primes containing all primes in $S_0$ and all primes dividing any of the denominators $c_1, \ldots, c_n$. For $i\in \{1,\ldots, n\}$, we define $\mc{P}_i$ to be the set of primes $p \notin S$, such that for some place $\mf{p}$ of $k_i$ above $p$, $L_i$ does not possess a place of degree $1$ above $\mf{p}$. 

\begin{lemma}\label{HW to sieve} Let $k_1, \ldots, k_n$ and $L_1, \ldots, L_n$ and $S_0$ be as in Conjecture \ref{Harpaz--Wittenberg 9.1}, and let $\mc{P}_i$ and  $f_i(x,y)$ be as defined above. Suppose that there exists a finite set of primes $S\supset S_0\bs\{\infty\}$ such that for any congruence condition $C(x,y)$ on $x,y$ modulo an integer $\Delta$ with only prime factors in $S$, and any real numbers $r, \xi >0$, there exists $x_0,y_0 \in \N$ such that
\begin{enumerate}[label=(\roman*)]
    \item $C(x_0,y_0)$ holds,
    \item $|x_0/y_0 - r| < \xi$,
    \item $f_i(x_0,y_0)$ has no prime factors in $\mc{P}_i$ for all $i\in \{1, \ldots, n\}$.
\end{enumerate}
Then Conjecture \ref{Harpaz--Wittenberg 9.1} holds for this choice of $k_1,\ldots,k_n,L_1, \ldots, L_n$ and $S_0$.
\end{lemma}

\begin{proof}From \cite[Remark 9.3 (iii)]{harpaz2016fibration}, we are free to adjoin to $S_0$ a finite number of places, and so we may assume that $S_0 = S\cup\{\infty\}$. We choose $t_0 =  x_0/y_0$. Then property (1) of Conjecture \ref{Harpaz--Wittenberg 9.1}
 immediately follows from (i) and (ii) by appropriate choices of $C(x,y), r$ and $\xi$. Let $\mf{p}$ be a place of $k_i$ above a prime $p \notin S$, satisfying $\ord_{\mf{p}}(t_0-a_i)>0$. Then 
 $$ f_i(x_0,y_0) = y_0^{\deg f_i} f_i(x_0/y_0,1) = y_0^{\deg f_i}c_i N_{k_i/\Q}(t_0-a_i).$$
Now $\ord_{\mf{p}}(t_0-a_i)>0$ implies that $\ord_p(N_{k_i/\Q}(t_0-a_i))>0$. Since $p \notin S$, we have $\ord_p(y_0c_i) \geq 0$, and so $p\mid f_i(x_0,y_0)$. By (iii), we have $p \notin \mc{P}_i$, and so by construction of $\mc{P}_i$, we deduce that property (2) of Conjecture \ref{Harpaz--Wittenberg 9.1} holds. 
\end{proof}

In view of Lemma \ref{HW to sieve}, we let 
\begin{align*}
{\mcbf{P}} &= (\mc{P}_1, \ldots, \mc{P}_n),\\
\mc{R} &= \{(x_0,y_0) \in [0,N]^2: |x_0/y_0 -r|< \xi\},
\end{align*}
and we aim to show that the sifting function
\begin{equation}\label{sifting vector}
S(\mc{A},\mcbf{P},x) = \#\{(x_0,y_0) \in \mc{R}\cap \Z^2: C(x_0,y_0), \gcd(f_i(x_0,y_0), P_i(x))=1 \,\forall i\}
\end{equation}
is positive for sufficiently large $N$. We do not attempt here to generalise Theorem \ref{main sieve result} to deal with different sifting sets $\mc{P}_i$ for each $i$, but instead define below in (\ref{application of Poitou--Tate}) a set $\mc{P} \supseteq \bigcup_{i=1}^n \mc{P}_i$ and replace each of the conditions $\gcd(f_i(x_0,y_0), P_i(x))=1$ with $\gcd(f_i(x_0,y_0), P(x))=1$. 

\begin{proof}[Proof of Theorem \ref{HW result}] We recall that $L_i$ is the compositum $k_iM_i$, for a number field $M_i$ which is linearly disjoint to $k_i$ over $\Q$. Consequently, $[L_i:k_i] = [M_i:\Q]$. Writing $M_i = \Q(\beta_i)$ using the primitive element theorem, we therefore have that the minimum polynomial of $\beta_i$ over $\Q$ and over $k_i$ coincide. We denote this minimum polynomial by $g_i$.  

Let $\mf{p}$ denote a place of $k_i$. The inertia degrees of the places of $L_i$ above $\mf{p}$ are the degrees of $g_i$ when factored modulo $\mf{p}$. If $g_i$ has a root modulo $p$, then it has a root modulo every $\mf{p}\mid p$, and so $p \notin \mc{P}_i$. Therefore, we have
\begin{equation}\label{acutal choice of P_i}
    \mc{P}_i \subseteq \widetilde{\mc{P}_i}\colonequals \{p \notin S: M_i \textrm{ does not possess a place of degree 1 above }p \}.
\end{equation}
Define $\mc{P} = \bigcup_{i=1}^n \widetilde{\mc{P}}_i$. Clearly, to show that the sifting function $S(\mc{A},\mcbf{P},x)$ from (\ref{sifting vector}) is positive, it suffices to show the sifting function $S(\mc{A}, \mc{P},x)$ (in the notation of (\ref{def of sifting function})) is positive. 

By the Chebotarev density theorem (Lemma \ref{Chebotarev for non-Galois extensions}), the sets $\widetilde{\mc{P}_i}$ have density $\alpha_i = T_i$, where $T_i$ is defined in (\ref{fixed point density}). The sets $\widetilde{\mc{P}_i}$ are examples of \textit{Frobenian sets} as defined by Serre \cite[Section 3.3.1]{SerreNxp}, which roughly means that away from $S$ their membership is determined Artin symbols of some Galois extension (here $\hat{M_i}$). It follows from \cite[Proposition 3.7 b)]{SerreNxp} that the intersection of Frobenian sets is Frobenian. Since $\mc{P}$ is a disjoint union of such intersections, we conclude from Theorem \ref{Chebotarev for Galois extensions} that the set $\mc{P}$ does indeed have a density. We shall bound its density trivially by $\sum_{i=1}^n \alpha_i$. 

We now bound the value of $\theta$ defined in (\ref{density 2}). Here, we have already defined $f_i$ to be a binary form, and so no additional term $\theta_0$ coming from homogenisation is required. We apply the trivial estimate $\nu_i(p) \leq \deg f_i = [k_i:\Q]$ for all $p \notin S$. We conclude that $\theta \leq \sum_{i=1}^n[k_i:\Q]=d$. Combining Lemma \ref{HW to sieve} and Theorem \ref{main sieve result} completes the proof of part (1) of Theorem \ref{HW result}.

We now turn to the cubic case. Assumption \ref{assumtion on special P} holds for our choice $\mc{P}$ and $f$. As in Section \ref{proof of cubic sieve result}, we conclude that $S(\mc{A},\mcbf{P},x)>0$ for sufficiently large $N$, provided that $t\leq 0.32380$, where $t = \deg f \sum_{i=1}^n 1/(q_i-1)$. Rearranging, and recalling $\deg f=d$, we complete the proof of part (2) of Theorem \ref{HW result}. 
\end{proof}

\begin{remark}In contrast to Section \ref{section: Application to HP}, now $\Gal(\hat{M_i}/\Q)=S_n$ is not a case we can handle, because there the proportion of fixed point free elements is $1-1/e$ as $n\ra \infty$ (where $e=2.718...$ is Euler's constant), which is much too large. 
\end{remark}

For a permutation group $G$ acting on  $X=\{1, \ldots, k\}$, we define $h(G)$ to be the proportion of elements of $G$ with no fixed point. The family of of groups $G$ for which $h(G)$ is smallest are the \textit{Frobenius groups}. These are the groups where $G$ has a nontrivial element fixing one point of $X$, but no nontrivial elements fixing more than one point of $X$. We state two known results about Frobenius groups.  

\begin{lemma}[{\cite[Theorem 1]{FrobeniusGalois}}]
Any Frobenius group can be realised as a Galois group over $\Q$. 
\end{lemma}

\begin{lemma}[{\cite[Theorem 3.1]{frobeniusfacts}}]\label{fg facts} Let $G$ be a transitive permutation group on $k$ letters.
\begin{enumerate}
    \item We have $h(G) \geq 1/k$, with equality if and only if $G$ is a Frobenius group of order $k(k-1)$ and $k$ is a prime power.
    \item In all other cases, $h(G) \geq 2/k$.
\end{enumerate}
\end{lemma}

\begin{proof}[Proof of Corollary \ref{HW corollary}]
As computed in Example \ref{irvings example}, we have that $G_i :=\Gal(\hat{M_i}/\Q)$ is isomorphic to the group $\op{AGL}(1,q_i)$ of affine linear transformations on $\F_{q_i}$. This is a Frobenius group of order $q_i(q_i-1)$. By Lemma \ref{fg facts}, we have $T_i = h(G_i) = 1/q_i$. (This also agrees with our computation in Example \ref{irvings example}.) If $[k_i:\Q]\leq 2$ for all $i$, we can therefore apply part (1) of Theorem \ref{HW result} provided that $\sum_{i=1}^n 1/q_i \leq 0.39000/d$. Moreover, for all $i\in \{1,\ldots, n\}$, the sifting sets $\mc{P}_i$ are contained in $\{p\notin S: p \equiv 1 \Mod{q_i}\}$. Indeed, when $p \not\equiv 1 \Mod{q_i}$, the $q$th power map on $\F_p^{\times}$ is a bijection, and so $x^{q_i}-r_i$ has a root modulo $p$.  

the minimum polynomial $x^{q_i}-r_i$ has a root modulo $p$ for all but finitely many $p \not\equiv 1 \Mod{q_i}$, and these finitely many exceptional primes can be included in $S_0$. Therefore, we can apply part (2) of Theorem \ref{HW result} provided that $\sum_{i=1}^n 1/(q_i-1) \leq 0.32380/d$.
\end{proof}

\subsection{The hypothesis on $b_i$}\label{section:HW hypoth}
Given that the quantities $b_i$ play no role in Theorem \ref{HW result}, it is natural to ask under what circumstances we should expect a stronger version of Conjecture \ref{Harpaz--Wittenberg 9.1} to hold, without the hypothesis on the $b_i$. The following proposition demonstrates that the hypothesis remains unchanged after passing to maximal abelian subextensions of each $L_i/k_i$. 

\begin{proposition}\label{Wittenberg's email}
Let $L/k$ be an extension of number fields. Let $L'/k$ be the maximal abelian subextension of $L/k$. Let $S$ be a sufficiently large set of places of $k$, including all archimedean places. For each $v\in S$, fix an element $b_v \in k_v^{*}$. Then the following are equivalent:
\begin{enumerate}
    \item There exists an $S$-unit $b \in k$ such that for all $v\in S$, $b/b_v$ is in the image of the norm map $(L\tensor_k k_v)^* \ra k_v^*$.
    \item There exists an $S$-unit $b \in k$ such that for all $v\in S$, $b/b_v$ is in the image of the norm map $(L'\tensor_k k_v)^* \ra k_v^*$.
\end{enumerate}
\end{proposition}
We prove Proposition \ref{Wittenberg's email} at the end of this section. We now demonstrate how the lack of the hypothesis on $b_i$ in Theorem \ref{HW result} is explained by Proposition \ref{Wittenberg's email}. We shall choose $S$ to consist of all places that lie above a finite set of places $S_0$ of $k$, which corresponds to the set $S_0$ from Conjecture \ref{Harpaz--Wittenberg 9.1}. We recall that due to \cite[Remark 9.3 (iii)]{harpaz2016fibration}, we are free to make $S_0$, and hence $S$, large enough that Proposition \ref{Wittenberg's email} applies. We apply Proposition \ref{Wittenberg's email} with $L=L_i$ and $k=k_i$ for each extension $L_i/k_i$ from Conjecture \ref{Harpaz--Wittenberg 9.1}, and with $b_v = 1/(t_v-a_i)$. Clearly, if $L/k$ contains no nontrivial abelian subextensions (so that $L'=k$), then Condition (2) above is trivially satisfied, and so Proposition \ref{Wittenberg's email} implies that the hypothesis on $b_i$ in Conjecture \ref{Harpaz--Wittenberg 9.1} is vacuous. To complete the argument, it suffices to show that in the setting of Theorem \ref{HW result}, we have $L_i' = k_i$ for all $i$.  In fact, in the following lemma we show that the hypotheses of Theorem \ref{HW result} force the stronger property that the $L_i/k_i$ contain no nontrivial Galois subextensions.

\begin{lemma}
Suppose that $L/k$ is one of the extensions $L_i/k_i$ from Theorem \ref{HW result}, and let $T=T_i$ be as in (\ref{fixed point density}). Suppose that $T[k:\Q]< 1/2$. Then $L/k$ has no nontrivial Galois subextensions. 
\end{lemma}

\begin{proof}
We recall from the proof of Theorem \ref{HW result} that $T\geq \alpha$, where $\alpha$ is the natural density of the set $\mc{P}$ of primes $p\notin S$ such that there is some place $\mf{p}$ of $k$ above $p$ for which $L$ does not possess a place of degree 1 above $\mf{p}$. For $x\geq 1$, we have the trivial bound
\begin{align*}T\pi(x) &\geq \#(\mc{P}_{\leq x})\\
&\geq \f{1}{[k:\Q]}\#\{\mf{p} \subseteq \OOk: N(\mf{p})\leq x, L \textrm{ has no degree one place above }\mf{p}\},
\end{align*}
since there are at most $[k:\Q]$ prime ideals $\mf{p}$ above each $p$, and $N(\mf{p})\geq p$.

Suppose that $N/k$ is a Galois subextension of $L/k$. In order for $L$ to possess a degree one place above $\mf{p}$, so must $N$. However, since $N/k$ is Galois, $N$ possess a place of degree one above $\mf{p}$ if and only in $\mf{p}$ splits completely in $N$, and by the Chebotarev density theorem, this occurs with density $1/[N:k]$. We deduce that 
\begin{align*}T\pi(x)&\geq \f{1}{[k:\Q]}\#\{\mf{p} \subseteq \OOk: N(\mf{p})\leq x, N \textrm{ has no degree one place above }\mf{p}\}\\
&\geq \f{\#\{\mf{p} \subseteq \OOk: N(\mf{p}) \leq x\}}{[k:\Q]}\l(1-\f{1}{[N:k]}\r).
\end{align*}
Taking a limit at $x\ra \infty$ and applying the prime ideal theorem, we conclude that 
$$ T\geq \f{1}{[k:\Q]}\l(1-\f{1}{[N:k]}\r).$$
The assumption $T[k:\Q]<1/2$ therefore implies that $N=k$. 
\end{proof}

Let $T$ denote the norm one torus associated to $L/k$, which is the algebraic group over $k$ defined by the the equation $\bN_{L/k}(\bx) = 1$. We have a short exact sequence 
\begin{equation}\label{norm one torus SES}
1 \ra T \ra R_{L/k}\G_m \xrightarrow{N_{L/k}} \G_m \ra 1,
\end{equation}
where $R_{L/k}$ denotes the Weil restriction. Let $T_{\overline{k}} = T\times_k \overline{k}$ and $\G_{m,\overline{k}}\isom \overline{k}^{\times}$. We define the character group of $T$ to be $\hat{T} = \Hom(T_{\overline{k}},\G_{m,\overline{k}})$, viewed as a Galois module via the natural action of $\Gal(\overline{k}/k)$. 

Let $\OO_{S}$ denote the ring of $S$-integers of $k$, and $\OO_{L, S_L}$ the ring of $S_L$-integers of $L$, where $S_L$ consists of all places of $L$ above a place in $S$. We include in $S$ all places of $k$ which ramify in $L$. Then $\OO_{L,S_L}/\OO_{S}$ is \'{e}tale, and so the equation $\bN_{\OO_{L,S_L}/\OO_{S}}(\bz) = 1$ defines a model $\mc{T}$ of $T$ over $\OO_S$. Similarly to (\ref{norm one torus SES}), we have a short exact sequence 
\begin{equation}\label{GSSES}
1 \ra \mc{T} \ra R_{\OO_{L,S_L}/\OO_S}\G_{m,\OO_{L,S_L}} \xrightarrow{N_{L/k}} \G_{m,\OO_S} \ra 1.
\end{equation}

Let $k_S$ denote the maximal subextension of $\bar{k}/k$ which is unramified at all places not contained in $S$. Below, we shall work with profinite group cohomology of $G_S\colonequals \Gal(k_S/k)$. We note that $L$ is a subextension of $k_S$ since $S$ is assumed to contain all ramified places of $L/k$. We may therefore define $G_{L,S} = \Gal(k_S/L)$. Let $A_S$ denote the integral closure of $\OO_S$ in $k_S$.  The natural action of $\Gal(\overline{k}/k)$ on $\hat{T}$ factors through $G_S$, so $\hat{T}$ can be viewed as a $G_S$-module. The character group $\Hom(\mathcal{T}_{A_S}, \G_{m, A_S})$ is nothing more than $\hat{T}$ as a $G_S$-module, so we shall henceforth denote it by $\hat{T}$. 

\begin{lemma}\label{H1THAT} Let $\phi:H^1(G_S, \Q/\Z) \ra H^1(G_{L,S},\Q/\Z)$ be the restriction map induced by the inclusion $G_{L,S} \hookrightarrow G_S$. Then there is an exact sequence
\begin{equation}\label{GSshapiro}
    0 \ra H^1(G_S,\hat{T}) \ra H^1(G_S,\Q/\Z) \xrightarrow{\phi} H^1(G_{L,S},\Q/\Z).
\end{equation}
\end{lemma}

\begin{proof}
We begin by taking character groups of the short exact sequence (\ref{GSSES}), or in other words, applying the contravariant functor $\Hom(-, \G_{m,A_S})$. we obtain a short exact sequence
\begin{equation}\label{T hat SES}
    0 \ra \Z \ra \Z[L/k] \ra \hat{T} \ra 0,
\end{equation}
where $\Z[L/k]\isom \Z[\OO_{L,S_L}/\OO_{S}]$ denotes the free abelian group generated by the $k$-linear embeddings $L\inj \overline{k}$. We now take group cohomology of (\ref{T hat SES}) to obtain a long exact sequence
\begin{equation}\label{T hat LES}
    \cdots \ra H^1(G_S,\Z[L/k]) \ra H^1(G_S,\hat{T}) \ra H^2(G_S,\Z) \ra H^2(G_S,\Z[L/k]) \ra \cdots.
\end{equation}

For any groups $H\leq G$, any $H$-module $N$ and any integer $i\geq 0$, Shapiro's lemma \cite[Proposition 1.11]{milneCFT} states that $H^i(H,N)\isom H^i(G,\op{Coind}^G_H(N))$, where $\op{Coind}^G_H(N) = \op{Hom}_{\Z[H]}(\Z[G],N)$ denotes the coinduced module. We apply Shapiro's lemma to the first and last terms in the exact sequence (\ref{T hat LES}) by choosing $G=G_S, H=G_{L,S}$ and $N=\Z$. Then
$$ \op{Coind}^G_H(N) = \op{Hom}_{\Z[G_{L,S}]}(\Z[G_S],\Z)\isom \Z[G_S/G_{L,S}] \isom \Z[L/k],$$
and so we conclude that $H^i(G_S,\Z[L/k]) \isom H^i(G_{L,S},\Z)$ for all $i\geq 0$. Moreover, $H^1(G_{L,S},\Z)$ consists of all continuous group homomorphisms $G_{L,S} \ra \Z$. Since $G_{L,S}$ is compact, and $\Z$ is discrete, we have $H^1(G_{L,S},\Z) = 0$. Therefore, we obtain from (\ref{T hat LES}) the exact sequence
\begin{equation}
    0 \ra H^1(G_S,\hat{T}) \ra H^2(G_S,\Z) \ra H^2(G_{L,S},\Z).
\end{equation}
Now, for any subextension $E$ of $k_S/k$, we have $H^2(G_{E,S},\Z) \isom H^1(G_{E,S}, \Q/\Z)$. To see this, we take group cohomology of the exact sequence $0 \ra \Z \ra \Q \ra \Q/\Z \ra 0$ to obtain a long exact sequence
\begin{equation*}
    \cdots \ra H^1(G_{E,S},\Q) \ra H^1(G_{E,S}, \Q/\Z) \ra H^2(G_{E,S},\Z) \ra H^2(G_{E,S},\Q) \ra \cdots,
\end{equation*}
and note that $H^1(G_{E,S},\Q) = H^2(G_{E,S},\Q) = 0$ since $G_{E,S}$ is a profinite group \cite[Proposition 1.6.2 (c)]{neukirch2013cohomology}. Therefore, applying this fact with $E=k$ and $E=L$, we have an exact sequence 
\begin{equation}
    0 \ra H^1(G_S,\hat{T}) \ra H^1(G_S,\Q/\Z) \xrightarrow{\psi} H^1(G_{L,S},\Q/\Z).
\end{equation}

To complete the proof, we need to show that the map $\psi$ we have obtained from the above argument is equal to the map $\phi$ defined in the lemma. We consider the diagram

\[\begin{tikzcd}
	{H^2(G_S, \Z)} & {H^2(G_S, \Z[L/k])} \\
	{H^2(G_S, \Z)} & {H^2(G_{L,S},\Z)} \\
	{H^1(G_S, \Q/\Z)} & {H^1(G_{L,S},\Q/\Z)}
	\arrow[from=1-1, to=1-2]
	\arrow[r,-,double equal sign distance,double, from=1-1, to=2-1]
	\arrow["{\op{sh}}", from=1-2, to=2-2]
	\arrow["{\delta^{-1}}", from=2-1, to=3-1]
	\arrow["{\delta^{-1}}", from=2-2, to=3-2]
	\arrow[from=2-1, to=2-2]
	\arrow["\psi", from=3-1, to=3-2]
\end{tikzcd}\]
Here, $\delta^{-1}$ denotes the inverse of the connecting homomorphism, and $\op{sh}$ denotes the Shapiro map, i.e., the isomorphism from the above application of Shapiro's lemma. By definition, $\psi$ comes from applying these isomorphisms to the map $H^2(G_S, \Z) \ra H^2(G_S,\Z[L/k])$ from (\ref{T hat LES}), which is the map $i_*$ in the notation of \cite[Proposition 1.6.5]{neukirch2013cohomology}. Therefore, by \cite[Proposition 1.6.5]{neukirch2013cohomology}, the middle horizontal arrow is just the restriction map. Finally, restriction maps commute with connecting homomorphisms \cite[Proposition 1.5.2]{neukirch2013cohomology}, and so $\psi$ is the restriction homomorphism induced by the inclusion $G_{L,S} \hookrightarrow G_S$, as required.
\end{proof}

In what follows, we denote by $\op{Cl}_{S_L}(L)$ the $S_L$-ideal class group, i.e., the quotient of the usual class group $\op{Cl}(L)$ by the classes of all prime ideals in $S_L$. Since $\op{Cl}(L)$ is finite, by adjoining finitely many primes to $S$, we may assume that $\op{Cl}_{S_L}(L)=0$. 

\begin{lemma}\label{concrete version of obvious etale fact}
Let $S$ be as above. If $a\in k^*$ lies in the image of the norm map $(k_v\tensor_k L)^* \ra k_v^*$ for all $v\notin S$, then there exists $y\in N_{L/k}(L^*)$ such that $ay \in \OO_S^{\times}$. 
\end{lemma}

\begin{proof}
Fix a place $v\notin S$ at which $a$ is not a unit. Let $w_1\cdots w_r$ be the factorisation of $v$ into prime ideals in $\OO_L$, and let $c_i = [L_{w_i}: k_v]$ be the corresponding inertia degrees. (Since $L/k$ is unramified outside $S$, the ideals $w_1, \ldots, w_r$ are distinct.) We define $c=\gcd(c_1, \ldots, c_r)$. The image of the norm map $(k_v\tensor_k L)^* \ra k_v^*$ consists of the elements of $k_v^*$ whose valuation is divisible by $c$. In particular, we may write $v(a) = \sum_{i=1}^r n_ic_i$ for some integers $n_1, \ldots, n_r$. Since $\op{Cl}_{S_L}(L)=0$, we can find an element $z_v\in L$ such that $w_i(z_v) = n_i$ for all $i\in \{1,\ldots, r\}$, and such that $z_v$ is a unit at all other places not in $S_L$. Let $y_v = N_{L/k}(z_v)$. Then $v(y_v) = \sum_{i=1}^r n_i c_i = v(a)$, and $y_v$ is a unit at all other places outside $S$. The result now follows by taking $y$ to be the product of the elements $y_v^{-1}$ over the (finitely many) places $v\notin S$ at which $a$ is not a unit. 
\end{proof}

\begin{proof}[Proof of Proposition \ref{Wittenberg's email}]
The implication $(1)\implies (2)$ is trivial, so we only need to prove $(2)\implies (1)$. Taking Galois cohomology of (\ref{norm one torus SES}), we obtain a long exact sequence
\begin{equation}\label{norm torus LES}
1 \ra T(k) \ra L^{\times} \xrightarrow{N_{L/k}}  k^{\times} \xrightarrow{\delta_k} H^1(k,T)\ra H^1(k,R_{L/k}\G_m)\ra \cdots.
\end{equation}
Hilbert's Theorem 90 \cite[Theorem 1.3.2]{BrauerGrothendieckGroup} implies that $H^1(k,R_{L/k}\G_m)$ is trivial, and so we can naturally identify $H^1(k,T)$ with $\f{k^{\times}}{N_{L/k}(L^{\times})}$. In a similar way, we can identify $H^1(k_v, T)$ with the quotient of $k_v^*$ by the image of the norm map $(L\tensor_k k_v)^* \ra k_v^*$. We can also take group cohomology of the short exact sequence (\ref{GSSES}) to obtain a long exact sequence 
\begin{equation}\label{GSLES}
    1 \ra \mc{T}(\OO_{S}) \ra \OO_{L,S_L}^{\times} \ra \OO_{S}^{\times} \xrightarrow{\delta_{\OO_S}} H^1(G_S,\mc{T}_{A_S}) \ra H^1(G_S, A_{S_L}^{\times}) \ra \cdots,
\end{equation}
where $A_{S_L} = A_S \tensor_{\OO_S} \OO_{L, S_L}$ is the integral closure of $\OO_{L,S_L}$ in $k_S$. (We remark that the ring $A_S$ is denoted by $\OO_S$ in \cite{neukirch2013cohomology}.) By \cite[Proposition 8.3.11 (ii)]{neukirch2013cohomology}, we have $H^1(G_S, A_{S_L}^{\times}) \isom \op{Cl}_{S_L}(L)$, which we recall is trivial due to our choice of $S$. Therefore, we may identify $H^1(G_S, \mc{T}_{A_S})$ with $\f{\OO_{S}^{\times}}{\OO_{S}^{\times} \cap N_{L/k}(L^{\times})}$. 

Let $v\in S$, and let $\overline{k}_v$ denote an algebraic closure of $k_v$. Fixing an embedding $k_S \hookrightarrow \overline{k_v}$ determines a surjection $\Gal(\overline{k_v}/k_v) \ra G_S$, which induces restriction maps $H^1(G_S, \mc{T}_{A_S}) \ra H^1(k_v, T)$ and $H^1(k, T) \ra H^1(k_v, T)$ called \textit{localisation maps}. Below, we denote these maps by $\op{res}_v$. Under the above identifications, we obtain a commutative diagram 

\begin{equation}\label{H1s}
\begin{tikzcd}
	{H^1(G_S, \mc{T}_{A_S})} & {H^1(k,T)} & {H^1(k_v, T)} \\
	{\frac{\OO_S^{\times}}{\OO_S^{\times}\cap N_{L/k}(L^{\times})}} & {\f{k^{\times}}{N_{L/k}(L^{\times})}} & {\f{k_v^{*}}{\op{im}((L\tensor_k k_v)^* \ra k_v^*)}}
	\arrow["\isom"', from=1-1, to=2-1]
	\arrow[hook, from=1-1, to=1-2]
	\arrow[hook, from=2-1, to=2-2]
	\arrow["\isom"', from=1-2, to=2-2]
	\arrow[from=1-2, to=1-3]
	\arrow["\isom"', from=1-3, to=2-3]
	\arrow[from=2-2, to=2-3]
	\arrow["{\op{res}_v}", curve={height=-18pt}, from=1-1, to=1-3]
\end{tikzcd}
\end{equation}
where $(L \tensor_k k_v)^* \ra k_v^*$ denotes the norm map, and the bottom arrows are induced by the inclusions $\OO_S^{\times} \hookrightarrow k \hookrightarrow k_v^*$. Consequently, we can reformulate Condition (1) from Proposition \ref{Wittenberg's email} as 
\begin{enumerate}
\setcounter{enumi}{2}
    \item The class of $(b_v)_{v \in S}$ in $\prod_{v \in S}H^1(k_v, T)$ belongs to the image of the map $\prod_{v \in S}\op{res}_v$.
\end{enumerate}
We assume that $S$ is also large enough that $\op{Cl}_{S_{L'}}(L')=0$, so that the above argument gives a similar reformulation of (2), but with the torus $T'$ associated to $L'/k$ in place of $T$. 

Let $(-)^{\vee}=\Hom(-,\Q/\Z)$. Poitou--Tate duality \cite[Theorem 4.20 b)]{neukirch2013cohomology} gives an exact sequence 

\begin{equation}\label{Milne's Poitou-Tate}
    H^1(k,T) \xrightarrow{\op{res}} \sideset{}{'}\prod H^1(k_v, T) \xrightarrow{\xi} H^1(k, \hat{T})^{\vee}.
\end{equation}
Here, the restricted product is over all places $v$ of $k$, with the added assumption that the specified element of $H^1(k_v,T)$ comes from $H^1(\OO_v, T)$ for all but finitely many $v$, and the first map is induced by the residue maps $\op{res}_v$ for each place $v$ of $k$. 

Let $\iota: \prod_{v\in S}H^1(k_v, T) \hookrightarrow \sideset{}{'}\prod H^1(k_v, T)$ be defined by adding trivial classes at the places $v\notin S$, and let $\xi_S = \xi \circ \iota$. We now explain how to deduce from (\ref{Milne's Poitou-Tate}) an exact sequence of the form
\begin{equation}\label{application of Poitou--Tate} H^1(G_S, \mc{T}_{A_S}) \xrightarrow{\prod_{v\in S}\op{res}_v} \prod_{v \in S}H^1(k_v,T) \xrightarrow{\xi_S} H^1(G_S,\hat{T})^{\vee}.
\end{equation}

Suppose that $b= (b_v)_{v\in S} \in \ker(\xi_S)$. Then $\iota(b) \in \ker \xi$. By the exactness of (\ref{Milne's Poitou-Tate}), this implies $\iota(b) =\op{res}(a)$ for some $a\in H^1(k,T)$. Moreover, we know that $\op{res}_v(a)$ is trivial at all places $v\notin S$. Recalling the identifications from (\ref{H1s}), this means that $a$ lies in the image of the norm map $(k_v \tensor_L k)^* \ra k_v^*$ for all $v\notin S$. By Lemma \ref{concrete version of obvious etale fact}, there is a representative of the class $a\in H^1(k,T)$ which is contained in  $\OO_S^{\times}$. Therefore, $a$ in fact lies in $H^1(G_S, \mc{T}_{A_S})$. It follows that $\prod_{v\in S}\op{res}_v(a) = b$, establishing the exactness of (\ref{application of Poitou--Tate}). Hence, condition (1) from the proposition is equivalent to the condition $(b_v)_{v\in S} \in \ker \xi_S$. 

The map $\xi_S$ is induced by the localisation maps $\prod_{v \in S}\op{res}_v:H^1(G_S, \hat{T}) \ra \prod_{v\in S}H^1(k_v,\hat{T})$ and the pairing
$$\prod_{v\in S}\l(H^1(k_v,T)\times H^1(k_v,\hat{T})\r)\xrightarrow{\cup}\prod_{v\in S}H^2(k_v, \G_m)\xrightarrow{\sum_{v \in S}\op{inv}_v} \Q/\Z,$$
where $\cup$ denotes the cup product applied at each place $v\in S$ and $\op{inv}_v$ denotes the local invariant map, as defined in \cite[pp. 156]{neukirch2013cohomology}.

Consider the diagram 

\[\begin{tikzcd}
	{\prod_{v\in S}H^1(k_v, T)} & {\prod_{v\in S}H^1(k_v, T')} \\
	{\prod_{v\in S}\Hom\left(H^1(k_v, \hat{T}), H^2(k_v, \G_m)\right)} & {\prod_{v\in S}\Hom\left(H^1(k_v, \hat{T'}), H^2(k_v, \G_m)\right)} \\
	{\prod_{v\in S}H^1(k_v,\hat{T})^{\vee}} & {\prod_{v\in S}H^1(k_v,\hat{T'})^{\vee}} \\
	{H^1(G_S, \hat{T})^{\vee}} & {H^1(G_S, \hat{T'})^{\vee}}
	\arrow[from=1-1, to=1-2]
	\arrow["\cup", from=1-1, to=2-1]
	\arrow[from=2-1, to=2-2]
	\arrow["{\sum_{v\in S}\op{inv}_v}", from=2-1, to=3-1]
	\arrow[from=3-1, to=3-2]
	\arrow["\cup"', from=1-2, to=2-2]
	\arrow["{\sum_{v\in S}\op{inv}_v}"', from=2-2, to=3-2]
	\arrow["{\prod_{v\in S}\op{res}_v}", from=3-1, to=4-1]
	\arrow["\theta", from=4-1, to=4-2]
	\arrow["{\prod_{v\in S}\op{res}_v}"', from=3-2, to=4-2]
\end{tikzcd}\]
where all the horizontal arrows are induced by the map $T\xrightarrow{N_{L/L'}} T'$. (The norm map $N_{L/L'}: R_{L/k}\G_m \ra R_{L'/k}\G_m$ restricts to a map $T \ra T'$, because $N_{L/k}(z) = N_{L/L'}(N_{L'/k}(z))$ for any $z \in L$.) This diagram commutes, thanks to the functorality properties of the localisation maps $\op{res}_v$ \cite[Proposition 1.5.2]{neukirch2013cohomology}, and the projection formula for the cup product \cite[Proposition 1.4.2]{neukirch2013cohomology}. The composition of the vertical arrows are the map $\xi_S$ and the corresponding map $\xi'_{S}$ for $T'$. Therefore, we obtain a commutative diagram

\[\begin{tikzcd}
	{\prod_{v\in S}H^1(k_v, T)} & {\prod_{v\in S}H^1(k_v, T')} \\
	{H^1(G_S,\hat{T})^{\vee}} & {H^1(G_S,\hat{T'})^{\vee}}
	\arrow["\xi_S", from=1-1, to=2-1]
	\arrow[from=1-1, to=1-2]
	\arrow["{\xi'_S}", from=1-2, to=2-2]
	\arrow["\theta", from=2-1, to=2-2]
\end{tikzcd}\]

Applying Lemma \ref{H1THAT}, we have $H^1(G_S,\hat{T}) \isom \ker \varphi$, where $\varphi$ is as defined in the exact sequence (\ref{GSshapiro}). As discussed in the paragraphs preceding Lemma \ref{group theory result}, $\phi$ sends cyclic subextensions $M/k$ of $k_S$ (with a given choice of generator for $\Gal(M/k)$) to their compositum $LM/L$, and so elements of $\ker \phi$ correspond to cyclic subextensions of $L/k$. However, since $L'/k$ is the maximal abelian subextension of $L/k$, it contains all cyclic subextensions of $L/k$, and so the map $H^1(G_S, \hat{T'}) \ra H^1(G_S, \hat{T})$ is surjective. Since $\Q/\Z$ is a divisible group, it is an injective object in the category of abelian groups, and so the contravariant functor $\Hom(-,\Q/\Z)$ is exact. We deduce that $\theta$ is an injection. Therefore, if $\xi'_S((b_v)_{v\in S})=0$, then $\xi_S((b_v)_{v\in S})=0$, so $(b_v)_{v\in S} \in \ker \xi'_S \implies (b_v)_{v \in S} \in \ker \xi_S$. Recalling (\ref{application of Poitou--Tate}), this means that $(2) \implies (1)$. 
\end{proof}

\begin{appendix}

\section{The Brauer group for the equation $f(t) = \bN(\bx) \neq 0$}\label{section:Brauer group calculation}
This appendix will be concerned with the Brauer group of a smooth projective model $X$ of the equation $f(t) = \bN(\bx) \neq 0$. In particular, we prove that in the setting of Corollary \ref{a tale of two quadratics}, we have $\Br(X) = \Br(\Q)$ whenever $n\geq 3$. We are grateful to $\CT$ for providing the arguments presented in this appendix. 

\subsection{Main results}

\begin{theorem}\label{BRAUER GROUP CALCULATION}
Let $k$ be a field of characteristic zero. Let $K/k$ be an extension of degree $n$, and let $L/k$ be the Galois closure. Suppose that $\Gal(L/k) = S_n$. Let $f(t) \in k[t]$ be a squarefree polynomial. Let $Y/k$ be the affine variety given by the equation $f(t)= \bN(x_1, \ldots, x_n) \neq 0$, and $Y \ra \A^1_k$ its projection onto $t$. Let $\pi:X\ra \PP_k^1$ be a smooth projective birational model of $Y \ra \A^1_k$. Suppose that $L$ and the number field generated by $f$ are linearly disjoint over $k$. Then $\Br(k) = \Br(X)$. 
\end{theorem}

\begin{theorem}\label{surjective Br}
Let $k$ be a field of characteristic zero. Let $K/k$ be a finite extension of degree $n\geq 3$, such that the Galois closure $L/k$ satisfies $\Gal(L/k) = S_n$. Let $c\in k^{\times}$, and let $Z$ be a smooth projective model of $\bN(x_1, \ldots, x_n) = c$. Then $\Br(k) \ra \Br(Z)$ is surjective. 
\end{theorem}

\subsection{Proof of Theorem \ref{surjective Br}}

\begin{proof}
The key ideas of the proof are discussed in detail by Bayer-Fluckiger and Parimala \cite{parimalatorsortori}, and so here we just give a sketch. We would like to show that $\Br(Z)/\op{Im}(\Br(k))=0$. We begin by reducing to the case $c=1$. Suppose that $T$ is the norm one torus given by $\bN(x_1, \ldots, x_n) = 1$, and let $T^c$ denote a smooth compactification of $T$. Let $k_s$ denote the separable closure of $k$, and let $\overline{Z} = Z\times_{k}k_s$, and $\overline{T} = T \times _k k_s$. By \cite[Lemme 2.1]{colliotharariskoro2003valeurs}, we have an isomorphism $H^1(k,\Pic \overline{Z}) \isom H^1(k, \Pic \overline{T}^c)$. Combining this with \cite[Theorem 2.4]{parimalatorsortori}, we have 
\begin{equation}
    \Br(Z)/\op{Im}(\Br(k)) \hookrightarrow H^1(k, \Pic \overline{Z}) \isom H^1(k, \Pic \overline{T^c}) \isom \Br(T^{c})/\Br(k),
\end{equation}
and hence it suffices to show that $\Br(T^c)/ \Br(k) = 0$. 

Let $G = \Gal(L/k)$. The character group $\hat{T} = \Hom(T_{\overline{k}}, \G_{m, \overline{k}})$ can be viewed as a $G$-lattice. By \cite[Proposition 9.5 (ii)]{colliot1987flasquetori}, we have an isomorphism
$$\Br(T^c)/\Br(k) \isom \Sh^2_{\op{cycl}}(G, \hat{T}),$$
where
$$\Sh^2_{\op{cycl}}(G, \hat{T}) = \ker\l[H^2(G,\hat{T}) \ra \prod_{g\in G} H^2(\langle g \rangle, \hat{T}) \r].$$
Let $M/k$ be a finite extension with $M$ linearly disjoint from $L$, and let $L'=LM, K'=KM, k'=kM$. Then the extension $K'/k'$ has degree $n$ and Galois closure $L'$, with $\Gal(L'/k') = G$. Moreover, by a construction of Fr\"{o}lich \cite{frohlich_1962}, we may choose $M$ in such a way that $L'/k'$ is unramified.

Using \cite[Proposition 4.1]{parimalatorsortori}, we have $\Sh^2_{\op{cycl}}(G, \hat{T}) \isom \Sh^2(k', \hat{T})$, where $\hat{T}$ is regarded as a $\Gal(k'_s/k')$-module via the surjection $\Gal(k'_s/k') \ra G$. In turn, this is isomorphic to $\Sh^1(k', T)^{\vee}$ by Poitou--Tate duality \cite[Corollary 4.5]{parimalatorsortori}. To summarise, we have isomorphisms
\begin{equation}
    \Br(T^c)/\Br(k) \isom \Sh^2_{\op{cycl}}(G, \hat{T}) \isom \Sh^2(k', \hat{T}) \isom \Sh^1(k', T)^{\vee},
\end{equation}
and so it suffices to show that $\Sh^1(k', T)=0$. However, $\Sh^1(k', T)$ is isomorhpic to the knot group $\kappa(K'/k') = \f{k'^{\times}\cap N_{K'/k'}(\A_{K'}^{\times})}{N_{K'/k'}(k'^{\times})}$. Due to the assumption $\Gal(L'/k') = G = S_n$, we may apply the result of Kunyavski\u{\i} and Voskresenski\u{\i} \cite{HNPforSnextensions} to deduce that the Hasse norm principle holds for the extension $K'/k'$, and hence $\kappa(K'/k')=0$.
\end{proof}

\subsection{Proof of Theorem \ref{BRAUER GROUP CALCULATION}}
Before commencing with the proof, we require one more fact about the smooth projective model $X$ from the statement of Theorem \ref{BRAUER GROUP CALCULATION}. 
\begin{lemma}\label{rationality of base change}
In the notation of Theorem \ref{BRAUER GROUP CALCULATION}, the base change $X_K = X \times_k K$ is a $K$-rational variety. 
\end{lemma}

\begin{proof}
Since $X$ is a smooth projective model of $Y$, it suffices to show that $Y_K$ is $K$-rational. Let $K = k[x]/(p(x))$, where $p(x)$ is an irreducible polynomial over $k$. Let $a$ denote the class of $x$. Over $K$, the polynomial $p(x)$ factorises as $p(x) =\prod_{i=0}^r q_i(x)$, where $q_0(x), \ldots, q_r(x) \in K[x]$ are distinct and irreducible, and $q_0(x)  = x-a$. Let $K_i = K[x]/(q_i(x))$. We shall construct a birational map of the form
\begin{equation}\label{the birat map}
\begin{split}
\phi: Y_K &\ra \A^1_K \times \prod_{i=1}^r R_{K_i/K}\A^1\\
(t,x_0, \ldots, x_{n-1}) &\mapsto (t,z_1, \ldots, z_r),
\end{split}
\end{equation}
where $R_{K_i/K}$ denotes the Weil restriction. Since $R_{K_i/K}\A^1 \isom \A^{\deg q_i}_K$, the right hand side of (\ref{the birat map}) is isomorphic to $\A^n_K$, which is a Zariski open subset of $\PP^n_K$. Therefore, $\phi$ induces a birational map $Y_K \birat \PP^n_K$, as desired. 

Let $\overline{k}$ denote an algebraic closure of $k$. We denote by $\Emb_k(K, \overline{k})$ the embeddings $K\hookrightarrow \overline{k}$ fixing $k$, or in other words, the conjugates of $K/k$ in $\overline{k}$. Over $\overline{k}$, the polynomials $p(x), q_0(x), \ldots, q_r(x)$ split as 
\begin{equation}\label{factorising the factors}
    p(x) = \prod_{\sigma \in \Emb_k(K,\overline{k})}(x-\sigma(a)), \qquad q_i(x) = \prod_{\ss{\sigma \in \Emb_k(K,\overline{k})\\ q_i(\sigma(a))=0}}(x-\sigma(a)).
\end{equation}

For each $i$, we fix an isomorphism $K_i \isom K(\sigma_i(a))$ for some $\sigma_i \in \Emb_k(K,\overline{k})$ satisfying $q_i(\sigma_i(a))=0$, and view $\sigma_i(a)$ as the class of $x$ in $K_i/K$. (The particular choice of representative $\sigma_i$ does not matter.) Since $q_i(x)$ is the minimum polynomial of $\sigma_i(a)$ over $K$, it splits over $\overline{k}$ as the product of the conjugates of $\sigma_i(a)$, and so 
$$ q_i(x) = \prod_{\sigma' \in \Emb_K(K_i, \overline{k})}(x - \sigma'\sigma_i(a)).$$

For $i\in \{0,\ldots, r\}$, we define $z_i \in R_{K_i/K}\A^1$ as
\begin{equation}\label{definition of zi}
z_i = x_0+\sigma_i(a)x_1 + \cdots + \sigma_i(a)^{n-1}x_{n-1}.
\end{equation}
the polynomial $\sum_{j=0}^{\deg q_i -1}z_i^{(j)}x^j$ representing $z_i$ is the reduction of $x_0 + x_1x + \cdots + x_{n-1}x^{n-1}$ modulo $q_i$. Consequently, by the Chinese remainder theorem, $z_0, \ldots, z_r \in \prod_{i=0}^r R_{K_i/K}\A^1$ uniquely determine $x_0, \ldots, x_{n-1} \in \A^1_K$.   

For any number field extension $E/M$, and any $y \in E$, we have
$$N_{E/M}(y)  = \prod_{\sigma \in \Emb_M(E,\overline{M})} \sigma (y).$$
Therefore,
\begin{align*}
    N_{K/k}(y) = \prod_{\sigma \in \Emb_k(K,\overline{k})}\sigma(y) = \prod_{i=0}^r \prod_{\sigma' \in \Emb_{K}(K_i, \overline{k})}\sigma'\sigma_i(y) = \prod_{i=0}^r N_{K_i/K}(\sigma_i(y)),
\end{align*}
and so 
\begin{equation}\label{rewriting the norm form}
\bN(x_0, \ldots, x_{n-1}) = \prod_{i=0}^r N_{K_i/K}(z_i).
\end{equation}
We deduce that the equations (\ref{definition of zi}) define an isomorphism from $Y_K$ to the variety $V\subseteq \A^1_K \times \prod_{i=0}^rR_{K_i/K}\A^1$ given by $z_0\prod_{i=1}^r N_{K_i/K}(z_i) = f(t)\neq 0$. 
For $t,z_1, \ldots, z_r$ satisfying the Zariski open condition $\prod_{i=1}^r N_{K_i/K}(z_i) \neq 0$, we have $z_0 = f(t)/\prod_{i=1}^n N_{K_i/K}(z_i)$. Therefore, the projection of $V$ onto $\A^1_K \times \prod_{i=1}^r R_{K_i/K}\A^1$ is birational. We conclude that the map $\phi$ from (\ref{the birat map}) is birational.
\end{proof}

We now commence with the proof of Theorem \ref{BRAUER GROUP CALCULATION}. Let $k$ be a field of characteristic zero. For a smooth irreducible variety $X/k$ with function field $\kappa(x)$, we recall that $\Br(X)$ consists of all elements of $\Br(\kappa(X))$ which are unramified everywhere on $X$. Constant classes are unramified, and so we have $\Br(k) \subseteq \Br(X) \subseteq \Br(\kappa(X))$. By the purity theorem \cite[Theorem 3.7.1]{BrauerGrothendieckGroup}, the ramification locus of $\mc{A}\in \Br(\kappa(X))$ is pure of codimension one. Consequently, to check $\mc{A} \in \Br(X)$, it suffices to check it is  unramified at all codimension one points of $X$. 

Let $C$ be a codimension one point. We recall from \cite[Section 1.4.3]{BrauerGrothendieckGroup} the \textit{residue map} 
$$\del_C: \Br(\kappa (X)) \ra H^1(\kappa(C), \Q/\Z)$$
is such that $\mc{A}$ is unramified at $C$ if and only if $\del_C(\mc{A})$ is trivial. 

In our setting, codimension one points of $X$ come in two types:

\begin{enumerate}
    \item Irreducible components of fibres $X_c = \pi^{-1}(c)$ above codimension one points $c$ of $\PP_k^1$,
    \item  Codimension one points on the generic fibre $X_{\eta}$ of $\pi:X\ra \PP_k^1$. 
\end{enumerate}

We recall that $\kappa(X) = \kappa(X_\eta)$. The codimension one points of $X_{\eta}$ are a subset of the codimension one points of $X$, and so we have an inclusion $\Br(X) \inj \Br(X_{\eta})$. Since  $X_{\eta}$ is a smooth projective model of $\bN_{K(t)/k(t)}(x_1, \ldots, x_n) = f(t)$ over $k(t)$, it follows from Theorem \ref{surjective Br}, applied to the extension $K(t)/k(t)$ and with $Z=X_{\eta}$, that $\Br(k(t)) \ra \Br(X_{\eta})$ is surjective. Putting everything together, we obtain a commutative diagram
\[\begin{tikzcd}
	{\Br(X)} && {\Br(X_{\eta})} & {\Br(\kappa(X_{\eta})) = \Br(\kappa(X))} \\
	{\Br(k)} && {\Br(k(t))}
	\arrow[hook, from=2-1, to=2-3]
	\arrow[hook', from=2-1, to=1-1]
	\arrow[hook, from=1-1, to=1-3]
	\arrow[two heads, from=2-3, to=1-3]
	\arrow[hook, from=1-3, to=1-4]
\end{tikzcd}\]

Let $\alpha \in \Br(X)$. By the above diagram, we can find $\beta \in \Br(k(t))$ whose image in $\Br(X_{\eta})$ is equal to the image of $\alpha$ in $\Br(X_{\eta})$. We want to show that $\beta$ is the image of an element of $\Br(k)$, because then it follows from commutativity of the diagram that $\alpha$ is the image of an element of $\Br(k)$. 

For any $n\geq 1$ and any field $k$, we have $\Br(\PP_k^n) = \Br(k)$ \cite[Theorem 6.1.3]{BrauerGrothendieckGroup}. In particular, we have $\Br(k) = \Br(\PP_k^1)$. Also, $k(t) = \kappa(\PP_k^1)$, so $\Br(k(t)) = \Br(\kappa(\PP_k^1))$. Therefore, as discussed above, to prove that $\beta$ is in the image of $\Br(k)$, it suffices to show $\beta$ is unramified at every codimension one point of $\PP^1_k$. This is formalised by the \textit{Faddeev exact sequence} \cite[Theorem 1.5.2]{BrauerGrothendieckGroup}, which is the exact sequence
\begin{equation}
    0 \ra \Br(k) \inj \Br(k(t)) \ra \bigoplus_{Q\in (\PP^1_k)^{(1)}} H^1(k_Q , \Q/\Z) \surj H^1(k, \Q/\Z) \ra 0, 
\end{equation}
where $(\PP^1_k)^{(1)}$ denotes the codimension one points of $\PP_k^1$ and the third map is the direct sum of the residue maps $\del_Q$. In other words, to show that $\beta \in \Br(k(t))$ is actually in $\Br(k)$, it suffices to show that $\del_Q(\beta) = 0$ for all $Q \in (\PP_k^1)^{(1)}$. We have $\del_Q(\beta) = 0$ unless $Q$ is an irreducible factor of $f(t)$ by \cite[Proposition 11.1.5]{BrauerGrothendieckGroup}, so we suppose from now on that $Q$ is an irreducible factor of $f(t)$. 

By Lemma \ref{rationality of base change}, the base change $X_K = X \times_{k} K$ is birational to $\PP_K^n$. Since the Brauer group is a birational invariant on smooth projective varieties \cite[Corollary 6.2.11]{BrauerGrothendieckGroup}, it follows that $\Br(X_K) = \Br(\PP^n_K) = \Br(K)$. Therefore, we obtain the following commutative diagram:

\adjustbox{scale=0.85,right}{

\begin{tikzcd}
	{} & {\Br(X_K)} && {\Br(X_{K,\eta})} \\
	&& {\Br(X)} && {\Br(X_{\eta})} \\
	& {\Br(K)} && {\Br(K(t))} & {\bigoplus\limits_{Q \in (\PP^1_{k})^{(1)}}H^1(Kk_Q, \Q/\Z)} \\
	&& {\Br(k)} && {\Br(k(t))} & {\bigoplus\limits_{Q \in (\PP^1_{k})^{(1)}}H^1(k_Q, \Q/\Z)}
	\arrow["{\psi_K}"'{pos=0.3}, two heads, from=3-4, to=1-4]
	\arrow[two heads, from=3-4, to=3-5]
	\arrow["\isom"'{pos=0.3}, from=3-2, to=1-2]
	\arrow[hook, from=3-2, to=3-4]
	\arrow[from=4-3, to=3-2]
	\arrow["\iota", hook, from=2-3, to=2-5, crossing over]
	\arrow[hook, from=4-3, to=2-3, crossing over]
	\arrow[hook, from=4-3, to=4-5]
	\arrow["\psi"'{pos=0.3}, two heads, from=4-5, to=2-5]
	\arrow[two heads, from=4-5, to=4-6]
	\arrow["{\iota_K}", hook, from=1-2, to=1-4]
	\arrow[from=4-5, to=3-4]
	\arrow[from=2-3, to=1-2]
	\arrow[from=2-5, to=1-4]
	\arrow["\varphi"', from=4-6, to=3-5]
\end{tikzcd}
}
The map $\phi$ is a direct sum over the restriction maps $\phi_Q:H^1(k_Q, \Q/\Z)\ra H^1(Kk_Q, \Q/\Z)$ for each $Q \in (\PP^1_k)^{(1)}$. 
\begin{lemma}\label{diagram chase}
We have $\del_Q(\beta) \in \ker \varphi_Q$, where $\varphi_Q$ is as defined above.
\end{lemma}

\begin{proof}
Let $\beta_K$ denote the image of $\beta$ in $\Br(K(t))$, and $\del_{K,Q}$ the residue map at $Q$ on $\Br(K(t))$. We want to show that $\del_{K,Q}(\beta_K) = 0$. By exactness of the Faddeev exact sequence over $K$, for this it suffices to show $\beta_K$ is in the image $\Br(K) \ra \Br(K(t))$. We know that $\psi(\beta) = \iota(\alpha)$. Applying a base change to $K$, we see that $\psi_K(\beta_K) = \iota_K(\alpha_K)$, where $\alpha_K, \beta_K$ are the images of $\alpha, \beta$ under base change. Since $\Br(K) \isom \Br(X_K)$, we have that $\alpha_K$ is in the image of $\Br(K) \ra \Br(X_K)$, and hence by commutativity of the diagram, $\beta_K$ is in the image of $\Br(K) \ra \Br(K(t))$, as required. 
\end{proof}

Let $M$ be a number field, and let $G_M = \Gal(\overline{M}/M)$. For a finite Galois extension $M'/M$, we consider $\Gal(M'/M)$ as a topological space with the discrete toplology. We can then put the profinite topology on $G_M$, which is defined as the inverse limit
$$ G_M = \varprojlim_{M'/M \textrm{ Galois}} \Gal(M'/M). $$

We recall that $H^1(M,\Q/\Z) = \Hom_{\textrm{cont}}(G_M, \Q/\Z)$, the continuous group homomorphisms $G_M \ra \Q/\Z$ \cite[pp.16]{BrauerGrothendieckGroup}. Suppose that $\theta \in \Hom_{\textrm{cont}}(G_M, \Q/\Z)$. Then $\ker \theta$ is an open subgroup of $G_M$. Since $G_M$ is a profinite group, this implies that $\ker \theta$ has finite index in $G_M$. By the fundamental theorem of Galois theory, $\op{im}\theta \isom G_M/\ker \theta \isom \Gal(M'/M)$, for some finite Galois extension $M'/M$. Moreover, $\op{im}\theta$ is a finite subgroup of $\Q/\Z$. All finite subgroups of $\Q/\Z$ are cyclic groups of the form $\f{1}{n}\Z/\Z$ for some positive integer $n$. Consequently, $\ker \theta = \Gal(M'/M)$ for a cyclic extension $M'/M$. To summarise, we have the identification
$$ H^1(M, \Q/\Z) = \{M'/M \textrm{ cyclic, with a given map }\gamma:\Gal(M'/M) \inj \Q/\Z\}.$$

We now describe $\phi_Q: H^1(k_Q, \Q/\Z) \ra H^1(Kk_Q, \Q/\Z)$ explicitly. Using the above identification, we view an element $\theta \in H^1(k_Q, \Q/\Z)$ as a pair $(M'/k_Q, \gamma)$. The map $\phi_Q$ is given by taking the compositum with $K$. More precisely, it sends the above pair to $(KM'/Kk_Q, \gamma)$, where now $\gamma$ is viewed as a map $\Gal(KM'/KK_Q) \hookrightarrow \Q/\Z$ via the natural identification of $\Gal(KM'/Kk_Q)$ as a subgroup of $\Gal(M'/k_Q)$. Therefore, $(M'/k_Q, \gamma) \in \ker \phi_Q$ if and only if $M'/k_Q$ is a cyclic subextension of $Kk_Q/k_Q$. 

Due to the assumption that $k_Q$ and $L$ are linearly disjoint over $k$, we have $\Gal(Kk_Q/k_Q) \isom \Gal(K/k) \isom S_n$. We now complete the proof of Theorem \ref{BRAUER GROUP CALCULATION} with the following elementary group theory fact.

\begin{lemma}\label{group theory result}
Suppose that $K/k$ is a finite extension of degree $n\geq 3$, and the Galois group of the Galois closure $\Gal(L/k)$ is isomorphic to $S_n$. Then there are no nontrivial cyclic extensions $M/k$ with $M\subseteq K$.
\end{lemma}

\begin{proof}
By the fundamental theorem of Galois theory, if $M/k$ is a subextension of $K/k$, then $\Gal(L/K) \leq \Gal(L/M) \leq \Gal(L/k) = S_n$. However, $\Gal(L/K) \isom S_{n-1}$. (More explicitly, if $K=k(\alpha_1)$ and $\alpha_1, \ldots, \alpha_n$ are the roots of the minimum polynomial of $\alpha_1$ over $k$, then $\Gal(L/K)$ consists of all permutations of $\{\alpha_1, \ldots, \alpha_n\}$ which fix $\alpha_1$.) However, $S_{n-1}$ is a maximal subgroup of $S_n$, and so $M=k$ or $M=K$. Since $n\geq 3$, the extension $K/k$ is not cyclic. Therefore, $M=k$. 
\end{proof}

In conclusion, the map $\phi_Q$ is injective by Lemma \ref{group theory result}, and $\del_Q(\beta) \in \ker \phi_Q$ by Lemma \ref{diagram chase}, and hence $\del_Q(\beta)=0$. This means that all the residue maps of $\beta$ are trivial, so $\beta$ is in the image of $\Br(k) \ra \Br(k(t))$. Hence $\alpha$ is in the image of $\Br(k) \ra \Br(X)$, and so $\Br(X) = \Br(k)$. This completes the proof of Theorem \ref{BRAUER GROUP CALCULATION}.

\end{appendix}

\bibliographystyle{plain}
\bibliography{references.bib}

\end{document}